\documentclass[aip,amsmath,amssymb,cha
 ,
 reprint,%
]{revtex4-1}

\usepackage[colorlinks=true, linkcolor=blue, citecolor=blue, urlcolor=blue]{hyperref}

\usepackage{graphicx}
\usepackage{dcolumn}
\usepackage{bm}
\usepackage[utf8]{inputenc}
\usepackage[T1]{fontenc}
\usepackage{mathptmx}
\usepackage{etoolbox}

\usepackage[left]{lineno}

\newcommand{\R}{\mathbb{R}}
\newcommand{\mbX}{\mathbf{X}}
\newcommand{\mbY}{\mathbf{Y}}

\newcommand{\mbf}{\mathbf{f}}

\newcommand{\mbg}{\mathbf{g}}

\newcommand{\mbx}{\mathbf{x}}
\newcommand{\mby}{\mathbf{y}}
\newcommand{\mbW}{\mathbf{W}}
\newcommand{\mbh}{\mathbf{h}}

\newcommand{\mbz}{\mathbf{z}}
\newcommand{\mbu}{\mathbf{u}}
\newcommand{\given}{\:|\:}

\usepackage{amsthm}
\usepackage{graphicx}
\usepackage{listings}
\usepackage[breakable]{tcolorbox}
\usepackage{holtpolt}
\usepackage{mathrsfs}
\usepackage{xcolor}
\usepackage{bbm}
\usepackage[bb=boondox]{mathalfa}

\newtheorem{theorem}{Theorem}
\newtheorem{corollary}{Corollary}[theorem]
\newtheorem{lemma}[theorem]{Lemma}
\theoremstyle{definition}
\newtheorem{definition}{Definition}[section] 

\newtheorem{assumption}{Assumption}[section]

\makeatletter
\def\@email#1#2{%
 \endgroup
 \patchcmd{\titleblock@produce}
  {\frontmatter@RRAPformat}
  {\frontmatter@RRAPformat{\produce@RRAP{*#1\href{mailto:#2}{#2}}}\frontmatter@RRAPformat}
  {}{}
}%
\makeatother
\begin{document}

\preprint{AIP/123-QED}

\title[Arnold tongues for stochastic oscillators]{$Q$-functions, synchronization, and Arnold tongues for coupled stochastic oscillators}
\author{Max Kreider}
\affiliation{Department of
Mathematics, Applied Mathematics, and Statistics, Case Western Reserve University,
Cleveland, OH 44106 USA.}
\email{mbk62@case.edu}

\author{Benjamin Lindner}%
\affiliation{ 
Bernstein Center for Computational Neuroscience Berlin, Berlin
10115, Germany.}%
\affiliation{Department of Physics, Humboldt Universität zu Berlin, Berlin
D-12489, Germany.}

\author{Peter J. Thomas}
\affiliation{Department of
Mathematics, Applied Mathematics, and Statistics, Case Western Reserve University,
Cleveland, OH 44106 USA.}%
\affiliation{Department of Mathematics, Applied Mathematics, and Statistics, Department of Biology, Department of Dataand Computer Science, Department of Electrical, Control and Systems Engineering, Case Western Reserve University,Cleveland, OH 44106 USA.}

\date{\today}

\begin{abstract}
Phase reduction is an effective theoretical and numerical tool for studying synchronization of coupled deterministic oscillators.
\textit{Stochastic} oscillators require new definitions of asymptotic phase.
The $Q$-function, i.e.~the slowest decaying complex mode of the stochastic Koopman operator (SKO), was proposed as a means of phase reduction for stochastic oscillators.  
In this paper, we show that the $Q$-function approach also leads to a novel definition of ``synchronization" for coupled stochastic oscillators.   
A system of coupled oscillators in the synchronous regime may be viewed as a single (higher-dimensional) oscillator.
Therefore, we investigate the relation between the $Q$-functions of the uncoupled oscillators and the higher-dimensional $Q$-function for the coupled system.
We propose a definition of synchronization between coupled stochastic oscillators in terms of the  eigenvalue spectrum of  Kolmogorov's backward operator (the generator of the Markov process, or the SKO) of the higher dimensional coupled system. 
We observe a novel type of bifurcation reflecting (i) the relationship between the leading eigenvalues of the SKO for the coupled system and (ii) qualitative changes in the cross-spectral density of the coupled oscillators.
Using our proposed definition, we observe synchronization domains for  symmetrically-coupled stochastic oscillators that are analogous to Arnold tongues for coupled deterministic oscillators.
\end{abstract}

\maketitle

\begin{quotation}
Synchronization and phase reduction have precise meanings for deterministic limit-cycle oscillators, but \textit{stochastic} oscillators require new definitions.
The slowest decaying eigenmode of the stochastic Koopman operator (SKO), the $Q$-function, is a nonlinear coordinate transformation that provides a universal description for stochastic oscillators.
The leading SKO eigenvalues for \emph{coupled}  oscillators undergo a bifurcation, reflecting a qualitative change in system dynamics.
Here, we present a new definition of ``synchronization'' for stochastic oscillators in terms of this SKO eigenvalue bifurcation.
\end{quotation}

\section{\label{section:introduction} Introduction}

Synchronized oscillations are ubiquitous in natural and engineered systems \cite{Pikovsky_Rosenblum_Kurths_2001}.
Examples include: synchronized neural dynamics underlying perception, cognition, and memory \cite{ermentrout2010mathematical,izhikevich2007dynamical,stiefel2016neurons}, chemical kinetics \cite{kiss2002emerging, taylor2009dynamical},  firefly populations
\cite{sarfati2021self,strogatz2018nonlinear}, locomotion and rhythmic movement \cite{liu2009coupled, santos2017biped, wang2017cpg}, circadian rhythms \cite{aton2005come, gonze2005spontaneous}, cardiac rhythms \cite{ mirollo1990synchronization, dos2004rhythm}, EEG data \cite{burke2004stochastic, herrmann200511, monto2008very}.

In the deterministic case, uncoupled limit-cycle oscillators operate independently, each sustaining a stable limit-cycle solution with an arbitrary phase relationship relative to the others. 
Recall that a stable limit-cycle is a closed, isolated periodic orbit which is the limit set for trajectories in a subset of the domain called the basin of attraction.
Consequently, when considered as a whole, the uncoupled system does not possess a unique limit-cycle solution; the phase ambiguity among the oscillators implies that the high-dimensional periodic orbit is not isolated. 
One way to describe synchronization is as a state in which the joint system behaves as a single high-dimensional oscillator. 
In this synchronized state, the system exhibits a unique attracting limit-cycle solution, where all oscillators have identical periods and maintain a fixed, unambiguous phase relationship.

Phase reduction greatly simplifies the analysis of limit-cycle oscillators by describing their dynamics with a single phase (timing) variable that evolves at a constant rate \cite{brown2004phase, ermentrout2010mathematical, izhikevich2007dynamical, kuramoto1984chemical, nakao2016phase} ({isostable (amplitude) coordinates in directions transverse to the limit cycle, along with phase and isostable response curves, allow for accurate yet simple dynamical descriptions of both weakly and strongly coupled oscillators \cite{castejon2013phase, guillamon2009computational, wilson2016isostable, wilson2018greater, wilson2019phase}.
However, in the present manuscript we will not consider so-called ``augmented phase reduction" or isostable coordinates}).  
Phase is defined away from the limit cycle via isochrons, Poincar\'{e} sections foliating the basin of attraction with the property that the asymptotic timing behavior of trajectories originating on the same isochron are identical \cite{winfree1980geometry}, or equivalently that the time it takes for a trajectory to perform one full oscillation and return to the same isochron is constant \cite{guckenheimer1975isochrons}.
Phase descriptions are especially suited to study synchronization.
Solutions of the equations governing the evolution of the phase difference of the oscillators can accurately reproduce the nonlinear bifurcation associated with the onset of synchronization \cite{izhikevich2007dynamical, park2021high}, allowing one to predict the existence and stability of synchronized solutions, and to identify Arnold tongues \cite{arnold2012geometrical, guckenheimer2013nonlinear, izhikevich2007dynamical, schilder2007computing}, i.e., regions of mode-locking as a function of frequency difference  and coupling strength.

To illustrate these ideas in the deterministic setting, consider a system of coupled oscillators which has already been reduced to a phase description: Kuramoto oscillators of the form
\begin{equation}\label{eq:intro_kuramoto}
    \begin{split}
        x' &= \omega + \tau + \kappa \sin(y-x)
        \\
        y' &= \omega + \kappa \sin(x-y)
    \end{split}
\end{equation}
where $x$ and $y$ are phase-like variables evolving on the torus: $x,y\in [0,2\pi)$.
The oscillators become synchronized when their phase difference is asymptotically constant.
Therefore, we consider the evolution of the phase difference, $\psi = y-x$,
\begin{equation}
    \psi' = \tau -2\kappa\sin(\psi) 
\end{equation}
where $\tau$ is the frequency difference of the oscillators.
It is straightforward to show that when $\kappa > {|\tau|}/{2}$, the oscillators synchronize with a constant phase difference.
Note that the synchronization boundary, $\kappa={|\tau|}/{2}$, is specified in terms of two parameters: coupling strength $\kappa$ and frequency difference (detuning) $\tau$.
Figure \ref{fig:intro} displays the synchronization region or Arnold tongue for equation \eqref{eq:intro_kuramoto}, along with sample trajectories in phase space.

\begin{figure*}[ht]
    \centering
    \includegraphics[scale=.5]{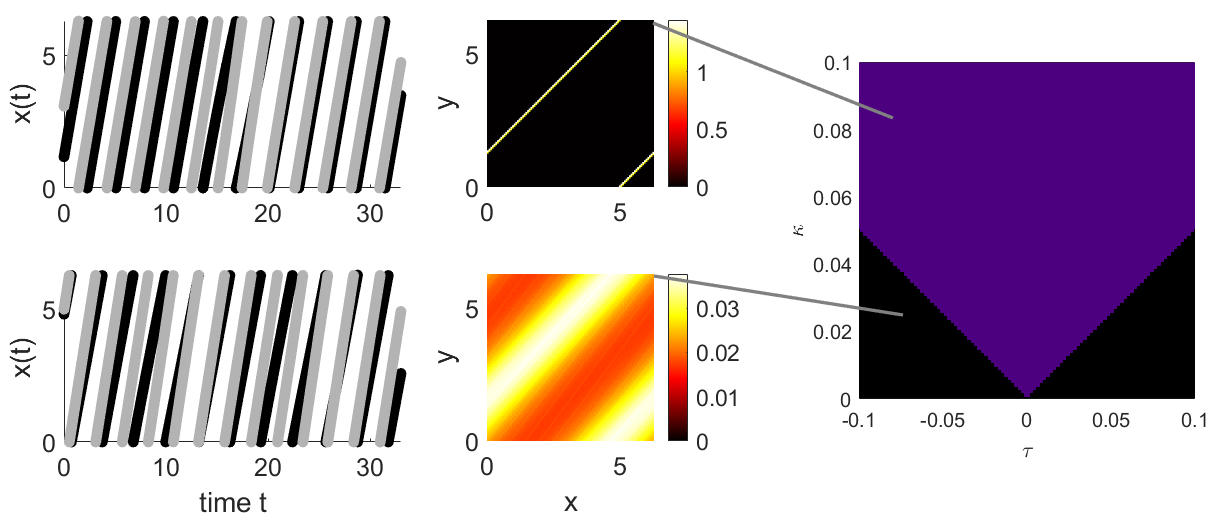}
    \caption{Synchronization behavior for the deterministic Kuramoto system \eqref{eq:intro_kuramoto} with $\omega=2$ and $\tau = 0.5$.
    \textbf{Bottom row:} time series with $\kappa=0.09$ (left) and a histogram of bin counts as a probability distribution (right).
    The system is not synchronized and trajectories (starting from $t=10$) fill the entire torus.
    \textbf{Top row:} time series with $\kappa=0.26$ (left) and a histogram of bins counts (right).
    The system is asymptotically synchronized and trajectories (starting from $t=10$) do not fill the entire torus.
    \textbf{Right column:} Arnold tongue diagram depicting the parameter regime for which the system synchronizes (purple) and for which it does not synchronize (black).}
    \label{fig:intro}
\end{figure*}

The picture changes considerably in the presence of noise.
Noise precludes perfect periodicity; 
in contrast to the deterministic case, stochastic oscillators do not admit periodic limit-cycle solutions.
In turn, deterministic phase is no longer well-defined, and isochrons with ideal asymptotic and return-time properties do not exist for noisy systems.
Consequently, it is natural to ask after the properties of a given system of coupled stochastic oscillators: are there certain characteristics which guarantee that the system behaves ``like a single stochastic oscillator", and which lead to a universal phase-reduction description  analogous to the deterministic case? 

Many efforts have been made to study synchronization  phenomena for stochastic oscillators \cite{amro2015phase, callenbach2002oscillatory,deng2016measuring,ermentrout2009noisy,freund2000analytic,han1999interacting,medvedev2010synchronization,nandi2007effective,neiman1994synchronizationlike,zakharova2011analysing, zakharova2013coherence,zaks2003noise,zhang2008interacting}.
Previous approaches to quantifying synchronization have considered similarity in the power spectra of the oscillations \cite{han1999interacting}, phase diffusion coefficients for the individual oscillators \cite{amro2015phase, callenbach2002oscillatory}, intraclass correlation \cite{deng2016measuring}, degree of coherence \cite{medvedev2010synchronization}, average frequency difference \cite{nandi2007effective}, and effective phase synchronization \cite{callenbach2002oscillatory,han1999interacting}.
These approaches produce continuously graded measures in the sense that, for stronger coupling, the oscillators may be more coherent, or have more similar power spectra, or have smaller mean frequency difference.
However, in general, one does not expect to obtain identical statistical descriptors for the coupled oscillators, even in the limit of strong coupling.
Ideally, one would search for a universal phase description of stochastic oscillators in analogy to deterministic phase reduction.
 
The search for a rigorous notion of phase for stochastic oscillators remained an  open question until two distinct formulations were introduced.
Schwabedal and Pikovsky introduced the mean-return-time phase, which defines isochrons as Poincar\'{e} sections with a \textit{mean} return time property, i.e., the \textit{average} time for a trajectory to perform one full oscillation and return to the same isochron is a constant \cite{cao2020partial,schwabedal2013phase}.
On the other hand, Lindner and Thomas introduced the notion of stochastic asymptotic phase as the complex argument of the $Q$-function, the slowest decaying mode of the Kolmogorov backward operator (also called the stochastic Koopman operator, SKO) \cite{thomas2014asymptotic}.
In the deterministic case, isochrons formulated in terms of the return-time property or the asymptotic behavior of trajectories are equivalent;
perhaps surprisingly, the stochastic generalizations are not equivalent (see
\cite{perez2022quantitative} for a precise relationship).
In the present paper, we will argue that the $Q$-function approach provides a natural framework for formulating the problem of synchronization of coupled stochastic oscillators.

It has been shown that the $Q$-function gives a universal description of stochastic oscillators, greatly simplifying their dynamics and statistical properties \cite{perez2021isostables,perez2023universal}.
The $Q$-function is a nonlinear change of coordinates;
nonlinear stochastic oscillators behave linearly \textit{in the mean} upon being transformed into $Q$-function coordinates, 
just as dynamical phase reduction gives a simplified dynamical description of a deterministic oscillator \cite{perez2020global}.
The relationship between the $Q$-function and deterministic phase reduction was elucidated in recent works, which highlight both the $Q$-function and the deterministic phase function as an eigenfunction of the SKO and the classical Koopman operator, respectively \cite{kato2021asymptotic, perez2023universal}.
Consequently, the $Q$-function approach provides a notion of phase reduction for noisy oscillators which is consistent with phase reduction in the deterministic case.
Indeed, the $Q$-function approach introduces the ``robustly oscillatory" criteria as the defining properties of a stochastic oscillator, i.e., systems which are robustly oscillatory behave ``like stochastic oscillators" and only these systems admit a well-defined $Q$-phase reduction \cite{thomas2014asymptotic}.
In this work, we leverage a $Q$-function approach to study systems of symmetrically coupled stochastic oscillators.
For clarity of exposition, we briefly sketch here the main results of our work; mathematically precise formulations and derivations may be found in section \S \ref{section:results}.

We consider systems of symmetrically coupled oscillators of the form
\begin{equation}\label{eq:intro_coupled_system}
\begin{split}
    d\mathbf{X} &= [\mathbf{f}(\mathbf{X}) + \tau\mathbf{u}_1(\mbX) + \kappa\mathbf{h}(\mathbf{X},\mathbf{Y})]dt + \mathbf{g}_1(\mathbf{X})d\mathbf{W}_1(t)
    \\
    d\mathbf{Y} &= [\mathbf{f}(\mathbf{Y}) + \tau\mathbf{u}_2(\mbY) +\kappa\mathbf{h}(\mathbf{Y},\mathbf{X} )]dt + \mathbf{g}_2(\mathbf{Y})d\mathbf{W}_2(t)
\end{split}
\end{equation}
with related eigenvalue problem
\begin{equation}
\begin{split}
    (\mathcal{L}_x^\dagger + \mathcal{L}_y^\dagger + \tau \mathcal{L}^\dagger_\tau + \kappa \mathcal{L}_{\kappa}^\dagger) Q^*(\mbx,\mby) &= \lambda Q^*(\mbx,\mby)
    \\
    (\mathcal{L}_x + \mathcal{L}_y + \tau \mathcal{L}_\tau + \kappa \mathcal{L}_{\kappa}) P(\mbx,\mby) &= \lambda P(\mbx,\mby)
\end{split} 
\end{equation}
where $\mathcal{L}_x^\dagger$ is the SKO for the individual oscillator in $\mbx$ coordinates (respectively, $\mathcal{L}^\dagger_y$ for the individual oscillator in $\mby$ coordinates), $\mathcal{L}^\dagger_{\kappa}$ is the perturbation to the SKO of the joint system due to  coupling with  strength $\kappa$, and $\mathcal{L}^\dagger_{\tau}$ is the perturbation to the SKO of the joint system due to detuning, regulated by the parameter $\tau$.
That is, $|\tau|>0$ represents a perturbation to the underlying dynamics that leads to a difference in the oscillators' frequencies of oscillation.
The operators $\mathcal{L}_x$, $\mathcal{L}_y$, $\mathcal{L}_{\kappa}$, and $\mathcal{L}_\tau$ are forward (Fokker-Planck) operators, which are the formal adjoints of their respective SKO counterparts (see \S\ref{subsection:Q_phase_reduction} for a more systematic description of these operators).

In the case when \eqref{eq:intro_coupled_system} consists of identical oscillators ($\tau=0$, $\mbg_1=\mbg_2$), when $\kappa=0$ the SKO has repeated eigenvalues with linearly independent eigenfunctions.
Nonzero coupling $\kappa\neq 0$ generically perturbs these eigenvalues and eigenfunctions.
In the case of non-identical oscillators, when $\kappa=0$ the SKO has distinct eigenvalues with linearly independent eigenfunctions.
However, in cases where the non-identical oscillators have similar dynamics, for sufficiently large $\kappa$ we observe numerically that the coupling pulls the eigenvalues together to form a repeated eigenvalue pair at a finite coupling strength.
Further increasing the coupling strength causes the eigenvalues to split.
In bifurcation theory, eigenvalue splitting is often associated with qualitative changes in system behavior, motivating the characterization of conditions under which splitting occurs and its qualitative implications for the coupled system.

To that end, we define a splitting discriminant, $\mathcal{D}$, in terms of the perturbation parameters $\kappa$ and $\tau$, the perturbations to the SKO, $\mathcal{L}^\dagger_\tau$ and $\mathcal{L}^\dagger_\kappa$, and the unperturbed $Q$-functions and corresponding forward eigenfunctions of \eqref{eq:intro_coupled_system}, $Q_{1x}^*(\mbx,\mby)$, $Q_{1y}^*(\mbx,\mby)$, $P_{1x}(\mbx,\mby)$, and $P_{1y}(\mbx,\mby)$,
\begin{equation}
    \mathcal{D}(\kappa,\tau) = \sqrt{\text{trace}(\mathcal{M})^2 - 4\det(\mathcal{M})}
\end{equation}
where the matrix $\mathcal{M}$ is defined in the statement of our Theorem \ref{theorem:Q function perturbation}.

We prove that when $\mathcal{D}(\kappa,\tau)\neq 0$, the repeated $Q$-function eigenvalues of the identical system split to leading order.
We demonstrate that this eigenvalue splitting corresponds to qualitative changes in the coupled system dynamics: a transition from two distinct oscillating units to a single ``robustly oscillatory'' stochastic oscillator.
We also prove that the eigenvalue splitting entails a qualitative change in the power spectra and cross-spectral density of the coupled system in $Q$-function coordinates.
In the special case that $\mathcal{D}(\kappa,\tau)\neq 0$ is purely real, we prove that the power spectra of the individual oscillators peak at identical frequencies, i.e., the oscillators have identical frequencies in $Q$-function coordinates.
In contrast, commonly used statistical descriptors are not expected to match exactly.
These results motivate the following novel definition of synchronization for stochastic oscillators:

\begin{tcolorbox}
\begin{definition}\label{definition} 
    Symmetrically coupled oscillators of the form \eqref{eq:intro_coupled_system} exhibit \textit{$Q$-synchronization} if the leading nontrivial complex conjugate eigenvalues of the SKO undergo a qualitative change, i.e., a repeated pair of eigenvalues is created or destroyed, as system parameters are varied. 
\end{definition}
\end{tcolorbox}

The synchronization behavior of deterministic oscillators is often studied by constructing Arnold tongues, or synchronization regimes, as a function of frequency difference and coupling strength. 
In the case of symmetrically coupled stochastic oscillators of the form \eqref{eq:intro_coupled_system}, we prove that under certain conditions, there exists a relationship $\kappa^* = \mathfrak{K}(\tau^*)$ for which the splitting discriminant vanishes, $\mathcal{D}(\mathfrak{K}(\tau^*),\tau^*)=0$, and the repeated SKO eigenvalues do not split.
In such cases, we derive a closed-form expression for $\mathfrak{K}$ for \textit{any} coupled system of the form \eqref{eq:intro_coupled_system}.
We call parameter values which satisfy $\kappa^*=\mathfrak{K}(\tau^*)$ KT points, and show that KT points give a leading order synchronization boundary in the $\tau$-$\kappa$ plane that is analogous to a $1:1$ mode-locking Arnold tongue for deterministic oscillators. 

The manuscript is organized as follows. 
Section \S \ref{section:background} provides a detailed summary of existing theory for $Q$-phase reduction.
Section \S\ref{section:results} establishes the main contributions of our work.
We first derive expressions governing the leading order change of eigenfunction-eigenvalue pairs of the SKO and Fokker-Planck operator.
We apply this theory to derive general conditions under which a qualitative change in the SKO spectrum occurs, which we relate to synchronization.
We prove that the onset of $Q$-synchronization is accompanied by qualitative changes in the power spectra and cross-spectral density of the oscillators.
We also derive a leading order expression for $Q$-synchronization regimes which are analogous to $1:1$ mode-locking Arnold tongues for deterministic oscillators.
In section \S\ref{section:examples}, we consider several examples, including 
a 2D system comprising two coupled  noisy Kuramoto oscillators,
a 4D system comprising two coupled planar Ornstein-Uhlenbeck processes, 
and a nine-dimensional system of coupled 3D discrete-state oscillators.
Section \S\ref{section:discussion} provides a discussion and directions for future work.


\section{Background and Notation}\label{section:background}

\subsection{Phase-reduction of deterministic oscillators}\label{subsection:phase-reduction}

For the sake of completeness, we briefly review phase reduction for deterministic limit-cycle oscillators.
Consider a smooth ($\mathcal{C}^2$) dynamical system
\begin{equation}\label{eq:oscillator}
    \mathbf{x}' = \mathbf{f}(\mathbf{x})
\end{equation}
with $\mathbf{x}\in \R^n$, assumed to admit a $T$-periodic stable limit-cycle solution, $\gamma(t)=\gamma(t+T)$.
Here and elsewhere, $'$ denotes a derivative with respect to time unless otherwise specified.
One may define a phase function $\Theta(\mathbf{x}(t))=\phi(t)\in [0,2\pi)$ on the limit cycle so that the phase evolves at a constant rate
\begin{equation}\label{eq:phase_equation}
    \phi' = \frac{2\pi}{T} = \omega
\end{equation}
with $\omega$ the frequency of oscillation.
Note the phase function, $\Theta(\mbx)$, is uniquely defined, up to an arbitrary phase offset. 
The phase function may be extended to the basin of attraction of the limit cycle via isochrons, level curves of the phase function \cite{guckenheimer1975isochrons, winfree1980geometry}.
Identification of phase on versus near the periodic orbit exploits  the asymptotic convergence of trajectories to the stable limit cycle.
Any initial conditions $\mathbf{x}_0$ and $\mathbf{y}_0$ with corresponding trajectories $\mathbf{x}(t)$ and $\mathbf{y(t)}$ that satisfy
\begin{equation}
    \lim_{t\to \infty} \|\mathbf{x}(t)-\mathbf{y}(t)\| = 0
\end{equation}
for any norm $\|\cdot \|$ have the same phase. 
Note that \eqref{eq:phase_equation} holds everywhere in the basin of attraction, provided that 
\begin{equation}\label{eq:normalization}
    \mathbf{f}(\gamma(t)) \cdot \nabla \Theta(\gamma(t)) = \omega
\end{equation}
The function $\nabla \Theta(\gamma(t))$ is called the (infinitesimal) phase response curve and is often denoted as $\nabla \Theta \equiv \mathbf{Z}$.


\subsection{\textit{Q}-phase reduction}\label{subsection:Q_phase_reduction}

Deterministic asymptotic phase is not well-defined for stochastic systems.
Unlike the deterministic case, at long times an ensemble of independent realizations will tend toward a unique stationary distribution, so the long-time asymptotic behavior of the trajectories can no longer distinguish between initial conditions.
Accordingly, Thomas and Lindner extended the notion of asymptotic phase to stochastic systems by considering the evolution of the density describing an ensemble of trajectories, rather than a single trajectory \cite{thomas2014asymptotic}.
Here, we briefly summarize previously established theory for the $Q$-phase reduction of stochastic oscillators \cite{perez2021isostables, perez2023universal, thomas2014asymptotic}.

\subsubsection{Operator definitions}

Consider a Langevin equation\footnote{The $Q$ function approach \cite{perez2023universal,thomas2014asymptotic} applies also to processes with discrete jumps and hybrid discrete/continuous processes.
See \S\ref{subsubsection:9D discrete-state model} for an application involving a discrete-state system} of the form
\begin{equation}\label{eq:stochastic_oscillator}
    d\mathbf{X} = \mathbf{f}(\mathbf{X}) dt + \mathbf{g}(\mathbf{X})d\mathbf{W}(t)
\end{equation}
with $\mathbf{X}(t)$ a stochastic process taking values in $\R^n$, $\mbf:\R^n\to\R^n$ a smooth ($\mathcal{C}^2$) vector field, $\mathbf{g}$ an $n\times m$ matrix with smooth ($\mathcal{C}^2$) entries such that $\mathbf{g}\mathbf{g}^T$ is nonsingular, and $d\mathbf{W}(t)$ an $m \times 1$ vector of independent increments of a Wiener process. 
We interpret \eqref{eq:stochastic_oscillator} in the sense of It\^{o}.
The conditional density 
\begin{equation}\label{eq:condition density}
    \rho(\mathbf{y},t \given \mathbf{x},s) = \frac{1}{|d\mby|}\text{Pr}\{\mbX(t) \in [\mby, \mby + d\mby] \given \mbX(s) = \mbx \}
\end{equation}
for times $t>s$ satisfies the forward equation
\begin{equation}\label{eq:forward_equation}
    \frac{\partial}{\partial t} \rho = \mathcal{L}[\rho] = -\sum_{i} \partial_{x_i}[\mathbf{f} \rho] + \frac{1}{2}\sum_{ij} \partial_{x_i x_j}\{ [\mathbf{g}\mathbf{g}^T]_{ij} \rho \}
\end{equation}
and the backward equation
\begin{equation}\label{eq:backward_equation}
    -\frac{\partial}{\partial s} \rho = \mathcal{L}^\dagger[\rho] = \sum_{i} \mathbf{f}_i\partial_{x_i}[ \rho] + \frac{1}{2}\sum_{ij} [\mathbf{g}\mathbf{g}^T]_{ij}\partial_{x_i x_j}\{  \rho \}
\end{equation}
with $1\leq i,j \leq n$. 
We discuss boundary conditions \cite{gardiner1985handbook}   for the forward and backward equations \eqref{eq:forward_equation} and \eqref{eq:backward_equation} on a case-by-case basis for  specific examples later in this manuscript.

The backward operator $\mathcal{L}^\dagger$, also known as the infinitesimal generator of the Markov process, or equivalently the stochastic Koopman operator (SKO), describes how observables evolve forward in time, in the mean.
It acts on smooth functions, $q$, which are bounded on compact sets, meaning that 
$q\in \mathcal{C}^2(\Omega^n)\cap L_\infty(\Omega^n)$
for every compact subset
$\Omega^n\subset\mathbb{R}^n$.
The forward operator $\mathcal{L}$, or Perron-Frobenius operator, specifies how probability densities evolve forward in time. 
It acts on smooth, integrable functions, specifically $\mathcal{C}^2(\R^n)\cap L_1(\R^n)$.
Note that (under an appropriate choice of boundary conditions) $\mathcal{L}$ and $\mathcal{L}^\dagger$ are adjoint to one another with respect to the standard inner product on $\R^n$.

\subsubsection{Single-oscillator assumptions}

We restrict attention to systems for which the forward and backward operators have a discrete set of eigenvalues (with no accumulation point at zero) with a bi-orthonormal set of eigenfunctions
\begin{equation}\label{eq: biorth}
\begin{split}
    &\mathcal{L}[P_\lambda] = \lambda P_\lambda, \quad \mathcal{L}^\dagger[Q^*_\lambda] = \lambda Q^*_\lambda
    \\
    &\langle Q_\lambda,P_{\lambda'} \rangle = \int d\mathbf{x}\; Q_\lambda^*(\mathbf{x}) P_{\lambda'}(\mathbf{x}) \;  = \delta_{\lambda,\lambda'}
\end{split}
\end{equation}
and for which the Fredholm alternative holds for both $\mathcal{L}$ and $\mathcal{L}^\dagger$.
Under these assumptions, the conditional probability density of \eqref{eq:stochastic_oscillator} may be expressed as a linear combination of the eigenmodes
\begin{equation}
    \rho(\mathbf{y},t \given \mathbf{x},s) = P_0 + \sum_{\lambda\neq 0} e^{\lambda(t-s)}P_\lambda(\mathbf{y})Q^*_\lambda(\mathbf{x})
\end{equation}
Here, $P_0$ is the stationary distribution corresponding to $\lambda_0=0$, which we assume to be unique ($\dim\ker\mathcal{L}=1$).
We further assume that the system satisfies the  \textit{robustly oscillatory criteria}, i.e., that

\begin{assumption}\label{robustly_osc_criteria} \textit{(robustly oscillatory criteria)}
\begin{enumerate}
    \item (stability) all eigenvalues apart from the stationary eigenvalue, $\lambda_0$, have negative real parts,
    \item (oscillation) the first nontrivial eigenvalue, $\lambda_1 = \mu \pm i\omega $, is a unique complex-conjugate pair, with quality factor $|\omega/\mu| \gg 1$, and
    \item (spectral gap) all other eigenvalues $\lambda'$ have more negative real parts, $\Re[{\lambda'}] < \upsilon\,\Re[{\lambda_1}]$, where $0<\upsilon\in \mathbb{R}$.
\end{enumerate}
\end{assumption}

These conditions ensure that at intermediate-long times, the dynamics of the density obey
\begin{equation}
    \rho(\mathbf{y},t \given \mathbf{x},s) - P_0 \approx    e^{\lambda_1(t-s)}P_{\lambda_1}(\mathbf{y})Q^*_{\lambda_1}(\mathbf{x}) + \text{c.c.}
\end{equation}
which implies that the dynamics are dominated by the slowest decaying eigenfunctions.
Of particular utility is the slowest decaying complex-valued backward mode, the $Q$-function.
One may extract a well-defined phase for stochastic oscillators by considering the complex argument of the $Q$-function
\begin{equation}
    \Psi(\mathbf{x}) = \arg{Q_{\lambda_1}^*(\mathbf{x})},
\end{equation}
where we take $\text{arg}(z)\in[0,2\pi)$ for complex $z\not=0$.
We call $\Psi$ the \emph{stochastic asymptotic phase}.
More generally, the $Q$-function itself represents a  simplifying change of coordinates with advantages analogous to classical phase reduction.
Trajectories in $Q$-function coordinates evolve linearly in the mean along trajectories $\mbX(t)$ \cite{perez2021isostables}
\begin{equation}\label{eq:Q evolution}
    \frac{d}{dt}\mathbb{E}^\mbx[Q^*_{\lambda_1}(\mathbf{X}(t))] = \lambda_1 \mathbb{E}^\mbx[Q^*_{\lambda_1}(\mathbf{X}(t))]
\end{equation}
where $\mathbb{E}^\mbx$ denotes expectation with respect to the law of the process $\mbX(t)$ started with initial condition $\mbX(0)=\mbx$ \cite{oksendal2013stochastic}.\footnote{We will use $\langle q\rangle$ to denote expectation of a random variable with respect to the stationary distribution $P_0$.}
Moreover, with the convention that $Q^*_{\lambda_1}$ is normalized to have unit variance, i.e., $\langle |Q_{\lambda_1}^*(\mathbf{X}(t))|^2 \rangle = 1$, exact expressions are available for the power spectrum of $Q^*_{\lambda_1}(\mathbf{X}(t))$,
\begin{equation}\label{eq:Q function power spectrum}
    S(\nu) = \frac{2|\mu|}{\mu^2 + (\nu - \omega)^2}
\end{equation}
and for the cross-spectrum between eigenfunctions $Q^*_{\lambda}(\mathbf{X}(t))$ and $Q^*_{\lambda'}(\mathbf{X}(t))$ corresponding to different eigenvalues $\lambda,$ $\lambda'$ \cite{perez2023universal}
\begin{equation}\label{eq:Q function cross spectrum}
    C_{\lambda',\lambda}(\nu) = -\langle Q_{\lambda'}^*Q_\lambda\rangle \left(\frac{1}{\lambda'-i\nu} + \frac{1}{\lambda^*+i\nu}\right).
\end{equation}

The $Q$-function directly generalizes the notion of deterministic phase.
That is, in the case where a system of the form \eqref{eq:stochastic_oscillator} with $\mathbf{g}(\mathbf{x})=0$ admits a stable limit-cycle solution of period $T=2\pi/\omega$, the stochastic asymptotic phase reduces to the classical deterministic phase function.
If $\Theta(\mathbf{x})$ is the deterministic phase function, the function $Q^* = e^{i\Theta(\mathbf{x})}$ is an eigenfunction of the backward operator with eigenvalue $\lambda = i \omega$ \cite{thomas2014asymptotic}.

\subsubsection{Coupled-oscillator  assumptions}\label{para:general_assumptions}
Now consider a system of symmetrically coupled stochastic oscillators of the form
\begin{equation}\label{eq:SDE}
\begin{split}
    d\mathbf{X} &= [\mathbf{f}(\mathbf{X}) + \tau\mathbf{u}_1(\mbX) + \kappa\mathbf{h}(\mathbf{X},\mathbf{Y})]dt + \mathbf{g}_1(\mathbf{X})d\mathbf{W}_1(t)
    \\
    d\mathbf{Y} &= [\mathbf{f}(\mathbf{Y}) + \tau\mathbf{u}_2(\mbY) +\kappa\mathbf{h}(\mathbf{Y},\mathbf{X} )]dt + \mathbf{g}_2(\mathbf{Y})d\mathbf{W}_2(t)
\end{split}
\end{equation}
where $\mbX(t)$ and $\mbY(t)$ are stochastic processes taking values in $\mathbb{R}^n$, $\mbf, 
\mbu_i: \mathbb{R}^n \to \mathbb{R}^n$ are smooth ($\mathcal{C}^2)$ vector fields, $\mbg_i$ are  nonsingular $n\times m$ matrices with smooth ($\mathcal{C}^2$) entries, and the components of the $d\mbW_i(t)$ are increments of independent Weiner processes.
The function $\mbh : \mathbb{R}^n\times \mathbb{R}^n \to \mathbb{R}^n$ is a smooth ($\mathcal{C}^2$) vector field describing the coupling between the oscillators.
We interpret \eqref{eq:SDE} in the sense of It\^{o}.
Unless stated otherwise below, we assume throughout that the individual isolated oscillators ($\kappa=0$) satisfy the \textit{robustly oscillatory criteria} \ref{robustly_osc_criteria}.
We further assume the following \textit{coupled oscillator criteria}:
\begin{assumption}\label{assumptions:coupled oscillator criteria} \textit{(coupled oscillator criteria)}
    \begin{itemize}
    \item When $\kappa=\tau=0$,
    \begin{itemize}
        \item each individual unperturbed oscillator satisfies the \textit{robustly oscillatory criteria} \ref{robustly_osc_criteria},
        \item the forward and backward operators of the isolated oscillators and the joint system have a discrete eigenvalue spectrum (with no accumulation point at zero), and a complete bi-orthonormal set of eigenfunctions, and
        \item the Fredholm alternative holds for both the forward and backward operators of the joint system.
    \end{itemize}
    
    \item There exist open sets $\mathcal{K},\mathcal{T}\subseteq \mathbb{R}$ containing 0 where 
    for all $(\kappa,\tau)\in \mathcal{K}\times \mathcal{T}$, 
\begin{itemize} \item the joint system admits a unique stationary distribution $P_0(\mbx,\mby;\kappa,\tau)$ which corresponds to the unique zero eigenvalue of its forward operator,
\item the stationary distribution, when written as  $P_0(\mbx,\mby;\epsilon \kappa,\epsilon \tau)$, has $\mathcal{C}^1$ dependence on $\epsilon$, under the re-parameterization $(\kappa,\tau) \to (\epsilon \kappa,\epsilon \tau)$, for $|\epsilon|\le 1$, and
\item  all other forward and backward eigenfunctions and corresponding eigenvalues are also $\mathcal{C}^1$ in $\epsilon$ under the same re-parameterization $(\kappa,\tau)\to(\epsilon \kappa,\epsilon \tau)$ for $|\epsilon|\le 1$.
\end{itemize}
\end{itemize}
\end{assumption}

In isolation $(\kappa=0)$, each individual oscillator admits its own eigenfunction expansion, and has its own $Q$-function and corresponding forward eigenfunction.
However, the $2n$-dimensional system \eqref{eq:SDE} also has an eigenfunction expansion.
It is straightforward to show that when $\kappa=0$, eigenfunctions of the system \eqref{eq:SDE} may be written as products of eigenfunctions of the isolated oscillators.
The corresponding eigenvalues are the sums of the eigenvalues of the isolated oscillators.
Because the backward eigenfunction $Q_0^*=1$ with eigenvalue $\lambda_0=0$ is always an eigenpair of a robustly oscillatory system, the individual $Q$-functions and eigenvalues are also eigenfunctions of the $2n$-dimensional system \eqref{eq:SDE} when $\kappa=0$.
Formally, these eigenfunctions are functions of both $\mathbf{X}$ and $\mathbf{Y}$; practically they depend on only their respective coordinates.
As a consequence, \eqref{eq:SDE} may not satisfy the robustly oscillatory criteria (lack of a spectra gap) when $\kappa=0$. 
As a function of $\kappa$, one generically expects that \eqref{eq:SDE} admits a family of eigenvalues and eigenfunctions perturbed from their $\kappa=0$ counterparts.
Such eigenfunctions may depend on both $\mathbf{X}$ and $\mathbf{Y}$ non-trivially.
In the following, when we speak of the $Q$ functions (plural), we refer to the eigenfunctions of the $2n$-dimensional coupled system as a function of both variables, as opposed to the $Q$ functions of the isolated oscillators in only one variable, unless otherwise specified.

We remark that the small parameters $\kappa$ and $\tau$ are analogous to coupling strength and frequency difference parameters for Arnold tongues for \textit{deterministic oscillators}.
In \S\ref{section:results}, our Corollary \ref{corollary:splitting discriminant} describes how to analytically construct similar regions in the two-parameter plane for \textit{stochastic} oscillators.

\subsubsection{The Fredholm alternative}

In \S\ref{section:results}, we use the Fredholm alternative for differential operators, which we state below without proof for the purpose of a self-contained manuscript.
Let the operators $L$ and $L^\dagger$, along with specific boundary conditions, be adjoint with respect to some inner product $\langle\cdot,\cdot\rangle$, i.e.  $\langle L u,v\rangle=\langle u,L^\dagger v\rangle$ for all $u$ and $v$. 
See \S4.3.5 in \cite{keener2018principles} for more details.
\begin{theorem}\label{freddy} 
(The Fredholm alternative)
    For the differential operator $L$
    \begin{enumerate}
        \item The solution of $Lu=f$ is unique if and only if $Lu=0$ has only the trivial solution $u=0$.
        \item The solution of $Lu=f$ exists if and only if $\langle f,v \rangle=0$ for all $v$ for which $L^\dagger v =0$.
    \end{enumerate}
\end{theorem}

In what follows, we take $L^\dagger$ to be the operator adjoint to $L$ with respect to the inner product as defined in equation \eqref{eq: biorth}.
Practically, the Fredholm alternative implies that the inhomogenous equation $Lu=f$ has a solution if and only if the function $f$ is orthogonal to the span of the nullspace of the adjoint operator, $L^\dagger$.
In the following, we refer to this `orthogonality constraint' as a solvability criterion for $Lu=f$.

\subsubsection{Notation}
We denote the $Q$-functions and corresponding forward eigenfunctions for the $\mbx$ and $\mby$ \textit{isolated} oscillators as $Q_1^*(\mbx),P_1(\mbx)$ and $Q_1^*(\mby),P_1(\mby)$, respectively.
We denote the $Q$-functions and corresponding forward eigenfunctions of the joint identical system \eqref{eq:SDE} $(\kappa = 0, \tau=0,\mbg_1=\mbg_2)$ as
\begin{equation}\label{eq:notation}
    \begin{split}
        Q_{1x}^*(\mbx,\mby) &\equiv Q_1^*(\mbx)\mathbb{1}(\mby)
        \\
        P_{1x}(\mbx,\mby) &\equiv P_1(\mbx)P_0(\mby)
        \\
        Q_{1y}^*(\mbx,\mby) &\equiv Q_1^*(\mby)\mathbb{1}(\mbx)
        \\
        P_{1y}(\mbx,\mby) &\equiv P_1(\mby)P_0(\mbx)
    \end{split}
\end{equation}
where $\mathbb{1}(\mbz)=1$ is the  constant unit function for any input $\mbz\in\mathbb{R}^n$, and $P_0(\mbz)$ is understood to represent the stationary distribution of an individual oscillator.
The eigenvalues associated with these eigenfunctions are denoted $\lambda_1$ and have algebraic multiplicity 2 because each oscillator is identical when unperturbed.
We use calligraphic typesetting $\mathcal{Q}_\pm^*(\mbx,\mby)$ and $\mathcal{P}_\pm(\mbx,\mby)$ to denote the true eigenfunctions of \eqref{eq:SDE} when at least one of $\kappa,\tau$ is nonzero.
We denote the corresponding eigenvalues $\lambda_\pm(\kappa,\tau)=\mu_\pm+i\omega_\pm$.
The $\pm$ subscript is meant to indicate the splitting of the eigenvalues under perturbation, as we will show in our Theorem \ref{theorem:Q function perturbation}.


\section{Results}\label{section:results}

In the following sections, we study the relation between the eigenfunctions and eigenvalues of the jointly perturbed ($\kappa,\tau\neq 0$) and unperturbed $(\kappa=\tau=0)$ system \eqref{eq:SDE}.
In \S\ref{subsec: stationary}, we study the distortion of the stationary distribution, while in \S\ref{subsection:results_Q_functions} we study the distortion of the $Q$-function and corresponding eigenvalues.
Then, we relate qualitative changes in the $Q$-function eigenvalues to qualitative changes in the power spectra and cross-spectral density in \S\ref{ssec:PowerSpectra_and_CrossSpectralDensity}.
These results motivate a novel definition of synchronization, which we discuss in \S\ref{subsec: q synch}.
We apply our theory to several examples in \S\ref{section:examples}.

\subsection{Perturbation of the stationary distribution}\label{subsec: stationary}

Generically, one expects that a nonzero perturbation $\kappa,\tau\neq 0$ will distort the unperturbed stationary distribution of \eqref{eq:SDE}.
To that end, we seek a description of the stationary distribution of \eqref{eq:SDE} to leading order in both $\kappa$ and $\tau$.
The stationary distribution is the eigenfunction, $\mathcal{P}_0(\mbx,\mby;\kappa,\tau)$, corresponding to eigenvalue, $\lambda_0=0$, of the Fokker-Planck equation.
The stationary distributions of the isolated oscillators satisfy
\begin{equation}
    \begin{split}
        \mathcal{L}_x P_0(\mbx) &= 0
        \\
        \mathcal{L}_y P_0(\mby) &= 0
    \end{split}
\end{equation}
The stationary distribution of the jointly coupled system satisfies
\begin{equation}\label{eq: forward equation}
    (\mathcal{L}_x + \mathcal{L}_y + \tau \mathcal{L}_{\tau} + \kappa \mathcal{L}_\kappa) \mathcal{P}_0(\mbx,\mby;\kappa,\tau) = 0
\end{equation}
The following lemma describes the leading-order form of the stationary distribution, $\mathcal{P}_0(\mbx,\mby;\kappa,\tau)$, under simultaneous perturbations in both $\kappa$ and $\tau$.

\begin{lemma}\label{lemma:stationary distribution perturbation}
Assume the coupled oscillator criteria \ref{assumptions:coupled oscillator criteria} hold for a system of symmetrically coupled oscillators of the form \eqref{eq:SDE}.
Then, for fixed $(\kappa,\tau)\in \mathcal{K}\times\mathcal{T}$, under the re-parameterization $(\kappa,\tau) \to (\epsilon\kappa,\epsilon \tau)$, there exists a unique function, $\mathcal{P}_c(\mbx,\mby;\kappa,\tau)$, which satisfies
\begin{equation}
\begin{split}
    &(\mathcal{L}_x + \mathcal{L}_y)\mathcal{P}_c(\mbx,\mby;\kappa,\tau) = - \left(\tau \mathcal{L}_{\tau} + \kappa \mathcal{L}_\kappa \right) [P_0(\mbx)P_0(\mby)]
    \\
    &\iint_{\mathbb{R}^{2n}} d\mbx \, d\mby \;\mathcal{P}_c(\mbx,\mby;\kappa,\tau) = 0
\end{split}
\end{equation}
Moreover, the stationary distribution under the re-parameterization, $\mathcal{P}_0(\mbx,\mby; \epsilon \kappa,\epsilon \tau)$, of \eqref{eq:SDE} is given to leading order $\epsilon$ by
\begin{equation}
    \mathcal{P}_0(\mbx,\mby;\epsilon\kappa,\epsilon\tau) = P_0(\mbx)P_0(\mby) + \epsilon\mathcal{P}_c(\mbx,\mby;\kappa,\tau)  + \mathcal{O}(\epsilon^2)
\end{equation}
\end{lemma}

\begin{proof}
Note that when $\kappa=\tau=0$, it is clear that $\mathcal{P}_0(\mbx,\mby;\kappa,\tau)=P_0(\mbx)P_0(\mby)$.
Now, for fixed $(\kappa,\tau)\in \mathcal{K}\times \mathcal{T}$, consider the re-parameterization
$(\kappa,\tau) \to (\epsilon \kappa, \epsilon \tau)$ for  $|\epsilon|\leq 1$.
By assumption, the stationary distribution has an expansion of the form
\begin{equation}
    \mathcal{P}_0(\mbx,\mby;\epsilon\kappa,\epsilon\tau) = P_0(\mbx)P_0(\mby) + \epsilon\mathcal{P}_c(\mbx,\mby;\kappa,\tau) + \mathcal{O}(\epsilon^2)
\end{equation}
Then, the forward equation \eqref{eq: forward equation} may be expressed in terms of $\epsilon$,
\begin{equation}
\begin{split}
    &(\mathcal{L}_x + \mathcal{L}_y + \epsilon \tau \mathcal{L}_\tau + \epsilon\kappa \mathcal{L}_\kappa)\left(P_0(\mbx)P_0(\mby) + \epsilon\mathcal{P}_c(\mbx,\mby;\kappa,\tau) \right) 
    \\
    &\quad+ \mathcal{O}(\epsilon^2) = 0
\end{split}
\end{equation}
To first order in $\epsilon$, we obtain
\begin{equation}\label{eq:linear expansion forward}
    (\mathcal{L}_x + \mathcal{L}_y)\mathcal{P}_c(\mbx,\mby;\kappa,\tau) = - (\tau \mathcal{L}_\tau + \kappa \mathcal{L}_\kappa)[P_0(\mbx) P_0(\mby)] 
\end{equation}
By the assumption of a unique stationary distribution, equation \eqref{eq:linear expansion forward}
is guaranteed to have a solution.
However, we remark that the first order correction, $\mathcal{P}_c(\mbx,\mby;\kappa,\tau)$, need not be unique, as any solution of the form $\mathscr{P}(\mbx,\mby) = \mathcal{P}_c(\mbx,\mby;\kappa,\tau) + k P_0(\mbx)P_0(\mby)$ for $k\in \mathbb{R}$ also satisfies \eqref{eq:linear expansion forward}.
To establish uniqueness, we impose the condition
\begin{equation}\label{eq:constraint}
        \mathcal{P}_c(\mbx,\mby;\kappa,\tau) \perp Q_0^*(\mbx)Q_0^*(\mby)
    \end{equation}
    or rather 
    \begin{equation}\label{eq:stationary orthogonality constraint}
        \iint_{\mathbb{R}^{2n}} d\mbx d\mby \; \mathcal{P}_c(\mbx,\mby;\kappa,\tau) = 0
    \end{equation}
    which follows immediately from the normalization of the stationary distribution.
    We claim that together, the two conditions 
    \begin{equation}\label{eq:linear expansion forward 2}
        \begin{split}
            (\mathcal{L}_x + \mathcal{L}_y) \mathscr{P}(\mbx,\mby) &= -\left(\tau \mathcal{L}_{\tau} + \kappa \mathcal{L}_\kappa \right) [P_0(\mbx) P_0(\mby)]
            \\
            \iint_{\mathbb{R}^{2n}} d\mbx d\mby \; \mathscr{P}(\mbx,\mby) &= 0
        \end{split}
    \end{equation}
    uniquely specify $\mathscr{P}(\mbx,\mby) = \mathcal{P}_c(\mbx,\mby;\kappa,\tau)$, i.e., that $k=0$.
    Indeed, note that
    \begin{equation}
        \begin{split}
            0&= \iint_{\mathbb{R}^{2n}} d\mbx d\mby \; \mathscr{P}(\mbx,\mby) 
            \\
            &= \iint_{\mathbb{R}^{2n}} d\mbx d\mby \; \mathcal{P}_c(\mbx,\mby;\kappa,\tau) + kP_0(\mbx)P_0(\mby)
            \\
            &=\iint_{\mathbb{R}^{2n}} d\mbx d\mby \; \mathcal{P}_c(\mbx,\mby;\kappa,\tau) + k\iint_{\mathbb{R}^{2n}} d\mbx d\mby\; P_0(\mbx)P_0(\mby)
            \\
            &= 0 + k\iint_{\mathbb{R}^{2n}} d\mbx d\mby\; P_0(\mbx)P_0(\mby)
            \\
            &= k
        \end{split}
    \end{equation}

    Thus the arbitrary degree of freedom arising from the nullspace of $\mathcal{L}_x + \mathcal{L}_y$ is fixed by \eqref{eq:stationary orthogonality constraint}; the solution to \eqref{eq:linear expansion forward 2} is unique.
\end{proof}

We remark that any nonzero scalar multiple of an eigenfunction is still an eigenfunction. 
Note that the normalization of the unperturbed stationary distribution \eqref{eq:SDE}, $P_0(\mbx)P_0(\mby)$, determines the normalization of $\mathcal{P}_c(\mbx,\mby;\kappa,\tau)$ by virtue of \eqref{eq:linear expansion forward}.
Moreover, note that to leading order, the perturbed stationary distribution (under re-parameterization) of $\eqref{eq:SDE}$, $\mathcal{P}_0(\mbx,\mby;\epsilon\kappa,\epsilon\tau)$, integrates to unity by virtue of \eqref{eq:stationary orthogonality constraint}.
Consequently, the first order correction, $\mathcal{P}_c(\mbx,\mby;\kappa,\tau)$, maintains normalization of the stationary distribution of the jointly coupled system.
Similarly, the backward eigenfunction corresponding to the zero eigenvalue is identically equal to unity, whether we consider a single oscillator, a system of two or more oscillators, whether identical or not and whether coupled or not.  Thus the first-order correction to the backward eigenfunction for the zero eigenvalue is identically zero.  

Appendix \ref{appendix:stationary eigenfunction perturbation} provides
an example computation for the leading order correction to a stationary distribution.


\subsection{Perturbation of the Q-functions}\label{subsection:results_Q_functions}

Generically, one expects that nonzero perturbation $\kappa,\tau \neq 0$ will distort the unperturbed $Q$-functions of \eqref{eq:SDE}.
We therefore seek a description of the $Q$-function to leading order in both $\kappa$ and $\tau$.
Note that because the oscillators are identical when $\kappa=\tau=0$, the eigenvalue corresponding to these eigenfunctions is repeated with algebraic multiplicity two.
The geometric multiplicity of the eigenvalue is also two, because each oscillator contributes a linearly independent eigenfunction.
The $Q$-functions of the isolated oscillators satisfy
\begin{equation}
    \begin{split}
        \mathcal{L}_x^\dagger Q_1^*(\mbx) &= \lambda_1 Q_1^*(\mbx)
        \\
        \mathcal{L}_y^\dagger Q_1^*(\mby) &= \lambda_1 Q_1^*(\mby)
    \end{split}
\end{equation}
Because there are two $Q$-functions when $\kappa=\tau=0$, generically one expects perturbed versions of both when $\kappa,\tau\neq 0$.
We denote the perturbed $Q$-functions of the joint system as $\mathcal{Q}_+(\mbx,\mby;\kappa,\tau)$ and $\mathcal{Q}_-(\mbx,\mby;\kappa,\tau)$, with corresponding eigenvalues $\lambda_+(\kappa,\tau)$ and $\lambda_-(\kappa,\tau)$.
These eigenfunctions satisfy 
\begin{equation}
    \begin{split}
        &(\mathcal{L}_x^\dagger + \mathcal{L}_y^\dagger + \tau \mathcal{L}_{\tau}^\dagger + \kappa \mathcal{L}_{\kappa}^\dagger)\mathcal{Q}_+^*(\mbx,\mby;\kappa,\tau) =
        \\
        &\quad\lambda_+(\kappa,\tau)\mathcal{Q}_+^*(\mbx,\mby;\kappa,\tau)
        \\
        \\
        &(\mathcal{L}_x^\dagger + \mathcal{L}_y^\dagger + \tau \mathcal{L}_{\tau}^\dagger + \kappa \mathcal{L}_{\kappa}^\dagger)\mathcal{Q}_-^*(\mbx,\mby;\kappa,\tau) =
        \\
        &\quad\lambda_-(\kappa,\tau)\mathcal{Q}_-^*(\mbx,\mby;\kappa,\tau)
    \end{split}
\end{equation}
Analogous statements hold for the corresponding isolated forward eigenfunctions of the Fokker-Planck operator, $P_1(\mbx)$ and $P_1(\mby)$, along with the perturbed forward eigenfunctions of the joint system, $\mathcal{P}_+(\mbx,\mby;\kappa,\tau)$ and $\mathcal{P}_-(\mbx,\mby;\kappa,\tau)$.
Recall that we denote the $Q$-functions and corresponding forward eigenfunctions of the unperturbed $(\kappa=\tau=0)$ \textit{joint} system as $Q^*_{1x}(\mbx,\mby)$, $Q^*_{1y}(\mbx,\mby)$, $P_{1x}(\mbx,\mby)$, and $P_{1y}(\mbx,\mby)$ (see the notation established in \eqref{eq:notation}).

\begin{theorem}\label{theorem:Q function perturbation}
Assume the coupled oscillator criteria \ref{assumptions:coupled oscillator criteria} hold for a system of symmetrically coupled oscillators of the form \eqref{eq:SDE}.
Define the $2\times 2$ complex-valued matrix
\begin{widetext}
\begin{equation}\label{eq:M matrix}
  \mathcal{M}(\kappa,\tau) = \begin{bmatrix}
        \iint_{\mathbb{R}^{2n}} d\mbx d\mby \; P_{1x}(\mbx,\mby)[\tau \mathcal{L}_\tau^\dagger + \kappa \mathcal{L}_\kappa^\dagger]Q_{1x}^*(\mbx,\mby) & \iint_{\mathbb{R}^{2n}} d\mbx d\mby \; P_{1x}(\mbx,\mby)[\tau \mathcal{L}_\tau^\dagger + \kappa \mathcal{L}_\kappa^\dagger]Q_{1y}^*(\mbx,\mby)
        \\
        \iint_{\mathbb{R}^{2n}} d\mbx d\mby \; P_{1y}(\mbx,\mby)[\tau \mathcal{L}_\tau^\dagger + \kappa \mathcal{L}_\kappa^\dagger]Q_{1x}^*(\mbx,\mby) & \iint_{\mathbb{R}^{2n}} d\mbx d\mby \; P_{1y}(\mbx,\mby)[\tau \mathcal{L}_\tau^\dagger + \kappa \mathcal{L}_\kappa^\dagger]Q_{1y}^*(\mbx,\mby)
    \end{bmatrix}
\end{equation}
\end{widetext}

Then, for any fixed $(\kappa,\tau)\in\mathcal{K}\times \mathcal{T}$ for which $\mathcal{M}$ is diagonalizable with eigenpairs $(\lambda_{c\pm}(\kappa,\tau),v_{\pm}(\kappa,\tau))$, there exist unique\footnote{The uniqueness of $\mathcal{Q}_{c+}(\mbx,\mby;\kappa,\tau)$ and $\mathcal{Q}_{c-}(\mbx,\mby;\kappa,\tau)$ holds for a fixed, \textit{a-priori} normalization of the unperturbed $Q$ functions.
See the remark after the proof for a more detailed discussion of uniqueness.} functions $\mathcal{Q}_{c+}(\mbx,\mby;\kappa,\tau)$ and $\mathcal{Q}_{c-}(\mbx,\mby;\kappa,\tau)$ that satisfy

\begin{widetext}
\begin{equation}\label{eq:splitting_equation}
\begin{split}
    (\mathcal{L}_x^\dagger + \mathcal{L}_y^\dagger - \lambda_1)\mathcal{Q}_{c+}^*(\mbx,\mby;\kappa,\tau) &= - (\tau\mathcal{L}_{\tau}^\dagger+\kappa\mathcal{L}_\kappa^\dagger -\lambda_{c+}(\kappa,\tau))\Big( (v_+(\kappa,\tau))_1Q_{1x}^*(\mbx,\mby)+(v_+(\kappa,\tau))_2Q_{1y}^*(\mbx,\mby) \Big)
    \\
    (\mathcal{L}_x^\dagger + \mathcal{L}_y^\dagger - \lambda_1)\mathcal{Q}_{c-}^*(\mbx,\mby;\kappa,\tau) &= - (\tau\mathcal{L}_{\tau}^\dagger + \kappa\mathcal{L}_\kappa^\dagger -\lambda_{c-}(\kappa,\tau))\Big( (v_-(\kappa,\tau))_1Q_{1x}^*(\mbx,\mby)+(v_-(\kappa,\tau))_2Q_{1y}^*(\mbx,\mby) \Big) 
\end{split}
\end{equation}
\end{widetext}
where the functions $\mathcal{Q}^*_{c\pm}(\mbx,\mby;\kappa,\tau)$ satisfy the orthogonality conditions
\begin{equation}
    \begin{split}
        \mathcal{Q}_{c+}^*(\mbx,\mby;\kappa,\tau) &\perp P_{1x}(\mbx,\mby), P_{1y}(\mbx,\mby)
        \\
        \mathcal{Q}_{c-}^*(\mbx,\mby;\kappa,\tau) &\perp P_{1x}(\mbx,\mby), P_{1y}(\mbx,\mby)
    \end{split}
\end{equation}
The notation $(v_{\pm}(\kappa,\tau))_j$ is understood as the $j$th element of the vector $v_{\pm}(\kappa,\tau)$, for $j=1,2$.
Under the re-parameterization $(\kappa,\tau)\to(\epsilon\kappa,\epsilon\tau)$, the $Q$-functions and eigenvalues of the joint system \eqref{eq:SDE} are given to leading order in $\epsilon$ by
\begin{widetext}
\begin{equation}
\begin{split}
    \mathcal{Q}_+^*(\mbx,\mby;\epsilon\kappa,\epsilon\tau) &= (v_+(\kappa,\tau))_1Q_{1x}^*(\mbx,\mby)+(v_+(\kappa,\tau))_2Q_{1y}^*(\mbx,\mby) + \epsilon\mathcal{Q}_{c+}^*(\mbx,\mby;\kappa,\tau) + \mathcal{O}(\epsilon^2)
    \\
    \mathcal{Q}_-^*(\mbx,\mby;\epsilon\kappa,\epsilon\tau) &= (v_-(\kappa,\tau))_1Q_{1x}^*(\mbx,\mby)+(v_-(\kappa,\tau))_2Q_{1y}^*(\mbx,\mby) + \epsilon\mathcal{Q}_{c-}^*(\mbx,\mby;\kappa,\tau) + \mathcal{O}(\epsilon^2)
    \\
    \lambda_{+}(\epsilon\kappa,\epsilon\tau) &= \lambda_1 +  \epsilon\lambda_{c+}(\kappa,\tau) + \mathcal{O}(\epsilon^2)
    \\
    \lambda_{-}(\epsilon\kappa,\epsilon\tau) &= \lambda_1 +  \epsilon\lambda_{c-}(\kappa,\tau) + \mathcal{O}(\epsilon^2)
\end{split}
\end{equation}
\end{widetext}
\end{theorem}

\begin{proof}
 Fix $(\kappa,\tau)\in \mathcal{K}\times \mathcal{T}$, and consider the re-parameterization $(\kappa,\tau) \to (\epsilon \kappa, \epsilon \tau)$ for $|\epsilon|\leq 1$.
By assumption, an expansion of the form
\begin{equation}\label{eq:ansatz backward Q}
        \begin{split}
            \mathcal{Q}^*(\mbx,\mby;\epsilon\kappa,\epsilon\tau) &= c_1Q_{1x}^*(\mbx,\mby)+c_2Q_{1y}^*(\mbx,\mby) + 
            \\
            &\quad\epsilon \mathcal{Q}_c^*(\mbx,\mby;\kappa,\tau) + \mathcal{O}(\epsilon^2)
            \\
            \lambda(\epsilon\kappa,\epsilon\tau) &= \lambda_1 + \epsilon \lambda_c + \mathcal{O}(\epsilon^2)
        \end{split}
    \end{equation}
    is guaranteed to exist.
    Here, $c_1$, $c_2$, and $\lambda_c$ are unknown constants (possibly depending on $\kappa$ and $\tau$) to be determined in the following.
    Insertion of the ansatz \eqref{eq:ansatz backward Q} into the backward equation
    \begin{equation}
    \begin{split}
        &(\mathcal{L}_x^\dagger + \mathcal{L}_y^\dagger + \epsilon\tau\mathcal{L}_{\tau}^\dagger+ \epsilon\kappa\mathcal{L}_{\kappa}^\dagger)\mathcal{Q}^*(\mbx,\mby;\epsilon\kappa,\epsilon\tau) 
        \\
        &\quad- \lambda(\epsilon\kappa,\epsilon\tau) \mathcal{Q}^*(\mbx,\mby;\epsilon\kappa,\epsilon\tau) = 0
    \end{split}
    \end{equation}
    gives, to first order in $\epsilon$
    \begin{equation}\label{eq:Q function linear equation}
    \begin{split}
    \\
        &(\mathcal{L}_x^\dagger + \mathcal{L}_y^\dagger - \lambda_1)\mathcal{Q}^*_c(\mbx,\mby;\kappa,\tau) = 
        \\
        &\;-(\tau\mathcal{L}_{\tau}^\dagger +\kappa\mathcal{L}_{\kappa}^\dagger - \lambda_c)\Big[c_1Q_{1x}^*(\mbx,\mby)+c_2Q_{1y}^*(\mbx,\mby) \Big]
    \end{split}
    \end{equation}
    Note that the nullspace of the adjoint of $\mathcal{L}_x^\dagger + \mathcal{L}_y^\dagger - \lambda_1$ is spanned by the forward eigenfunctions of the respective oscillators, $P_{1x}(\mbx,\mby)$ and $P_{1y}(\mbx,\mby)$.
    The Fredholm alternative (Theorem \ref{freddy}) establishes solvability criteria
    \begin{equation}
    \begin{split}
        -(\tau\mathcal{L}_{\tau}^\dagger +\kappa\mathcal{L}_{\kappa}^\dagger - \lambda_c)\Big[c_1Q_{1x}^*(\mbx,\mby)+c_2Q_{1y}^*(\mbx,\mby) \Big] &\perp P_{1x}(\mbx,\mby)
        \\
        -(\tau\mathcal{L}_{\tau}^\dagger +\kappa\mathcal{L}_{\kappa}^\dagger - \lambda_c)\Big[c_1Q_{1x}^*(\mbx,\mby)+c_2Q_{1y}^*(\mbx,\mby) \Big] &\perp P_{1y}(\mbx,\mby)
    \end{split}
    \end{equation}
    Written explicitly, and using the biorthogonality of the eigenfunctions, one has that
    \begin{widetext}
    \begin{equation}
    \begin{split}
        c_1 \lambda_{c} &= \iint_{\mathbb{R}^{2n}} d\mbx d\mby \; P_{1x}(\mbx,\mby)(\tau\mathcal{L}_{\tau}^\dagger +\kappa\mathcal{L}_{\kappa}^\dagger)[c_1 Q_{1x}^*(\mbx,\mby) + c_2 Q_{1y}^*(\mbx,\mby)]
        \\
        c_2 \lambda_{c} &= \iint_{\mathbb{R}^{2n}} d\mbx d\mby \; P_{1y}(\mbx,\mby)(\tau\mathcal{L}_{\tau}^\dagger +\kappa\mathcal{L}_{\kappa}^\dagger)  [c_1 Q_{1x}^*(\mbx,\mby) + c_2 Q_{1y}^*(\mbx,\mby)]
    \end{split}
\end{equation}
\end{widetext}
which can be expressed as a linear system 
\begin{widetext}
\begin{equation}\label{eq:matrix_equation}
    \lambda_c \begin{bmatrix}
        c_1 \\ c_2
    \end{bmatrix} = \begin{bmatrix}
        \iint_{\mathbb{R}^{2n}} d\mbx d\mby \; P_{1x}(\mbx,\mby)[\tau \mathcal{L}_\tau^\dagger + \kappa \mathcal{L}_\kappa^\dagger]Q_{1x}^*(\mbx,\mby) & \iint_{\mathbb{R}^{2n}} d\mbx d\mby \; P_{1x}(\mbx,\mby)[\tau \mathcal{L}_\tau^\dagger + \kappa \mathcal{L}_\kappa^\dagger]Q_{1y}^*(\mbx,\mby)
        \\
        \iint_{\mathbb{R}^{2n}} d\mbx d\mby \; P_{1y}(\mbx,\mby)[\tau \mathcal{L}_\tau^\dagger + \kappa \mathcal{L}_\kappa^\dagger]Q_{1x}^*(\mbx,\mby) & \iint_{\mathbb{R}^{2n}} d\mbx d\mby \; P_{1y}(\mbx,\mby)[\tau \mathcal{L}_\tau^\dagger + \kappa \mathcal{L}_\kappa^\dagger]Q_{1y}^*(\mbx,\mby)
    \end{bmatrix} \begin{bmatrix}
        c_1 \\ c_2
    \end{bmatrix}
\end{equation}
\end{widetext}
or rather, by the definition of $\mathcal{M}$,
\begin{equation}
     \mathcal{M}(\kappa,\tau)\begin{bmatrix}
        c_1 \\ c_2
    \end{bmatrix} = \lambda_c \begin{bmatrix}
        c_1 \\ c_2
    \end{bmatrix}
\end{equation}
Note that \eqref{eq:matrix_equation} determines the values of $\lambda_c$, $c_1$, and $c_2$.
By assumption, $\mathcal{M}$, is diagonalizable.
Consequently, $\mathcal{M}$ has eigenpairs $(\lambda_{c\pm}(\kappa,\tau),v_{\pm}(\kappa,\tau))$.
The eigenvalues $\lambda_{c\pm}(\kappa,\tau)$ are the leading order corrections to the unperturbed $Q$-function eigenvalues, and the entries of the eigenvectors $v_{\pm}(\kappa,\tau)$ correspond to the weights of the unperturbed $Q$-functions.
The existence of these eigenvalues and eigenfunctions guarantees that \eqref{eq:Q function linear equation} has a solution.

However, we remark that the leading order corrections $\mathcal{Q}_{c+}^*(\mbx,\mby;\kappa,\tau)$ and $\mathcal{Q}_{c-}^*(\mbx,\mby;\kappa,\tau)$ need not be unique, as any solution of the form $\mathscr{Q}^*(\mbx,\mby) = \mathcal{Q}_{c+}^*(\mbx,\mby;\kappa,\tau) + k_1 Q_{1x}^*(\mbx,\mby) + k_2Q_{1y}^*(\mbx,\mby)$ for $k_1,k_2\in \mathbb{C}$ also satisfies \eqref{eq:Q function linear equation} (and similarly for $\mathcal{Q}_{c-}^*(\mbx,\mby;\kappa,\tau)$).
We establish uniqueness of $\mathcal{Q}_{c+}^*(\mbx,\mby;\kappa,\tau)$; the argument for $\mathcal{Q}_{c-}^*(\mbx,\mby;\kappa,\tau)$ is analogous.
To establish uniqueness, we impose the conditions
\begin{equation}\label{eq:constraint Q}
        \mathcal{Q}_{c+}^*(\mbx,\mby;\kappa,\tau) \perp P_{1x}(\mbx,\mby), \quad \mathcal{Q}_{c+}^*(\mbx,\mby;\kappa,\tau) \perp P_{1y}(\mbx,\mby) 
    \end{equation}
    or rather 
    \begin{equation}\label{eq:Q orthogonality constraint}
    \begin{split}
        &\iint_{\mathbb{R}^n} d\mbx d\mby \; \mathcal{Q}_{c+}(\mbx,\mby;\kappa,\tau)P_{1x}(\mbx,\mby) = 0
        \\
        &\iint_{\mathbb{R}^n} d\mbx d\mby \; \mathcal{Q}_{c+}(\mbx,\mby;\kappa,\tau)P_{1y}(\mbx,\mby) = 0
    \end{split}  
    \end{equation}
    We claim  that together, the conditions
    \begin{equation}\label{eq:linear expansion forward Q}
        \begin{split}
            &(\mathcal{L}_x^\dagger + \mathcal{L}_y^\dagger - \lambda_1) \mathscr{Q}^*(\mbx,\mby) = 
            \\
            &\;-(\tau\mathcal{L}_{\tau}^\dagger + \kappa\mathcal{L}_\kappa^\dagger - \lambda_c) [c_1Q_{1x}^*(\mbx,\mby) +  c_2Q_{1y}^*(\mbx,\mby)]
            \\
            &\iint_{\mathbb{R}^{2n}} d\mbx d\mby \; \mathscr{Q}^*(\mbx,\mby) P_{1x}(\mbx,\mby) = 0
            \\
            &\iint_{\mathbb{R}^{2n}} d\mbx d\mby \; \mathscr{Q}^*(\mbx,\mby) P_{1y}(\mbx,\mby) = 0
        \end{split}
    \end{equation}
    uniquely specify $\mathscr{Q}^*(\mbx,\mby) = \mathcal{Q}_{c+}^*(\mbx,\mby;\kappa,\tau)$, i.e., that $k_1=k_2=0$.
    Indeed, note that
    \begin{equation}
        \begin{split}
            0&= \iint_{\mathbb{R}^{2n}} d\mbx d\mby \; \mathscr{Q}^*(\mbx,\mby) P_{1x}(\mbx,\mby)
            \\
            &=\iint_{\mathbb{R}^{2n}} d\mbx d\mby \; \mathcal{Q}_{c+}^*(\mbx,\mby;\kappa,\tau) P_{1x}(\mbx,\mby) +
            \\
            &\quad k_1\iint_{\mathbb{R}^{2n}} d\mbx d\mby\; Q_{1x}^*(\mbx,\mby)P_{1x}(\mbx,\mby)
            \\
            &\quad+ k_2\iint_{\mathbb{R}^{2n}} d\mbx d\mby\; Q_{1y}^*(\mbx,\mby)P_{1x}(\mbx,\mby)
            \\
            &= 0 + k_1\iint_{\mathbb{R}^{2n}} d\mbx d\mby\; Q_{1x}^*(\mbx,\mby)P_{1x}(\mbx,\mby) + 0
            \\
            &= k_1
        \end{split}
    \end{equation}
    Consideration of the second constraint $\iint_{\mathbb{R}^n} d\mbx d\mby \; \mathscr{Q}^*(\mbx,\mby) P_{1y}(\mbx,\mby) = 0$ shows that $k_2=0$.
    One concludes that the arbitrary degrees of freedom arising from the nullspace of $\mathcal{L}_x^\dagger + \mathcal{L}_y^\dagger - \lambda_1$ are fixed by \eqref{eq:Q orthogonality constraint}; the solution to \eqref{eq:linear expansion forward Q} is unique.
\end{proof}

Note that the normalization of the unperturbed $Q$-functions of \eqref{eq:SDE} determine the normalization of the first order update $\mathcal{Q}^*_c(\mbx,\mby;\kappa,\tau)$ by virtue of \eqref{eq:Q function linear equation}.
In this manuscript, we enforce the condition that the $Q$-functions have unit variance, which uniquely specifies their magnitude \cite{perez2023universal}.
However, the unperturbed $Q$-functions are only defined up to a complex rotation,
i.e., $Q_{1x}^*(\mbx,\mby)e^{i\xi}$ and $Q_{1y}^*(\mbx,\mby)e^{i\Xi}$ are still $Q$-functions of the unperturbed system \eqref{eq:SDE} for any $\xi,\Xi\in [0,2\pi)$.
Therefore, the results of Theorem \ref{theorem:Q function perturbation} should be interpreted in the following sense: the first order update $\mathcal{Q}_c^*(\mbx,\mby;\kappa,\tau)$ is unique given a fixed \textit{a priori} choice of $\xi$ and $\Xi$.
That is, the leading order update, $\mathcal{Q}_c^*(\mbx,\mby;\kappa,\tau)$, is unique up to rotation in the complex plane, and is determined by the choice of rotation of the unperturbed $Q$-functions.
We remark that this ambiguity in the ``phase" of the $Q$-function is directly analogous to the ambiguity of the asymptotic phase function $\Theta(\mbx)$ of a classical oscillator \eqref{eq:oscillator}
which is only defined up to a constant phase shift.  
We discuss this issue further in
\S \ref{ssec:PowerSpectra_and_CrossSpectralDensity}.

The following corollary shows that in the case of identical symmetrically coupled oscillators ($\tau=0$ in \eqref{eq:SDE}), the results of Theorem \ref{theorem:Q function perturbation} take on a particularly simple form.
In particular, exact expressions are available for the leading order correction to the eigenvalues.
Moreover, we show that the eigenfunctions of the perturbed system are related to the sum and difference of the unperturbed eigenfunctions.

\begin{corollary}\label{corollary:id osci}
    Assume the coupled oscillator criteria \ref{assumptions:coupled oscillator criteria} hold for a system of identical symmetrically coupled oscillators of the form \eqref{eq:SDE} (with $\tau=0$). 
Define
\begin{equation}\label{eq:alpha_annd_beta}
\begin{split}
    \alpha &\equiv \iint_{\mathbb{R}^{2n}} d\mbx d\mby \; P_{1x}(\mbx,\mby)(\mathcal{L}_{\kappa}^\dagger Q_{1x}^*(\mbx,\mby))
    \\
    \beta &\equiv \iint_{\mathbb{R}^{2n}} d\mbx d\mby \; P_{1x}(\mbx,\mby)(\mathcal{L}_{\kappa}^\dagger Q_{1y}^*(\mbx,\mby))
\end{split}
\end{equation}
By virtue of the assumed symmetry, we now have
\begin{equation}
    \mathcal{M} = \begin{bmatrix}
        \alpha & \beta \\ \beta & \alpha
    \end{bmatrix}
\end{equation}
There are two cases:
\begin{enumerate}
    \item (No splitting) Suppose $\beta=0$, 
Then, for all $\kappa\in \mathcal{K}$, there exist constants $\lambda_{c+}$ and $\lambda_{c-}$, and unique functions $\mathcal{Q}_{c+}(\mbx,\mby)$ and $\mathcal{Q}_{c-}(\mbx,\mby)$ that satisfy
\begin{equation}
\begin{split}
    (\mathcal{L}_x^\dagger + \mathcal{L}_y^\dagger - \lambda_1)\mathcal{Q}_{c+}^*(\mbx,\mby) &= - (\mathcal{L}_\kappa^\dagger -\lambda_{c+}) Q_{1x}^*(\mbx,\mby)
    \\
    (\mathcal{L}_x^\dagger + \mathcal{L}_y^\dagger - \lambda_1)\mathcal{Q}_{c-}^*(\mbx,\mby) &= - (\mathcal{L}_\kappa^\dagger -\lambda_{c-})Q_{1y}^*(\mbx,\mby)
    \\
    \lambda_{c+} &= \alpha
    \\
    \lambda_{c-} &= \alpha
\end{split}
\end{equation}
along with orthogonality conditions
\begin{equation}
    \begin{split}
        \mathcal{Q}_{c+}^*(\mbx,\mby) &\perp P_{1x}(\mbx,\mby) , P_{1y}(\mbx,\mby) 
        \\
        \mathcal{Q}_{c-}^*(\mbx,\mby) &\perp P_{1x}(\mbx,\mby) , P_{1y}(\mbx,\mby) 
    \end{split}
\end{equation}
The $Q$-functions and eigenvalues of the joint system \eqref{eq:SDE} are given to first order in $\kappa$ by
\begin{equation}\label{eq:underline 1}
\begin{split}
    Q_+^*(\mbx,\mby) &= \underline{Q_{1x}^*(\mbx,\mby)} + \kappa \mathcal{Q}_{c+}^*(\mbx,\mby) + \mathcal{O}(\kappa^2)
    \\
    Q_-^*(\mbx,\mby) &= \underline{Q_{1y}^*(\mbx,\mby)} + \kappa \mathcal{Q}_{c-}^*(\mbx,\mby) + \mathcal{O}(\kappa^2)
    \\
    \lambda_{+} &= \lambda_1 + \kappa \lambda_{c+} + \mathcal{O}(\kappa^2)
    \\
    \lambda_{-} &= \lambda_1 + \kappa \lambda_{c-} + \mathcal{O}(\kappa^2)
\end{split}
\end{equation}
    \item (Splitting) Suppose $\beta \neq 0$.
    Then, for all $\kappa\in \mathcal{K}$, there exist constants $ \lambda_{c+}$ and $\lambda_{c-}$, and unique functions $\mathcal{Q}_{c+}^*(\mbx,\mby)$ and $\mathcal{Q}_{c-}^*(\mbx,\mby)$ that satisfy
    \begin{equation}
\begin{split}
    &(\mathcal{L}_x^\dagger + \mathcal{L}_y^\dagger - \lambda_1)\mathcal{Q}_{c+}^*(\mbx,\mby) = 
    \\
    &\;- (\mathcal{L}_\kappa^\dagger -\lambda_{c+})\Big( Q_{1x}^*(\mbx,\mby)+Q_{1y}^*(\mbx,\mby) \Big)
    \\
    &(\mathcal{L}_x^\dagger + \mathcal{L}_y^\dagger - \lambda_1)\mathcal{Q}_{c-}^*(\mbx,\mby) = 
    \\
    &- (\mathcal{L}_\kappa^\dagger -\lambda_{c-})\Big( Q_{1x}^*(\mbx,\mby)-Q_{1y}^*(\mbx,\mby) \Big)
    \\
    &\lambda_{c+} = \alpha + \beta
    \\
    &\lambda_{c-} = \alpha - \beta
\end{split}
\end{equation}
along with the orthogonality conditions
\begin{equation}
    \begin{split}
        \mathcal{Q}_{c+}^*(\mbx,\mby) &\perp P_{1x}(\mbx,\mby) , P_{1y}(\mbx,\mby) 
        \\
        \mathcal{Q}_{c-}^*(\mbx,\mby) &\perp P_{1x}(\mbx,\mby) , P_{1y}(\mbx,\mby) 
    \end{split}
\end{equation}
    The $Q$-functions and eigenvalues of the joint system \eqref{eq:SDE} are given to first order in $\kappa$ by
\begin{equation}\label{eq:underline 2}
\begin{split}
    \mathcal{Q}_+^*(\mbx,\mby) &= \underline{Q_{1x}^*(\mbx,\mby)+Q_{1y}^*(\mbx,\mby)} + \kappa \mathcal{Q}_{c+}^*(\mbx,\mby) + \mathcal{O}(\kappa^2)
    \\
    \mathcal{Q}_-^*(\mbx,\mby) &= \underline{Q_{1x}^*(\mbx,\mby)-Q_{1y}^*(\mbx,\mby)} + \kappa \mathcal{Q}_{c-}^*(\mbx,\mby) + \mathcal{O}(\kappa^2)
    \\
    \lambda_{+} &= \lambda_1 + \kappa \lambda_{c+} + \mathcal{O}(\kappa^2)
    \\
    \lambda_{-} &= \lambda_1 + \kappa \lambda_{c-} + \mathcal{O}(\kappa^2)
    \\
    \\
\end{split}
\end{equation}
\end{enumerate}
\end{corollary}
We remark that, when $\beta=0$, the unperturbed eigenfunctions of the joint system are the eigenfunctions of the isolated oscillators.
In this case, the $Q$-function eigenvalues do not split.
However, when $\beta\neq 0$, the unperturbed eigenfunctions of the joint system are a particular linear combination (the sum and difference) of the eigenfunctions of the isolated oscillators.
In this case, the $Q$-function eigenvalues do split.
We highlight this distinction by underlining the unperturbed eigenfunction terms in equations \eqref{eq:underline 1} and \eqref{eq:underline 2}.

\begin{proof}
    We follow the proof of Theorem \ref{theorem:Q function perturbation} to the derivation of the matrix equation involving $\mathcal{M}$, which, because $\tau=0$, reduces to the form
    \begin{widetext}
\begin{equation}\label{eq:matrix_equation_id}
    \mathcal{M} = \kappa \begin{bmatrix}
        \iint_{\mathbb{R}^{2n}} d\mbx d\mby \; P_{1x}(\mbx,\mby)(\mathcal{L}_{\kappa}^\dagger Q_{1x}^*(\mbx,\mby)) & \iint_{\mathbb{R}^{2n}} d\mbx d\mby \; P_{1x}(\mbx,\mby)(\mathcal{L}_{\kappa}^\dagger Q_{1y}^*(\mbx,\mby))
        \\
        \iint_{\mathbb{R}^{2n}} d\mbx d\mby \; P_{1y}(\mbx,\mby)(\mathcal{L}_{\kappa}^\dagger Q_{1x}^*(\mbx,\mby)) & \iint_{\mathbb{R}^{2n}} d\mbx d\mby \; P_{1y}(\mbx,\mby)(\mathcal{L}_{\kappa}^\dagger Q_{1y}^*(\mbx,\mby))
    \end{bmatrix}
\end{equation}
\end{widetext}    
    
Notice that the top row of $\mathcal{M}$ is given by $[\alpha \;\beta]$.
We now explicitly verify that the bottom row is given by $[\beta \; \alpha]$, which follows from the assumption of symmetric coupling.
Writing out the backward coupling operator, one sees that
\begin{equation}
\begin{split}
    \mathcal{L}_{\kappa}^\dagger \mathcal{Q}_{1x}^*(\mbx,\mby)  &= \mbh(\mbx,\mby) \cdot \nabla_x Q_{1x}^*(\mbx,\mby)   + \mbh(\mby,\mbx) \cdot \nabla_y Q_{1x}^*(\mbx,\mby)  
    \\
    &= \mbh(\mbx,\mby) \cdot \nabla_x Q_{1x}^*(\mbx,\mby) 
    \\
    \\
    \mathcal{L}_{\kappa}^\dagger \mathcal{Q}_{1y}^*(\mbx,\mby) &= \mbh(\mbx,\mby) \cdot \nabla_x Q_{1y}^*(\mbx,\mby)  + \mbh(\mby,\mbx) \cdot \nabla_y Q_{1y}^*(\mbx,\mby) 
    \\
    &= \mbh(\mby,\mbx) \cdot \nabla_y Q_{1y}^*(\mbx,\mby)
\end{split}
\end{equation}
where we recall the definitions of the unperturbed $Q$-functions of the joint system \eqref{eq:notation}.
Then, because the oscillators are identical and have the same forward and backward eigenfunctions in their respective variables, one sees that 
\begin{equation}
\begin{split}
 \alpha &= \iint_{\mathbb{R}^{2n}} d\mbx d\mby \; P_{1x}(\mbx,\mby)(\mathcal{L}_{\kappa}^\dagger Q_{1x}^*(\mbx,\mby)) 
 \\
&= \iint_{\mathbb{R}^{2n}} d\mbx d\mby \; P_{1y}(\mbx,\mby) (\mathcal{L}_{\kappa}^\dagger Q_{1y}^*(\mbx,\mby))
\\
\\
\beta &= \iint_{\mathbb{R}^{2n}} d\mbx d\mby \; P_{1x}(\mbx,\mby)(\mathcal{L}_{\kappa}^\dagger Q_{1y}^*(\mbx,\mby)) 
\\
&= \iint_{\mathbb{R}^{2n}} d\mbx d\mby \; P_{1y}(\mbx,\mby)(\mathcal{L}_{\kappa}^\dagger Q_{1x}^*(\mbx,\mby))
\end{split}
\end{equation}
Therefore, the matrix equation \eqref{eq:matrix_equation_id} reduces to the system
\begin{equation}
     \kappa\begin{bmatrix}
        \alpha & \beta \\ \beta & \alpha
    \end{bmatrix}\begin{bmatrix}
        c_1 \\ c_2
    \end{bmatrix} = \lambda_c \begin{bmatrix}
        c_1 \\ c_2
    \end{bmatrix} 
\end{equation}

\textbf{Case 1:}
Suppose that $\beta=0$, then $\mathcal{M}$ has two eigenvectors
\begin{equation}
\begin{split}
    \lambda_{c+}v_+ &= \mathcal{M}v_+
    \\
    \lambda_{c-}v_- &= \mathcal{M}v_-
\end{split}
\end{equation}
where $v_+ = \begin{bmatrix}
    1 & 0
\end{bmatrix}$ and $v_- = \begin{bmatrix}
    0 & 1
\end{bmatrix}$
and $\lambda_{c+} = \kappa\alpha$ and $\lambda_{c-}=\kappa\alpha$.
Notice that, in this case, $\mathcal{M}$ is a multiple of the identity matrix, and so any nonzero vector is an eigenvector.
Here, we choose $v_{\pm}$ from a family of possible linear combinations of vectors when $\beta=0$.
However, when $\beta\neq 0$, we will see that the eigenvectors $v_{\pm}$ are specified as a particular linear combination of vectors.

\textbf{Case 2:} 
Suppose that $\beta\neq0$.
Then, $\mathcal{M}$ has two eigenvectors
\begin{equation}
\begin{split}
    \lambda_{c+}v_+ &= \mathcal{M}v_+
    \\
    \lambda_{c-}v_- &= \mathcal{M}v_-
\end{split}
\end{equation}
where $v_+ = \begin{bmatrix}
    1 & 1
\end{bmatrix}$ and $v_- = \begin{bmatrix}
    1 & -1
\end{bmatrix}$
and $\lambda_{c+} = \kappa(\alpha+\beta)$ and $\lambda_{c-}=\kappa(\alpha-\beta)$.

To conclude, we remark that uniqueness of the leading order correction to the eigenfunctions follows from the proof of Theorem \ref{theorem:Q function perturbation}.
\end{proof}

The fact that the perturbed $Q$-functions of the joint system of identical coupled oscillators go as the sum and difference of the unperturbed $Q$-functions bears some remark.
To gain physical intuition, consider two ideal mass-spring systems with equal length and mass that are coupled with a spring
\begin{equation}\label{eq:normal_mode_system}
    \begin{split}
        mx'' + m \omega_0^2 x &= -k(x-y)
        \\
        m y'' + m \omega_0^2 y &= -k(y-x)
    \end{split}
\end{equation}
for constants $m$, $\omega_0$, and $k$.
Let $q=y+x$ and $p=y-x$.
Then, 
\begin{equation}
    \begin{split}
        mq'' + m\omega_0^2 q &= 0
        \\
        mp'' + (m\omega_0^2+2k)p &=0
    \end{split}
\end{equation}
The resulting equations both describe the motion of ideal harmonic oscillators, but have different physical interpretations.
The dynamics of $q$ may be interpreted as the movement of the center of mass of the system.
If the two masses were initially displaced the same distance in the same direction, we would have $p=0$.
In this situation, the two masses oscillate as one unit, i.e., the distance between them stays the same.
On the other hand, the dynamics of $p$ capture anti-phase oscillations. 
In a situation where the two masses are initially displaced in equal but opposite directions, we would have $q=0$.
The two masses again oscillate as one unit, but swing in opposite directions.
These two solutions are sometimes referred to as the system's \textit{normal modes}, and linear combinations of these solutions can describe dynamics with non-symmetric initial conditions.
See (\cite{waltman2004second}, \S 1.12) for a detailed analysis of coupled pendulums.

We interpret the structure of the perturbed $Q$-functions of the identical joint system as normal modes for coupled stochastic oscillators, and argue that this structure provides physical intuition for `synchronization' in the case of stochastic oscillators.
In the case when $\tau=0$ and $\beta\neq 0$, the observable $\mathcal{Q}_+^*(\mbx,\mby)$ captures dynamics analogous to the ``center of mass'' mode, and $\mathcal{Q}_-^*(\mbx,\mby)$ captures dynamics analogous to the ``anti-phase mode''.

 It is instructive to consider the form of the constants $\alpha$ and $\beta$.
Notice that 
\begin{equation}
\begin{split}
    \mathcal{L}_\kappa^\dagger Q_{1y}^*(\mbx,\mby) &= \mbh(\mbx,\mby) \partial_x Q_{1y}^*(\mbx,\mby) + \mbh(\mby,\mbx) \partial_y Q_{1y}^*(\mbx,\mby)
    \\
    &= \left(\nabla Q_{1y}^*(\mbx,\mby)\right)^T\begin{bmatrix}
        \mbh(\mbx,\mby)
        \\
        \mbh(\mby,\mbx)
    \end{bmatrix}
\end{split}
\end{equation}
The gradient of the unperturbed $Q$-function, $Q_{1y}^*(\mbx,\mby)$, is analogous to the (infinitesimal) phase response curve (iPRC) from classical \textit{deterministic} theory \cite{perez2023universal}.
Therefore, the integrals
\begin{equation}
\begin{split}
    \beta &= \iint_{\mathbb{R}^{2n}} d\mbx d\mby \; P_{1x}(\mbx,\mby)(\mathcal{L}_{\kappa}^\dagger Q_{1y}^*(\mbx,\mby)) 
    \\
    &= \iint_{\mathbb{R}^{2n}} d\mbx d\mby \; P_{1x}(\mbx,\mby)\left(\nabla Q_{1y}^*(\mbx,\mby)\right)^T\begin{bmatrix}
        \mbh(\mbx,\mby)
        \\
        \mbh(\mby,\mbx)
    \end{bmatrix}
    \\
    \\
    \alpha &= \iint_{\mathbb{R}^{2n}} d\mbx d\mby \; P_{1x}(\mbx,\mby)(\mathcal{L}_{\kappa}^\dagger Q_{1x}^*(\mbx,\mby))    
    \\
    &= \iint_{\mathbb{R}^{2n}} d\mbx d\mby \; P_{1x}(\mbx,\mby)\left(\nabla Q_{1x}^*(\mbx,\mby)\right)^T\begin{bmatrix}
        \mbh(\mbx,\mby)
        \\
        \mbh(\mby,\mbx)
    \end{bmatrix}
\end{split}
\end{equation}
bear a strong resemblance to the deterministic interaction function (or $H$ function), an integral over the deterministic limit cycle of the inner product of the iPRC with the perturbation.
Interaction functions play a critical role in determining the bifurcation associated with the onset of synchronization in the deterministic setting; 
see, for instance, \cite{ermentrout2010mathematical,izhikevich2007dynamical, park2021high}.

Theorem \ref{theorem:Q function perturbation} gives the leading order behavior in perturbations $\kappa,\tau$ of the $Q$-function eigenvalues, assuming that the perturbations scale proportionally to each other.
Generically, one expects the repeated unperturbed eigenvalues to split.
However, in certain cases, the repeated eigenvalues are maintained under simultaneous perturbation of $\kappa$ and $\tau$.
The following Corollary establishes conditions under which the repeated eigenvalue structure is maintained.
This result gives an analytic expression for the linear synchronization boundary, i.e., the ``tip" of an Arnold tongue, for coupled stochastic oscillators (see \S\ref{ssec:Arnold-tongues} for a visual representation).

\begin{corollary}\label{corollary:splitting discriminant}
    Assume the coupled oscillator criteria \ref{assumptions:coupled oscillator criteria} hold for a system of symmetrically coupled oscillators of the form \eqref{eq:SDE}.
Let $\Upsilon(\kappa,\tau) = \frac{\operatorname{trace}(\mathcal{M})}{2}$.
Define the \emph{splitting discriminant}
\begin{equation}\label{eq:splitting discriminant}
    \mathcal{D}(\kappa,\tau) = \sqrt{\operatorname{trace}(\mathcal{M})^2 - 4\det(M)}
\end{equation}
where $\mathcal{M}$ is the matrix from Theorem \ref{theorem:Q function perturbation}.
Assume that
\begin{equation}
    \beta = \iint_{\mathbb{R}^{2n}} d\mbx d\mby \; P_{1x}(\mbx,\mby)(\mathcal{L}_{\kappa}^\dagger Q_{1y}^*(\mbx,\mby)) \neq 0
\end{equation}
Fix $(\kappa^*,\tau^*)\in \mathcal{K}\times \mathcal{T}$ such that
\begin{widetext}
\begin{equation}\label{eq:IFT}
    \kappa^* = \mathfrak{K}(\tau^*) = \pm i |\tau^*|\frac{\sqrt{\left(\iint_{\mathbb{R}^{2n}} d\mbx d\mby \; P_{1x}(\mbx,\mby)[ \mathcal{L}_\tau^\dagger Q_{1x}^*(\mbx,\mby)] - P_{1y}(\mbx,\mby)[ \mathcal{L}_\tau^\dagger Q_{1y}^*(\mbx,\mby)]\right)^2}}{2\beta}
\end{equation}
\end{widetext}
Then, $\mathcal{D}(\mathfrak{K}(\tau^*),\tau^*)=0$, and under the re-parameterization $(\kappa^*,\tau^*) \to (\epsilon\kappa^*,\epsilon\tau^*)$, the SKO eigenvalues of \eqref{eq:SDE} are repeated with algebraic multiplicity 2 to leading order in $\epsilon$
\begin{equation}
    \lambda_+(\epsilon\kappa,\epsilon\tau) = \lambda_-(\epsilon\kappa,\epsilon\tau) = \lambda_1 + \epsilon\Upsilon(\kappa^*,\tau^*) + \mathcal{O}(\epsilon^2)
\end{equation}
\end{corollary}

\begin{proof}
    By Theorem \ref{theorem:Q function perturbation}, the leading order update to the eigenvalues are given by the eigenvalues of the matrix $\mathcal{M}$.
    The eigenvalues of $\mathcal{M}$ are the roots of the characteristic polynomial, which are
    \begin{equation}\label{eq:eigenvalue_form}
        \lambda_{1\pm}(\kappa,\tau) = \Upsilon(\kappa,\tau) \pm \frac{1}{2}\mathcal{D}(\kappa,\tau)
    \end{equation}
    For ease of notation, define the constants $\alpha$ and $\beta$ as in Corollary \ref{corollary:id osci}, and the constants $a_1$, $b_1$, $c_1$, and $d_1$ as 
    \begin{equation}
        \begin{split}
            a_1 &= \iint_{\mathbb{R}^{2n}} d\mbx d\mby \; P_{1x}(\mbx,\mby)( \mathcal{L}_\tau^\dagger Q_{1x}^*(\mbx,\mby))
            \\
            b_1 &= \iint_{\mathbb{R}^{2n}} d\mbx d\mby \; P_{1x}(\mbx,\mby)( \mathcal{L}_\tau^\dagger Q_{1y}^*(\mbx,\mby))
            \\
            c_1 &= \iint_{\mathbb{R}^{2n}} d\mbx d\mby \; P_{1y}(\mbx,\mby)( \mathcal{L}_\tau^\dagger Q_{1x}^*(\mbx,\mby))
            \\
            d_1 &= \iint_{\mathbb{R}^{2n}} d\mbx d\mby \; P_{1y}(\mbx,\mby)( \mathcal{L}_\tau^\dagger Q_{1y}^*(\mbx,\mby))
        \end{split}
    \end{equation}
    so that the matrix $\mathcal{M}(\kappa,\tau)$ is expressed as
    \begin{equation}
        \begin{split}
            \mathcal{M}(\kappa,\tau) = \begin{bmatrix}
        \tau a_1 + \kappa \alpha & \tau b_1 + \kappa \beta
        \\
        \tau c_1 + \kappa \beta & \tau d_1 + \kappa \alpha
    \end{bmatrix}
        \end{split}
    \end{equation}
    We now show that $c_1=b_1=0$.
    Recalling \eqref{eq:SDE} we write $\mbu_i$ for the detuning perturbation of the vector field governing the $i^\text{th}$ oscillator.    Note that
    \begin{equation}
        \begin{split}
            c_1 &= \iint_{\mathbb{R}^{2n}} d\mbx d\mby \; P_{1y}(\mbx,\mby)[ \mathcal{L}_\tau^\dagger Q_{1x}^*(\mbx,\mby)]
            \\
            &= \iint_{\mathbb{R}^{2n}} d\mbx d\mby \; P_1(\mby)P_0(\mbx)[ \mbu_1(\mbx)\partial_x + \mbu_2(\mby)\partial_y]\{ Q_{1}^*(\mbx)\mathbb{1}(\mby)\}
            \\
            &= \iint_{\mathbb{R}^{2n}} d\mbx d\mby \; P_1(\mby)P_0(\mbx)[ \mbu_1(\mbx)\partial_x \{ Q_{1}^*(\mbx)\}]\mathbb{1}(\mby)
            \\
            &= \left(\int_{\mathbb{R}^n} d\mby \; P_1(\mby)\mathbb{1}(\mby)\right) \left(\int_{\mathbb{R}^n} d\mbx \; P_0(\mbx)[ \mbu_1(\mbx)\partial_x \{ Q_{1}^*(\mbx)\}]\right)
            \\
            &= 0 \times  \left(\int_{\mathbb{R}^n} d\mbx \; P_0(\mbx)[ \mbu_1(\mbx)\partial_x \{ Q_{1}^*(\mbx)\}]\right)
            \\
            &= 0
        \end{split}
    \end{equation}
    where the integral with respect to $\mby$ vanishes due to the biorthogonality of the eigenfunctions of the isolated oscillator in $\mby$.
    A similar argument shows that $b_1=0$.
    Then, simple algebra shows that
    \begin{equation}
        \mathcal{D}(\kappa,\tau) = \sqrt{\kappa^2(2\beta)^2 + \tau^2(d_1-a_1)^2}
    \end{equation}
    By assumption, $\beta\neq 0$.
    Therefore, provided that
    \begin{equation}
        \kappa^* = \mathfrak{K}(\tau^*) = \pm i |\tau^*| \frac{\sqrt{(a_1-d_1)^2}}{2\beta}
    \end{equation}
    it follows that $\mathcal{D}(\kappa^*,\tau^*)=0$.
    Moreover, it is clear that the eigenvalues of the matrix $\mathcal{M}$ are given by
    \begin{equation}
        \lambda_{1\pm}(\kappa,\tau) = \frac{\text{trace}(\mathcal{M}) \pm \mathcal{D}(\kappa,\tau)}{2}
    \end{equation}
    Consequently,
    \begin{equation}
        \lambda_{1\pm}(\kappa^*,\tau^*) = \frac{\text{trace}(\mathcal{M})}{2} = \Upsilon(\kappa^*,\tau^*)
    \end{equation}
    which completes the proof.
\end{proof}

We note that if the quantity $(a_1-d_1)$ is purely imaginary, then purely real perturbations $\kappa^*$ and $\tau^*$ exist which satisfy $\kappa^* = \mathfrak{K}(\tau^*)$.
Consequently, the splitting discriminant vanishes for physical (real) perturbations when the action of $\mathcal{L}_\tau^\dagger$ results in a purely imaginary perturbation, i.e., influences \textit{only} the frequency of the isolated oscillators.

We remark that it may not always be possible to choose parameters $\kappa^*,\tau^*\neq0$ such that $\mathcal{D}(\kappa^*,\tau^*)=0$.
In such cases, an eigenvalue bifurcation in the sense of $Q$-synchronization (recall definition \ref{definition}) does not occur to leading order in $\kappa$ and $\tau$.
See Appendix \ref{appendix:no bifurcation} for a specific example.
However, in cases where such parameter values do exist, Corollary \ref{corollary:splitting discriminant} proved that the condition $\mathcal{D}(\kappa^*,\tau^*)=0$ determines a functional relationship between the coupling strength and frequency parameter, $\kappa^* = \mathfrak{K}(\tau^*)$.
We call such parameter value pairs $(\kappa^*,\tau^*)$ KT points.
In \S\ref{section:examples}, we find KT points for three example systems, and show that the corresponding functional relationship determines a synchronization regime analogous to $1:1$ mode-locking Arnold tongues for coupled deterministic oscillators.

Corollary \ref{corollary:splitting discriminant} does not address the leading order update to the eigenfunctions.
We find that in many cases, the geometric multiplicity of a repeated eigenvalue at a KT point is unity.
Therefore, a higher order analysis may be necessary to elucidate the structure of the perturbed eigenfunctions, which is beyond the scope of this paper.

\subsection{Power Spectra and Cross-Spectral Density}
\label{ssec:PowerSpectra_and_CrossSpectralDensity}

In this section, we establish that the splitting of the $Q$-function eigenvalues determines the form of the cross spectra and cross-spectral density of the oscillating units in $Q$-function coordinates.
The form of the power spectra and cross-spectral density becomes particularly simple at the KT point.
In the following, the KT point is understood to exist at certain parameter values $\kappa^*$ and $\tau^*$.
We use the language ``above the KT point'' to refer to parameter value pairs $(\kappa,\tau^*)$ for $\kappa>\kappa^*$  and ``below the KT point'' to refer to parameter value pairs $(\kappa,\tau^*)$ for $\kappa<\kappa^*$.
As we shall see below, it is natural to view parameter combinations above the KT point as being within the synchronization region, and parameter combinations below the KT point as being outside the synchronization region.

Recall that in $Q$-function coordinates, the power spectrum of a stochastic oscillator is a Lorentzian whose exact form is given in terms of the eigenvalues of the $Q$-function, cf.~\eqref{eq:Q function power spectrum}. 
Recalling the discussion in \S\ref{para:general_assumptions}, a system of two robustly oscillatory \textit{uncoupled} stochastic oscillators has two $Q$-functions, $Q_{1x}^*(\mbx,\mby)$ and $Q_{1y}^*(\mbx,\mby)$ (one contribution from each oscillator).
As the coupling strength is increased, these eigenfunctions become $\mathcal{Q}_+^*(\mbx,\mby)$ and $\mathcal{Q}_-^*(\mbx,\mby)$, respectively.
Their corresponding eigenvalues, $\lambda_+ = \mu_+ + i\omega_+$ and $\lambda_- = \mu_- + i\omega_-$, are generically perturbed from their $\kappa=0$ counterpart $\lambda_1$ (repeated with algebraic multiplicity 2).
Even when $\kappa\neq 0$, the power spectrum, in the coordinates of each respective $Q$-function, is still a pure Lorentzian
\begin{equation}\label{eq:power spectra}
    S_+(\nu) = \frac{2|\mu_+|}{\mu_+^2 + (\nu-\omega_+)^2}, \quad S_-(\nu) = \frac{2|\mu_-|}{\mu_-^2 + (\nu-\omega_-)^2}
\end{equation}

We now discuss the shape of the power spectra in the special case that a system of coupled stochastic oscillators admits a KT point.
At the KT point, the eigenvalues are identical with $\mu_+=\mu_-$ and $\omega_+=\omega_-$.
Consequently, the power spectra of the two oscillators become identical.
In the important case that the splitting discriminant $\mathcal{D}(\kappa,\tau)\neq0$ is purely real, above the KT point the real parts of the eigenvalues are distinct ($\mu_+\neq \mu_-$), but the imaginary parts are identical ($\omega_+=\omega_-$).
Accordingly, both  power spectra  peak at the same frequency.
In contrast, in the original coordinates, the power spectra of each oscillator have peaks which generically only asymptotically approach each other as coupling strength increases.

Below the KT point, we make no assumptions about the form of the $Q$-function eigenvalues, save that they are non-identical.
In the case of two weakly coupled robustly oscillatory stochastic oscillators, one generically expects that the joint system is not robustly oscillatory.
Consequently, each oscillator has a prominently peaked power spectrum.
That is, the system behaves as two distinct stochastic oscillators.  
As the coupling increases, the power spectra ``collide'' and become identical at the KT point.
Above the bifurcation, the coupled system is robustly oscillatory provided that $\mathcal{D}(\kappa,\tau)\neq 0$ is not purely imaginary, i.e., there is a unique low-lying eigenmode which dominates system dynamics at intermediate-long times.
Only one of the two power spectra remains prominent, analogous to the center-of-mass mode of a coupled harmonic oscillator system \eqref{eq:normal_mode_system}.
The other flattens and becomes less prominent because the corresponding $Q$-function becomes a high-order mode, analogous to the anti-phase mode of a harmonic oscillator system.
We show examples of power spectra for several concrete model systems in \S\ref{ssec:Examples-PowerSpectra}.

The cross spectra between two eigenfunctions, $\mathcal{Q}_+^*$ and $\mathcal{Q}_-^*$, is given by \cite{perez2023universal}
\begin{equation}\label{eq:cross spectra}
    S_{+,-}(\nu) = -\langle \mathcal{Q}_+^* \mathcal{Q}_- \rangle \left( \frac{1}{\lambda_+ - i\nu} + \frac{1}{\lambda_-^* + i\nu}\right) 
\end{equation}
Before analyzing the behavior of \eqref{eq:cross spectra}, we consider the quantity $\langle \mathcal{Q}_+^* \mathcal{Q}_- \rangle$.
Recall that the backward eigenfunctions are normalized so that they have unit variance.
Consequently, at the KT point, the eigenfunctions are identical, but only up to a rotation. (c.f., the discussion of phase ambiguity following the proof of Theorem \ref{theorem:Q function perturbation}).
If one disregards the rotational ambiguity, the quantity $\langle \mathcal{Q}_+^* \mathcal{Q}_- \rangle$ will in general be complex, even at the KT point.
To that end, without loss of generality, we make the gauge transformation, $\mathcal{Q}^*_- \to \mathcal{Q}_-^*e^{i\alpha}$ where 
\begin{equation}\label{eq:minimization solution}
    \alpha = \frac{-i}{2}\log\frac{\iint_{\mathbb{R}^{2n}} d\mbx d\mby \; \mathcal{Q}_+^*(\mbx,\mby) \mathcal{Q}_-(\mbx,\mby)  P_0(\mbx,\mby)  }{\iint_{\mathbb{R}^{2n}} d\mbx d\mby \; \mathcal{Q}_+(\mbx,\mby) \mathcal{Q}_-^*(\mbx,\mby)  P_0(\mbx,\mby)   }
\end{equation}
We choose the branch of the logarithm in such a way that $\alpha\in [0,2\pi)$.
At the KT point one has $\mathcal{Q}_-e^{i\alpha}=\mathcal{Q}_+$, by construction, and thus
\begin{equation}\label{eq:rotation normalization bifurcation point}
    \langle \mathcal{Q}_+^* \mathcal{Q}_-e^{i\alpha} \rangle  = 
    \iint_{\mathbb{R}^{2n}} d\mbx d\mby \; \left|\mathcal{Q}_+^*(\mbx,\mby)\right|^2 P_0(\mbx,\mby)   = 1
\end{equation}
For a detailed derivation of \eqref{eq:minimization solution}, please see Appendix \ref{appendix:gauge_transformation}. 
That such a normalization procedure is possible demonstrates that the eigenvalue dynamics, i.e., the $Q$-synchronization bifurcation, captures the intrinsic behavior of the cross spectra of coupled stochastic oscillators.

The gauge transformation \eqref{eq:minimization solution} is  analogous to the well known ambiguity in  the asymptotic phase of a deterministic oscillator, which is only defined up to an additive constant.
A common method to eliminate this ambiguity when comparing several oscillators is to fix the phase of one oscillator as a reference, and relabel the phase of the remaining oscillators in relative to the reference oscillator.
Our gauge transformation treats $Q_{1x}$ as the `reference', and rotates the remaining eigenfunction to eliminate the ambiguity in the stochastic phase labels.

Returning to the form of the cross spectra, we find that under the gauge transformation \eqref{eq:minimization solution}, the shape of the cross-spectral density \eqref{eq:cross spectra} is completely captured by the low-lying eigenvalues of the SKO.  
There are a variety of possibilities, depending on the system in question.
For concreteness, here we consider the case when the splitting discriminant $\mathcal{D}(\kappa,\tau)$ is real and nonzero.
In this case the joint system has eigenvalues with identical real parts but nonidentical imaginary parts before the KT point, and nonidentical real parts but identical imaginary parts after the KT point.
Section \S\ref{section:examples} contains specific examples of systems which exhibit this behavior.
Let 
\begin{equation}
    S^\dagger_{+,-}(\nu) = - \langle \mathcal{Q}_+^* \mathcal{Q}_-e^{i\alpha} \rangle \left( \frac{1}{\lambda_+ - i\nu} + \frac{1}{\lambda_-^* + i\nu}\right) 
\end{equation}
denote the cross-spectral density under the gauge transformation.
Recall that $\lambda_+ = \mu_+ + i \omega_+$, $\lambda_- = \mu_- + i \omega_-$, and define $\nu^* = \frac{\omega_++\omega_-}{2}$.
We list several easily verifiable statements concerning the behavior of the imaginary part of $S^\dagger_{+,-}(\nu)$:
\begin{itemize}
    \item Below the KT point, $\text{Im}(S^\dagger)$ has a local extremum at $\nu^*$, and is an even function about $\nu^*$.
    \item At the KT point, $S^\dagger$ is purely real and identical to the (common) power spectrum, i.e., is the Lorentzian \eqref{eq:power spectra}
    \item Above the KT point, $\text{Im}(S^\dagger)$ equals zero at $\nu^*$, and is an odd function about $\nu^*$.
\end{itemize}
Consequently, the cross-spectral density undergoes a qualitative change at the KT point.

\subsection{\textit{Q}-Synchronization}\label{subsec: q synch}

The results of \S \ref{subsection:results_Q_functions} indicate that when the splitting discriminant $\mathcal{D}(\kappa,\tau)\neq 0$, a qualitative change in the eigenvalues of the SKO corresponding to the $Q$-function occurs, i.e., the eigenvalues split at the KT point.
The splitting eigenvalues also reflect a qualitative change in the power spectra and cross-spectral density in $Q$-function coordinates of the oscillators.
These results motivate the novel Definition \ref{definition} of synchronization for coupled stochastic oscillators as stated in the introduction:

\begin{quotation}
    ``Symmetrically coupled oscillators of the form \eqref{eq:SDE} exhibit \textit{$Q$-synchronization} if the leading nontrivial complex conjugate eigenvalues of the SKO undergo a qualitative change, i.e., a repeated pair of eigenvalues is created or destroyed, as system parameters are varied.''
\end{quotation}

For fixed $\tau$, $Q$-synchronization means that the joint system acts as a single high-dimensional oscillator (for sufficiently strong coupling) provided that $\mathcal{D}(\kappa,\tau)$ is not purely imaginary.
In the special case of identical $(\tau=0$, $\mbg_1=\mbg_2)$ symmetrically coupled oscillators, the $Q$-function eigenvalue of the joint uncoupled system is a repeated eigenvalue.
Consequently, when $\kappa=0$, the joint system is not robustly oscillatory due to lack of a spectral gap, i.e., it cannot be described as a single stochastic oscillator.
However, when $\kappa \neq 0$, and the splitting discriminant $\mathcal{D}(\kappa,\tau)$ is not purely imaginary, then the real parts of the $Q$-function eigenvalues split.
Therefore, the coupling introduces a spectral gap; there is a unique dominant eigenmode and the joint system is considered robustly oscillatory in the sense of Assumptions \ref{robustly_osc_criteria}.

In the special case that $\mathcal{D}(\kappa,\tau)$ is purely real, the coupling introduces a spectral gap while maintaining the identical mean frequencies of each oscillator (see equation \eqref{eq:eigenvalue_form}).
Such behavior is identical to the deterministic case, where any small amount of coupling between identical symmetrically coupled oscillators is expected to result in synchronization, i.e., a high-dimensional stable limit-cycle solution where the constituents have identical periods.

In the following section, we study three examples of symmetrically coupled stochastic oscillators which exhibit $Q$-synchronization in the sense of Definition \ref{definition}.
For each system, we analytically compute a relationship $\kappa^*=\mathfrak{K}(\tau^*)$ such that $\mathcal{D}(\mathfrak{K}(\tau^*),\tau^*)=0$, and use this condition to construct synchronization regimes analogous to deterministic 1:1 mode-locking regions (Arnold tongues).


\section{Examples}\label{section:examples}

\subsection{Model systems}\label{subsection:model systems}

Throughout this section, we apply our theory to three model systems: a 4D linear model which consists of two coupled planar Ornstein-Uhlenbeck processes, a stochastic version of the 2D coupled Kuramoto model considered  in the introduction  \eqref{eq:intro_kuramoto}, and a 9D discrete-state system comprising two coupled 3D discrete state oscillatory subsystems. 
We remark that each of the continuous-state models, in the absence of coupling, meet the `robustly oscillatory' criteria outlined in \S \ref{subsection:Q_phase_reduction}.
The discrete-state model does not meet the robustly oscillatory criteria because its quality factor is less than unity (see \S\ref{subsection: quality factor}).
Nevertheless, as we shown below, our framework still elucidates the synchronization behavior of the discrete-state system.
Figure \ref{fig:model_systems} shows sample trajectories, eigenvalue spectra, and power spectra for each uncoupled system.

\subsubsection{4D linear model}\label{subsubsection:4D linear model}

Our first example is a system of two diffusively coupled, rotating Ornstein-Uhlenbeck processes
 \begin{equation}\label{eq:4D linear model}
    \begin{split}
        x_1' &= -\eta x_1 +(\omega + \tau) x_2 + \kappa (y_1-x_1) + \sqrt{2D}\xi_1(t)
        \\
        x_2' &= -(\omega + \tau) x_1 - \eta x_2 + \kappa (y_2-x_2) + \sqrt{2D}\xi_2(t)
        \\
        y_1' &= -\eta y_1 +\omega y_2 + \kappa (x_1-y_1) + \sqrt{2D}\xi_3(t)
        \\
        y_2' &= -\omega y_1 - \eta y_2 + \kappa (x_2-y_2) + \sqrt{2D}\xi_4(t)
    \end{split}
\end{equation}
The parameter $\kappa \geq 0$ characterizes the interaction strength of the oscillators, $D$ governs the noise intensity, $\tau$ governs the frequency of the first oscillator, and $\eta$ governs the coherence of the oscillations.
Each oscillator is driven by independent delta-correlated Gaussian white noise $\langle \xi_i(t), \xi_j(s)\rangle = \delta(t-s)\delta_{i,j}$.
In all cases, we set $\eta=0.1$, $\omega=2$, and $D=0.1$.
Unless otherwise specified, all numerical simulations are performed with the Euler-Maruyama method \cite{kloeden1992stochastic} with a timestep of $0.002$.

Let $\mbz = [x_1, \;x_2, \; y_1,\; y_2]^T$.
Associated with the underlying SDEs \eqref{eq:4D linear model} is a forward equation 
\begin{equation}
    \mathcal{L}[P_\lambda(\mbz)] = \lambda P_\lambda(\mbz)
\end{equation}
and a backward equation
\begin{equation}
    \mathcal{L}^\dagger[Q^*_\lambda(\mbz)] = \lambda Q^*_\lambda(\mbz)
\end{equation}
defined on $\mathbb{R}^4$.
As boundary condition requirements, we specify that the forward eigenfunctions satisfy $P_\lambda(\mbz) \in \mathcal{C}^2(\R^4)\cap L_1(\R^4)$, and that $\lim_{\|\mbz\| \to \infty} P_\lambda(\mbz) = 0 $.
For any compact subset $\Omega^4 \subset \mathbb{R}^4$, the backward eigenfunctions satisfy $Q^*_\lambda \in \mathcal{C}^2(\R^4)\cap L_\infty(\Omega^4)$.

\subsubsection{2D ring model}\label{subsubsection:2D Kuramoto model}

Our second example consists of interacting noisy Kuramoto oscillators of the form
\begin{equation}\label{eq:2D Kuramoto model system}
    \begin{split}
        x' &= \omega + \tau + \kappa \sin(y-x) + \sqrt{2D}\xi_1(t)
        \\
        y' &= \omega + \kappa \sin(x-y) + \sqrt{2D}\xi_2(t)
    \end{split}
\end{equation}
where the system evolves on the torus, $x,y\in [0,2\pi)$.
In all cases, we set $\omega=2$ and $D=0.1$.
Unless otherwise specified, all numerical simulations are performed with the Euler-Maruyama method \cite{kloeden1992stochastic} with a timestep of $0.002$.

Associated with the underlying SDEs \eqref{eq:2D Kuramoto model system} is a forward equation 
\begin{equation}
    \mathcal{L}[P_\lambda(x,y)] = \lambda P_\lambda(x,y)
\end{equation}
and a backward equation
\begin{equation}
    \mathcal{L}^\dagger[Q^*_\lambda(x,y)] = \lambda Q^*_\lambda(x,y)
\end{equation}
defined on $\mathbb{T}^2 = [0,2\pi)\times [0,2\pi)$.
As boundary condition requirements, we specify that the forward eigenfunctions satisfy $P_\lambda(\mbz) \in \mathcal{C}^2(\mathbb{T}^2)\cap L_1(\mathbb{T}^2)$, and that $P_\lambda(0,0) = P_\lambda(0,2\pi) = P_\lambda(2\pi,0) = P_\lambda(2\pi,2\pi)$.
The backward eigenfunctions satisfy $Q^*_\lambda \in \mathcal{C}^2(\mathbb{T}^2)\cap L_\infty(\mathbb{T}^2)$, and that $Q^*_\lambda(0,0) = Q^*_\lambda(0,2\pi) = Q^*_\lambda(2\pi,0) = Q^*_\lambda(2\pi,2\pi)$.

\subsubsection{9D discrete-state system}\label{subsubsection:9D discrete-state model}

Our final example serves to demonstrate that our theory can be successfully applied to discrete-state systems.
The backward equation corresponding to a discrete-state system with transition-rate matrix $\mathcal{E}$ is given by
\begin{equation}
    \frac{df}{dt} = \mathcal{E} f(t)
\end{equation}
where $f(t)$ represents the expected value of a test function which may depend on the state of the system.
The eigenvalue-eigenfunction pairs of the backward operator are thus the eigenvalue-right eigenvector pairs of the matrix $\mathcal{E}$.
The corresponding forward eigenfunctions are left eigenvectors of $\mathcal{E}$.

We consider a specific example of two coupled discrete-state oscillators;
each individual oscillator is a three-state system with transition-rate matrices $A$ and $B$, respectively
\begin{equation}
\begin{split}
    A &= \begin{bmatrix}
-(\omega+\tau) & 0 & \omega+\tau \\
\omega+\tau & -(\omega+\tau) & 0 \\
0 & \omega+\tau & -(\omega+\tau)
\end{bmatrix} 
\\
B &= \begin{bmatrix}
-\omega & 0 & \omega \\
\omega & -\omega & 0 \\
0 & \omega & -\omega
\end{bmatrix}
\end{split}
\end{equation}
The states of each system may be represented as $X_A = [A_1, A_2, A_3]^T$ and $X_B = [B_1, B_2, B_3]^T$, where $A_i,B_i\in \{0,1\}$.
The parameter $\omega > 0$ describes the rate of transition between states, with cycle $A_1\to A_2 \to A_3 \to A_1$ and so forth.
The perturbation parameter, $\tau$, influences the transition rate, and thus the `frequency' of oscillation of the first discrete-state system.

Unless otherwise specified, all numerical simulations are performed with the Gillespie algorithm \cite{gillespie1976general} up to a finite time $T=T_{\text{max}}$.
The resulting time series is interpolated via next neighbor onto a evenly spaced time grid with spacing $0.0002$.
We remark that although the discrete trajectories do not qualitatively ``behave" as an oscillator \textit{in original coordinates} (since the uncoupled systems admit power spectra which peak at 0 frequency, see Figure \ref{fig:model_systems}) the power spectrum of the transformed $Q$-function coordinates nevertheless has a Lorentzian power spectrum peaking at nonzero frequency (see Figure \ref{fig:power spectra}).

When coupled together, the joint system is 9D, with transition-rate matrix given by
\begin{widetext}
\begin{equation}\label{eq:discrete_coupled_system}
\scalebox{.85}{
    $C = \begin{bmatrix}
-2\omega-\tau & 0 & \kappa + \omega+\tau & 0 & 0 & 0 & \kappa + \omega & 0 & 0 \\
\omega+\tau & -2\omega-\tau & 0 & 0 & 0 & 0 & 0 & -\kappa + \omega & 0 \\
0 & -\kappa + \omega+\tau & -2\omega-\tau & 0 & 0 & 0 & 0 & 0 & \omega \\
\omega & 0 & 0 & -2\omega-\tau & 0 & -\kappa + \omega+\tau & 0 & 0 & 0 \\
0 & \kappa + \omega & 0 & \kappa + \omega+\tau & -2\omega-\tau & 0 & 0 & 0 & 0 \\
0 & 0 & -\kappa + \omega & 0 & \omega+\tau & -2\omega-\tau & 0 & 0 & 0 \\
0 & 0 & 0 & -\kappa + \omega & 0 & 0 & -2\omega-\tau & 0 & \omega+\tau \\
0 & 0 & 0 & 0 & \omega & 0 & -\kappa + \omega+\tau & -2\omega-\tau & 0 \\
0 & 0 & 0 & 0 & 0 & \kappa + \omega & 0 & \kappa + \omega+\tau & -2\omega-\tau
\end{bmatrix}$
}
\end{equation}
\end{widetext}
where the joint states are given by $X_C=[(A_i,B_j)]^T$, with $i=\{1,2,3,1,2,3,1,2,3\}$ and $j=\{1,1,1,2,2,2,3,3,3\}$.
We employ ``boxcar'' coupling with coupling strength $\kappa$ to push system states toward the diagonal, i.e., $(A_1,B_1),(A_2,B_2), (A_3,B_3)$ are favored for large $\kappa$.
We impose the restriction that $\kappa<\min\{\omega,\omega+\tau\}$ so that state transition rates do not become negative.

\subsubsection{Quality factors}\label{subsection: quality factor}

For the parameter values indicated above, the quality factors for each of the isolated oscillators are: $|\omega_1/\mu_1|=20$ for the planar OU and 1D Kuramoto oscillators, and $|\omega_1/\mu_1| = \sqrt{3}/3$ for the 3D discrete-state subsystem (see \S\ref{subsection: eval bif} for more detailed information about the leading eigenvalues of each system).
Therefore, the discrete-state system does not meet the robustly oscillatory criteria because it has a quality factor less than unity. 
Nevertheless, as we will show in the following sections, our $Q$-function framework may still be applied to the discrete-state system in a manner which elucidates its synchronization behavior.

\begin{figure*}[ht]
    \centering
    
    \includegraphics[scale=.3]{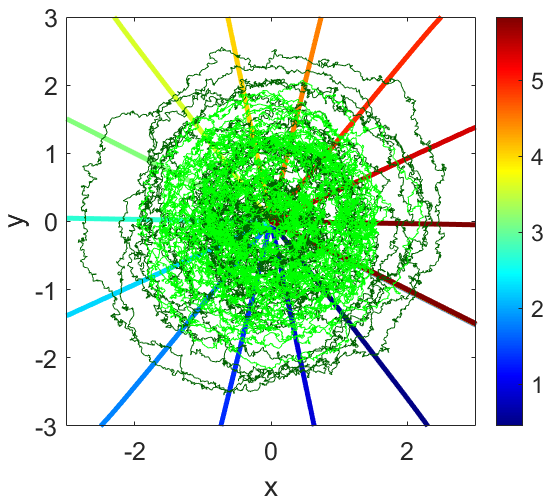}
    \includegraphics[scale=.3]{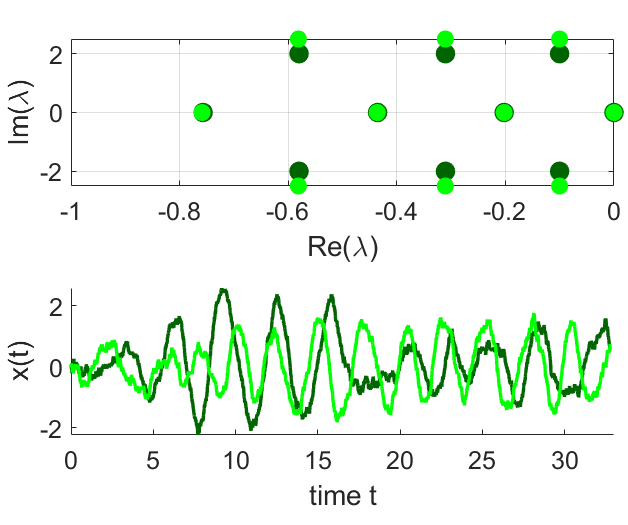}
    \includegraphics[scale=.3]{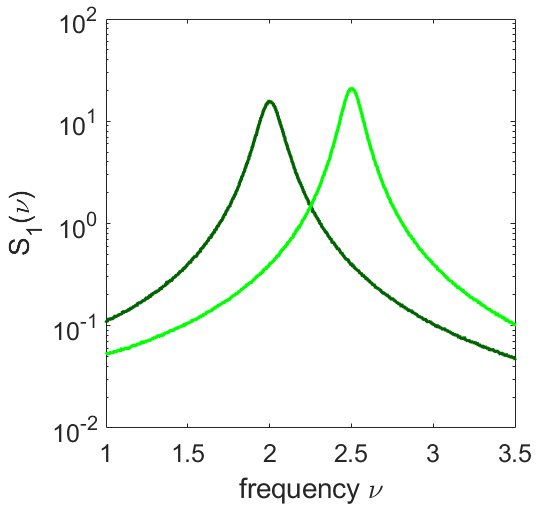}
    
    \includegraphics[scale=.3]{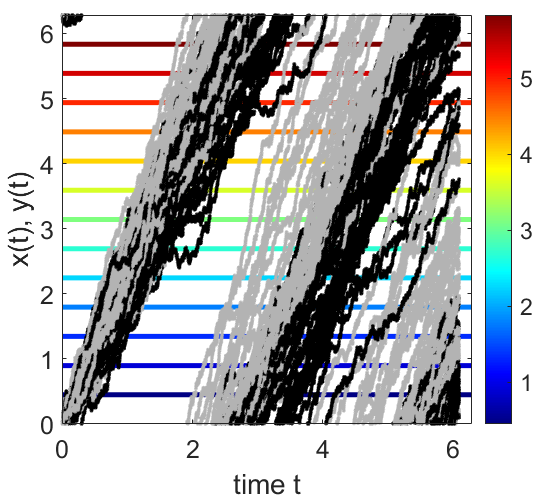}
    \includegraphics[scale=.3]{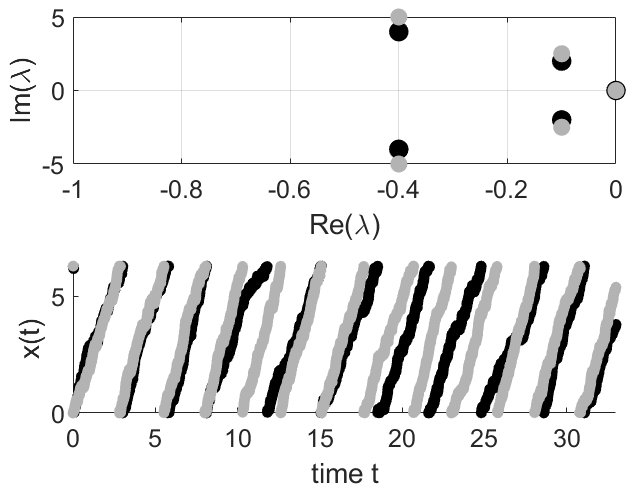}
    \includegraphics[scale=.3]{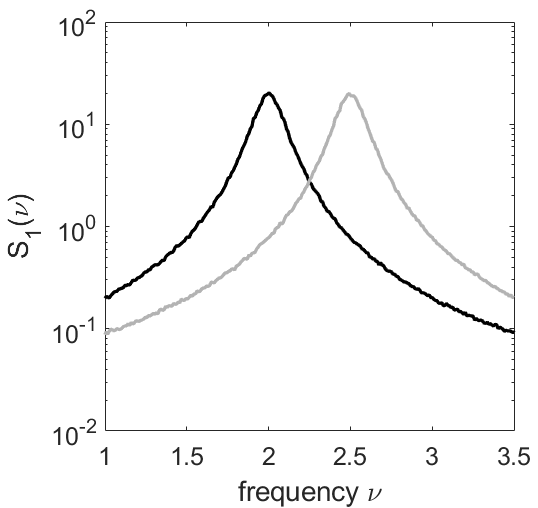}

    \includegraphics[scale=.3]{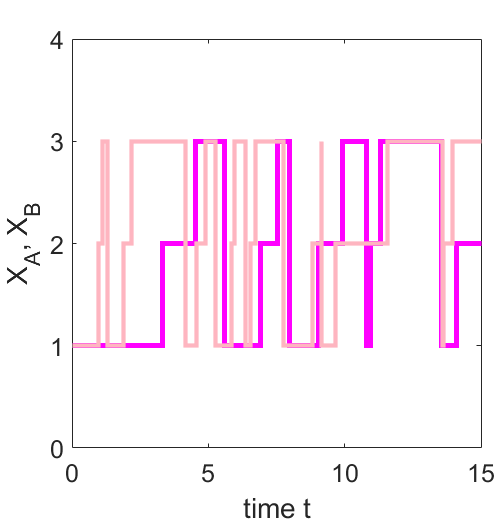}
    \includegraphics[scale=.3]{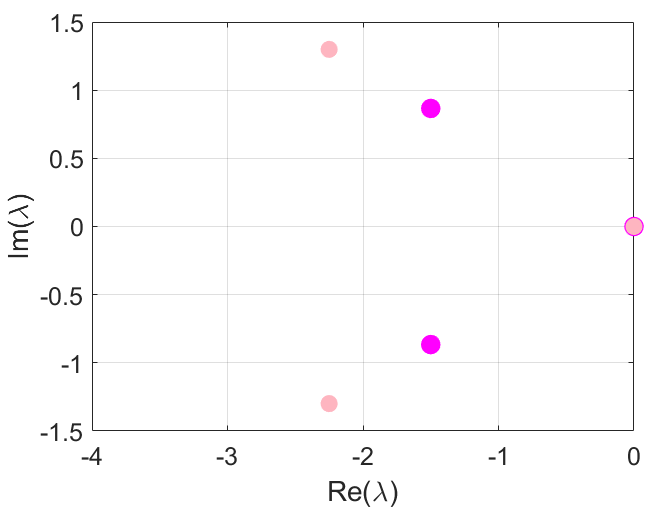}
    \includegraphics[scale=.3]{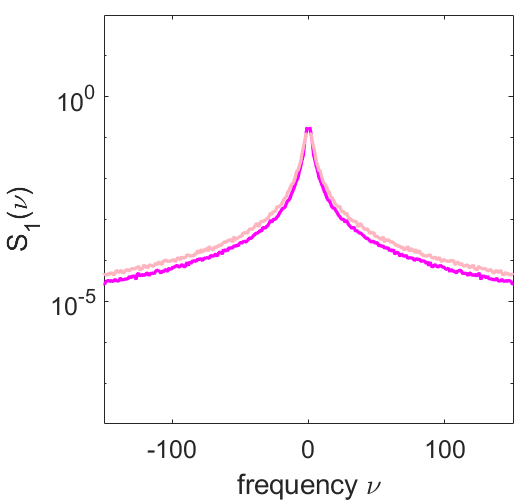}
    
    \caption{Uncoupled stochastic oscillator models.
    For each model, we show sample trajectories (left column), eigenvalue spectra and (or) time series of the first component of each oscillator (middle column), and power spectra of each oscillator (right column). 
    \textbf{Top row:} 4D linear model with parameters $\eta=0.1$, $\omega=2$, $D=0.1$, and $\kappa=0$ with $\tau=0$ (dark green), and $\tau=0.5$ (light green). 
    The left panel shows trajectories in the phase plane overlain with isochrons of the stochastic asymptotic phase function.
    \textbf{Middle row:} 2D ring model with parameters $\omega=2$, $D=0.1$, and $\kappa=0$ with $\tau=0$ (black) and $\tau=0.5$ (gray). The left panel depicts trajectories over a short time-period on the torus $\mathbb{T}^1\equiv[0,2\pi)$.
    \textbf{Bottom row:} 9D discrete-state model with parameters $\omega=1$ and $\kappa=0$, with $\tau=0$ (pink) and $\tau=0.5$ (light pink).
    The left panel depicts  realizations generated by the Gillespie algorithm.}
    \label{fig:model_systems}
\end{figure*}

\subsection{Eigenvalue bifurcations}\label{subsection: eval bif}

In this section, we consider the splitting of the $Q$-function eigenvalues, $\lambda_1=\mu_1+i\omega_1$, which are repeated with algebraic multiplicity two when the system is unperturbed, $\kappa=\tau=0$.
The analysis that we will employ also holds for the complex conjugate eigenvalue $\overline{\lambda_1}$, although we do not consider the complex conjugates here.
As in the proof of Theorem \ref{theorem:Q function perturbation}, we denote the true perturbed eigenvalues as $\lambda_\pm$, where the subscript $\pm$ refers to the splitting behavior of the eigenvalues.
Similarly, the leading order corrections to $\lambda_1$ are denoted $\lambda_{c\pm}$. 

\subsubsection{4D linear system}

We begin with the 4D linear system \eqref{eq:4D linear model}.
When $\kappa=\tau=0$, the unperturbed eigenvalues with positive imaginary part are
\begin{equation}
\label{eq:linear-system-evals-repeated-roots}
    \lambda_{1} = -\eta + i\omega \quad (\text{algebraic multiplicity 2})
\end{equation}
An exact expression for the $Q$-function eigenvalues as a function of $\kappa$ and $\tau$ is available, which serves as validation for our theory.
Let $\mathbf{w}^*$ be a left eigenvector of the matrix
\begin{equation}\label{eq:4D linear matrix}
    \mathcal{A} = \begin{bmatrix}
        -\eta & \omega+\tau & 0 & 0
        \\
        -(\omega+\tau) & -\eta & 0 & 0
        \\
        0 & 0 & -\eta & \omega
        \\
        0 & 0 & -\omega & -\eta
    \end{bmatrix} + \kappa \begin{bmatrix}
        -1 & 0 & 1 & 0
        \\
        0 & -1 & 0 & 1
        \\
        1 & 0 & -1 & 0
        \\
        0 & 1 & 0 & -1
    \end{bmatrix}
\end{equation}
Note that we write $\mathbf{w}^*$ for the row vector which is the conjugate transpose of the column vector $\mathbf{w}$.
Then, the $Q$-functions have the form
\begin{equation}\label{eq:4D linear backwards eigenfunctions}
    \mathcal{Q}^*(x_1,x_2,y_1,y_2) = \mathbf{w}^* \begin{bmatrix}
        x_1 \\ x_2 \\ y_1 \\ y_2
    \end{bmatrix}
\end{equation}
with eigenvalue that is identical to the  eigenvalue of $\mathcal{A}$ corresponding to $\mathbf{w}^*$.
In Appendix \S\ref{appendix: linear system eigenfunctions} we will prove the following general observation: 
The low-lying backward eigenfunctions of a linear system are themselves linear functions of the coordinates, with coefficients given by the left eigenvectors of the drift matrix $\mathcal{A}$.
That is, the form of $\mathcal{Q}^*$ given above is not a unique feature of  \eqref{eq:4D linear model}, but a general feature of any linear system of arbitrary dimension and architecture \cite{leen2016eigenfunctions}.
We remark that in the absence of noise, \eqref{eq:4D linear model} fails to oscillate, but irregular oscillations do occur for any finite-amplitude noise.
It is noteworthy that the backward eigenfunctions \eqref{eq:4D linear backwards eigenfunctions} are independent of the noise strength $D>0$; see \cite{thomas2019phase} for a more detailed study of noisy oscillating linear systems.


For the specific case of \eqref{eq:4D linear model}, exact expressions are available for the true eigenvalues as a function of $\kappa$ and $\tau$, which we denote as follows:
\begin{widetext}
\begin{equation}\label{eq:4D linear eigenvalues}
    \lambda_\pm(\kappa,\tau) = \begin{cases}
    -\eta-\kappa \pm \frac{1}{2}\sqrt{4\kappa^2 - \tau^2} + \left(\omega + \frac{\tau}{2}\right)i  , & \text{when } 4\kappa^2>\tau^2\\
    -\eta-\kappa + \left(\omega + \frac{\tau}{2}\right)i , & \text{when } 4\kappa^2=\tau^2\\
    -\eta-\kappa + \left[\left(\omega + \frac{\tau}{2}\right) \pm \frac{1}{2}\sqrt{\tau^2-4\kappa^2} \right]i, & \text{when } 4\kappa^2<\tau^2\\
    \end{cases}
\end{equation}
\end{widetext}
where we recall that the repeated unperturbed eigenvalue $\lambda_1$ generically splits into distinct eigenvalues $\lambda_\pm(\kappa,\tau)$ upon perturbation.

In the case of coupled identical oscillators ($\tau=0$ but $\kappa\neq 0$), the two eigenvalues take on a particularly simple form, namely
\begin{equation}\label{eq:linear_identical_eigenvalue_exact}
    \lambda_-(\kappa,0) = -\eta-2\kappa + i\omega, \quad \lambda_+(\kappa,0) = -\eta+ i\omega 
\end{equation}

We now analyze the leading order behavior of the repeated, unperturbed eigenvalue, $\lambda_1=-\eta+i\omega$, as a function of both $\kappa$ and $\tau$.
We directly compute the matrix $\mathcal{M}$ in \eqref{eq:M matrix} (see Appendix \ref{appendix:M compuation} for a more detailed derivation), which yields
\begin{equation}
    \mathcal{M} = \begin{bmatrix}
        -\kappa+i\tau & \kappa 
        \\
        \kappa & -\kappa
    \end{bmatrix}
\end{equation}
with eigenvalues 
\begin{equation}
    \lambda_{c\pm}(\kappa,\tau)=
    \frac{\tau}{2}i - \kappa \pm \frac{1}{2} \sqrt{4\kappa^2-\tau^2}
\end{equation}
Consequently, our analysis predicts that the repeated eigenvalue, $\lambda_1$, generically splits into distinct eigenvalues
\begin{equation}
\begin{split}
    \lambda_\pm(\kappa,\tau) &= (-\eta + i\omega) + \left[ \frac{\tau}{2}i - \kappa \pm  \frac{1}{2} \sqrt{4\kappa^2-\tau^2} \right] 
    \\
    &+ o(\kappa) + o(\tau)
\end{split}
\end{equation}
which recovers the eigenvalues of the exact solution \eqref{eq:4D linear eigenvalues}; see Figure \ref{fig:id_eval_bif}.
A similar analysis may be employed to study the complex conjugate eigenvalue $\overline{\lambda_1}$.

The splitting discriminant is given by
\begin{equation}\label{eq:linearD}
    \mathcal{D}(\kappa,\tau) = \sqrt{4\kappa^2 - \tau^2}
\end{equation}
One finds directly from the form of \eqref{eq:linearD}, or alternatively from Corollary \ref{corollary:splitting discriminant}, that
\begin{equation}
    \kappa^* = \mathfrak{K}(\tau^*) = \pm\frac{|\tau^*|}{2}
\end{equation}
satisfies $\mathcal{D}(\mathfrak{K}(\tau^*),\tau^*)=0$.
In the $(\tau,\kappa)$ plane, this relationship defines a $Q$-synchronization boundary,\; see Figure \ref{fig:tongues} in \S\ref{ssec:Arnold-tongues}.

\subsubsection{2D ring model}

When $\kappa=\tau=0$, analytic solutions are available for the eigenvalues and eigenfunctions of the backward operator (augmented with $2\pi$-periodic boundary conditions on $[0,2\pi) \times [0,2\pi)$) associated with \eqref{eq:2D Kuramoto model system} 
\begin{equation}\label{eq:ring_unperturbed_efunctions}
    \begin{split}
        \lambda_{m,n} &= -D(m^2+n^2) + i\omega (m+n)
        \\
        Q_{m,n}^*(x,y) &= [\cos(mx) + i\sin(mx)][\cos(ny) + i\sin(ny)]
        \\
        P_{m,n}(x,y) &= [\cos(mx) - i\sin(mx)][\cos(ny) - i\sin(ny)]
    \end{split}
\end{equation}
with $m,n\in \mathbb{Z}$.
The $Q$-functions of the unperturbed system correspond to the cases where $(m,n)=(1,0)$ and $(m,n)=(0,1)$.
The corresponding unperturbed eigenvalues $\lambda_{0,1}$ and $\lambda_{1,0}$ are given by
\begin{equation}
    \lambda_{1} = -D + i\omega \quad \text{(algebraic multiplicity 2)}
\end{equation}
For $\kappa,\tau\neq 0$, analytic solutions are not available.
However, one may establish a double Fourier ansatz, which transforms the finite-dimensional backward equation into an infinite dimensional linear system.
We implement an efficient, accurate numerical approximation  of the eigenfunctions and eigenvalues of the infinite-dimensional linear system using the method of continued fractions (CF) \cite{risken1996fokker}.
See Appendix \ref{appendix: continued fractions} for a detailed derivation of this procedure, which closely relates to applied Koopman theory and extended dynamic mode decomposition (EDMD) \cite{brunton2016koopman,brunton2021modern,colbrook2023mpedmd, mezic2022numerical, williams2015data}.

Our theory provides a leading order approximation in $\kappa$ and $\tau$ of the eigenvalues.
One may show by direct computation that the matrix $\mathcal{M}$ for the repeated, unperturbed eigenvalue, $\lambda_1=-D+i\omega$, is given by
\begin{equation}
    \mathcal{M} = \begin{bmatrix}
        i\tau & \kappa/2 
        \\
        \kappa/2 & 0
    \end{bmatrix}
\end{equation}
with eigenvalues 
\begin{equation}
    \lambda_{c\pm}= 
    \frac{\tau}{2}i \pm  \frac{1}{2}\sqrt{\kappa^2-\tau^2}
\end{equation}
Consequently, our analysis predicts that the unperturbed, repeated $Q$-function eigenvalue, $\lambda_1$, of \eqref{eq:2D Kuramoto model system} generically splits into distinct eigenvalues
\begin{equation}
\begin{split}
     \lambda_\pm(\kappa,\tau) &= (-D + i\omega) + \left[\frac{\tau}{2}i \pm \frac{1}{2}\sqrt{\kappa^2-\tau^2} \right] 
     \\
     &+ o(\kappa) + o(\tau)
\end{split}
\end{equation}
Figure \ref{fig:id_eval_bif} displays the results.

The splitting discriminant is given by
\begin{equation}\label{eq:ringD}
    \mathcal{D}(\kappa,\tau) = \sqrt{\kappa^2 - \tau^2}
\end{equation}
One finds directly from the form of \eqref{eq:ringD}, or alternatively from Corollary \ref{corollary:splitting discriminant}, that
\begin{equation}
    \kappa^* = \mathfrak{K}(\tau^*) = \pm|\tau^*|
\end{equation}
satisfies $\mathcal{D}(\mathfrak{K}(\tau^*),\tau^*)=0$.
In the $(\tau,\kappa)$ plane, this relationship defines a $Q$-synchronization boundary, see Figure \ref{fig:tongues} in \S\ref{ssec:Arnold-tongues}.

We remark that the CF method also produces accurate estimates of the eigenfunctions of the 2D ring model \eqref{eq:2D Kuramoto model system}.
Above the KT point,
the eigenfunction $\mathcal{Q}^*_+(\mbx,\mby)$ is the dominant eigenmode associated with eigenvalue $\lambda_+$, and corresponds to a synchronized activity (c.f., \S\ref{ssec:PowerSpectra_and_CrossSpectralDensity} pertaining to power spectra).
Meanwhile the eigenfunction $\mathcal{Q}_-^*(\mbx,\mby)$ associated with eigenvalue $\lambda_-$ corresponds to a nondominant mode corresponding to antisynchronized activity.  
In section \S\ref{ssec:PowerSpectra_and_CrossSpectralDensity}, we illustrate this interpretation by plotting the phase and magnitude of these eigenfunctions.

\subsubsection{9D discrete-state system}

Our first two examples show how our framework successfully applies to continuous time, continuous state sytems described by a Fokker-Planck equation.
However, our approach is not limited to such systems.  The $Q$-function approach and its associated phase reduction for stochastic oscillators applies in principle to any Markovian system satisfying the robustly oscillatory criterion (e.g., \cite{thomas2014asymptotic} exhibited both Fokker-Planck systems and also a conductance-based neural model with discrete-state channel noise). 
For our third example we illustrate how the framework applies to a discrete state system of the sort one might find in a finite-population chemical system \cite{anderson2015stochastic}.

We remark that although the discrete-state system \eqref{eq:discrete_coupled_system} is not of the form of the general system \eqref{eq:SDE} considered throughout the paper (and so the explicit results of Corollary \ref{corollary:splitting discriminant} do not apply to this system because the structure of the perturbation in $\kappa$ and $\tau$ differs from the specific case that we considered), a similar eigenvalue perturbation analysis may be employed.  

When $\kappa=\tau=0$, the unperturbed eigenvalues of \eqref{eq:discrete_coupled_system} with least negative real part and positive imaginary part are
\begin{equation}\label{eq:discrete-system-evals-repeated-roots}
    \lambda_1 = \left(-\frac{3}{2} + \frac{\sqrt{3}}{2}i\right)\omega \quad \text{(algebraic multiplicity 2)}
\end{equation}

In the case when $\kappa,\tau\neq 0$, an exact expression is available for the true perturbed eigenvalues 
\begin{widetext}
\begin{equation}\label{eq:discrete_eval_exact}
    \lambda_{\pm}(\kappa,\tau)=\begin{cases}
        -\frac{3}{2}\omega - \frac{3}{4}\tau  + \left(\frac{\sqrt{3}}{2}\omega  +\frac{\sqrt{3}}{4}\tau\right)i \pm  \frac{\sqrt{3}- i}{4}\sqrt{-4\kappa^2+3\tau^2}, & \text{when } 3\tau^2>4\kappa^2
        \\
        -\frac{3}{2}\omega - \frac{3}{4}\tau  + \left(\frac{\sqrt{3}}{2}\omega  +\frac{\sqrt{3}}{4}\tau\right)i, & \text{when } 3\tau^2=4\kappa^2
        \\
        -\frac{3}{2}\omega - \frac{3}{4}\tau  + \left(\frac{\sqrt{3}}{2}\omega  +\frac{\sqrt{3}}{4}\tau\right)i \pm \frac{\sqrt{3}i+ 1}{4}\sqrt{4\kappa^2-3\tau^2}, & \text{when } 3\tau^2<4\kappa^2
    \end{cases}
\end{equation}
\end{widetext}
where we recall that the unperturbed repeated eigenvalue $\lambda_1$ generically splits into distinct eigenvalues $\lambda_\pm(\kappa,\tau)$ upon perturbation.

In the case of identical oscillators ($\tau=0$), the two eigenvalues with positive imaginary part take on a simple form
\begin{equation}\label{eq:linear_identical_eigenvalue_exact2}
\begin{split}
    \lambda_-(\kappa,0) &= \left(-\frac{3}{2} + \frac{\sqrt{3}}{2} i\right)\omega - \kappa \left( \frac{1}{2} + \frac{\sqrt{3}}{2}i \right)
    \\
    \lambda_+(\kappa,0) &=  \left(-\frac{3}{2} + \frac{\sqrt{3}}{2} i\right)\omega + \kappa \left( \frac{1}{2} + \frac{\sqrt{3}}{2}i \right)
\end{split}
\end{equation}

We now analyze the leading order behavior of the repeated, unperturbed eigenvalue, $\lambda_1= \left(-\frac{3}{2} + \frac{\sqrt{3}}{2}i\right)\omega$, as a function of both $\kappa$ and $\tau$.
Directly computing the matrix $\mathcal{M}$ in \eqref{eq:M matrix} yields
\begin{equation}
\scalebox{.85}{$
    \mathcal{M} = \begin{bmatrix}
        \kappa\left(\frac{1}{2}+\frac{\sqrt{3}}{2}i\right) - \tau\left(\frac{3}{8}-\frac{5\sqrt{3}}{8}i\right) & \frac{3}{2}\kappa + \frac{21}{16}\tau \\ 
        \tau\left(\frac{1}{2}-\frac{\sqrt{3}}{2}i \right) & -\kappa\left(\frac{1}{2}+\frac{\sqrt{3}}{2}i\right) - \tau\left(\frac{9}{8} + \frac{\sqrt{3}}{8}i\right)
    \end{bmatrix}
    $}
\end{equation}
with eigenvalues 
\begin{equation}
\begin{split}
    \lambda_{c\pm}(\kappa,\tau) &= -\tau\left(\frac{3}{4} - \frac{\sqrt{3}}{4}i\right) 
    \\
    &\quad\pm \frac{\sqrt{2}}{4}\sqrt{-4\kappa^2\left(1-\sqrt{3}i\right) + 3\tau^2\left(1-\sqrt{3}i\right)}
\end{split}
\end{equation}
Consequently, our analysis predicts that the repeated eigenvalue, $\lambda_1$, generically splits into distinct eigenvalues
\begin{widetext}
\begin{equation}
    \lambda_\pm(\kappa,\tau) = \left(-\frac{3}{2} + \frac{\sqrt{3}}{2} i\right)\omega  + \left[ -\tau\left(\frac{3}{4} - \frac{\sqrt{3}}{4}i\right) \pm \frac{\sqrt{2}}{4}\sqrt{-4\kappa^2\left(1-\sqrt{3}i\right) + 3\tau^2\left(1-\sqrt{3}i\right)} \right] + o(\kappa) + o(\tau)
\end{equation}
\end{widetext}
which recovers the two eigenvalues of the exact solution \eqref{eq:discrete_eval_exact}; see Figure \ref{fig:id_eval_bif}.
A similar analysis may be employed for the complex conjugate eigenvalue $\overline{\lambda_1}$.

The splitting discriminant is given by
\begin{equation}\label{eq:discreteD}
    \mathcal{D}(\kappa,\tau) = \frac{\sqrt{3}i+ 1}{4}\sqrt{4\kappa^2-3\tau^2}
\end{equation}
One finds directly from the form of \eqref{eq:discreteD} that
\begin{equation}
    \kappa^* = \mathfrak{K}(\tau^*) = \pm\frac{\sqrt{3}}{2}|\tau^*|
\end{equation}
satisfies $\mathcal{D}(\mathfrak{K}(\tau^*),\tau^*)=0$.
In the $\tau,\kappa$ plane, this relationship defines a $Q$-synchronization boundary, see Figure \ref{fig:tongues} in \S\ref{ssec:Arnold-tongues}. 

Note that in this case, we include the multiplicative prefactor $\frac{\sqrt{3}i+1}{4}$ in \eqref{eq:discreteD} to emphasize that the discriminant in this case is neither purely real nor purely imaginary as in the previous two cases.
This fact is reflected in the right column of Figure \ref{fig:id_eval_bif}, which shows splitting of both the real and imaginary parts of the eigenvalues.

\begin{figure*}[ht]
    \centering
    \includegraphics[scale=.4]{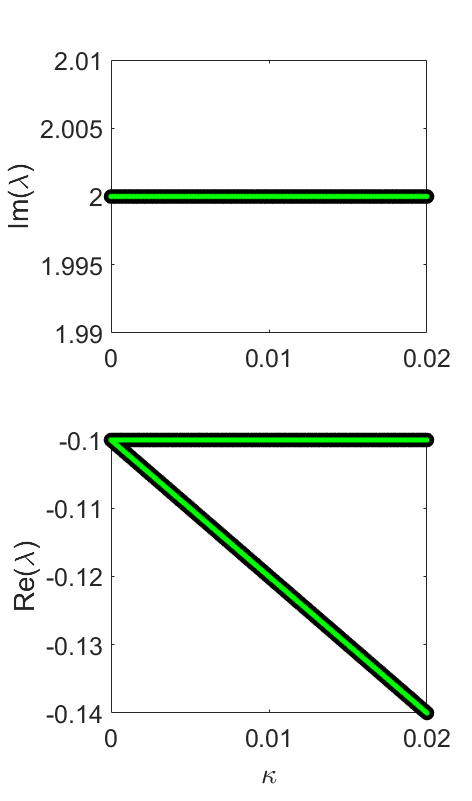}
    \includegraphics[scale=.4]{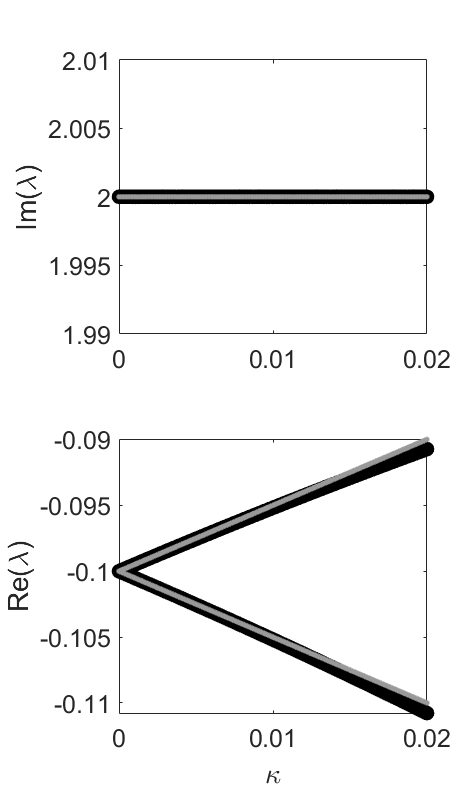}
    \includegraphics[scale=.4]{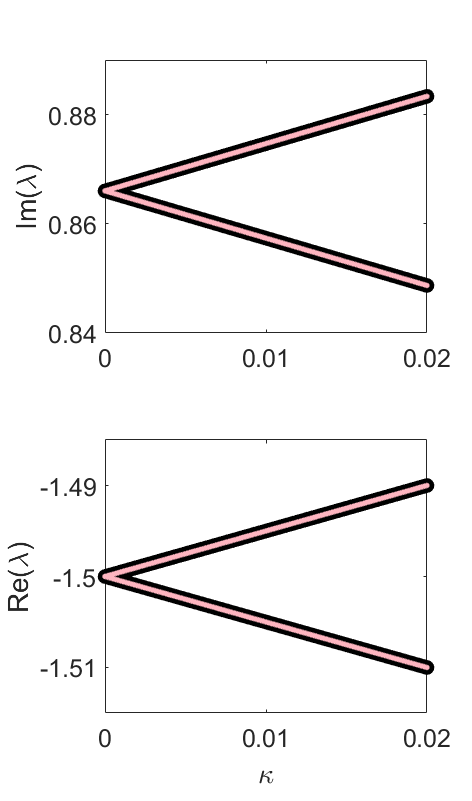}
    
    \hspace*{\fill}\textcolor{lightgray}{\rule{0.6\textwidth}{0.4pt}}\hspace*{\fill}
    \vspace{.15cm}
    
    \includegraphics[scale=.4]{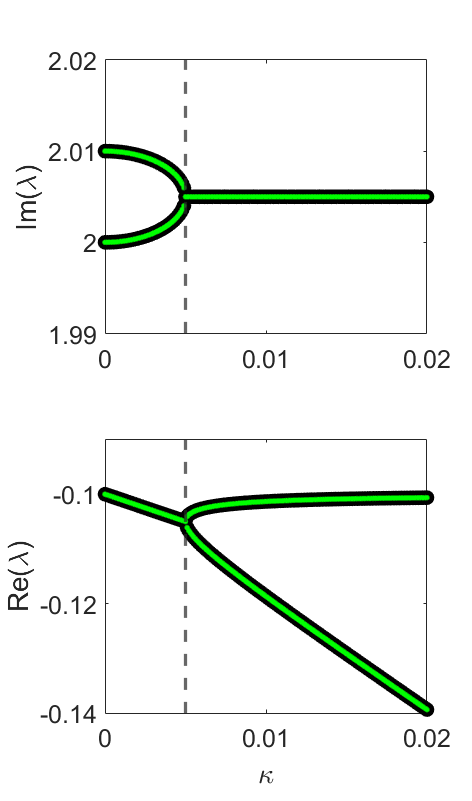}
    \includegraphics[scale=.4]{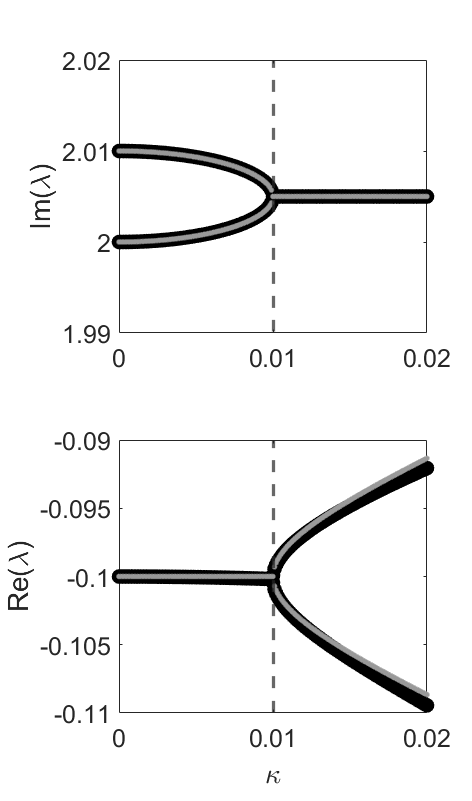}
    \includegraphics[scale=.4]{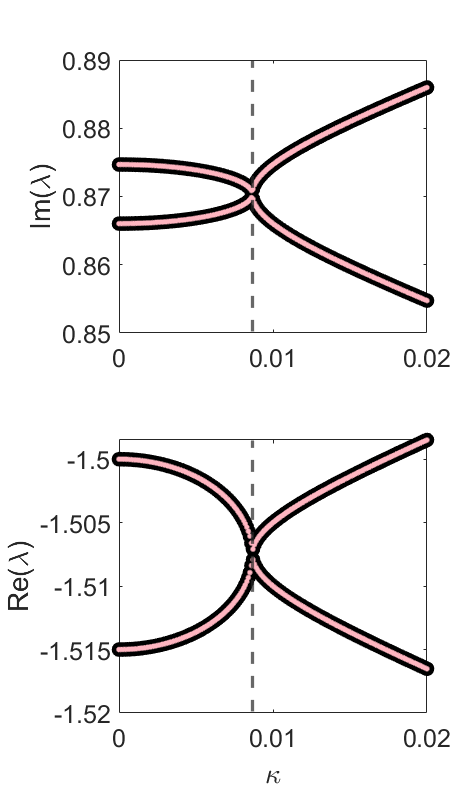}
    
    \caption{Eigenvalue bifurcations in the case of identical oscillators ($\tau=0$, top block) and in the case of non-identical oscillators ($\tau=0.01$, bottom block). Black lines denote ground truth, while colored lines denote the leading order analytic approximation.
    The top row of each block shows the imaginary parts of the eigenvalues, while the bottom row of each block shows the real parts.
    Vertical dotted lines denote the leading order approximation to the KT point.
    \textbf{Left:} 4D linear system with $\omega=2$ and $D=0.1$.
    \textbf{Middle:} 2D ring model with $\omega=2$ and $D=0.1$. \textbf{Right:} 9D discrete-state model with $\omega=2$.}
    \label{fig:id_eval_bif}
\end{figure*}

\subsection{Power spectra and cross-spectra}
\label{ssec:Examples-PowerSpectra}

In this section, we analyze the power spectra and cross-spectra of our three model systems in the case of non-identical coupled units $(\tau\neq 0)$.

As in the preceding sections, we denote the eigenvalues associated with the $Q$-functions of the joint system as $\lambda_\pm = \mu_\pm+i\omega_\pm$, which have positive imaginary part.
As always, the results presented in this section also hold for the complex conjugate eigenvalues with negative imaginary part, although we do not explicitly consider these eigenvalues in this section.

We begin with the power spectra.
In the case of our continuous-state examples, below the KT point, the eigenvalues have identical real parts, with $\mu_+=\mu_-$, but non-identical imaginary parts, with $\omega_+\neq \omega_-$.
Consequently, both power spectra are Lorentzian with equal half-width, but have peaks at different frequencies.
At the bifurcation point, the eigenvalues are identical with $\mu_+=\mu_-$ and $\omega_+=\omega_-$.
In consequence, the power spectra of the two oscillators are identical at the bifurcation point.
Above the bifurcation, the real parts of the eigenvalues are distinct, with $\mu_+\neq \mu_-$, but the imaginary parts are identical, with $\omega_+=\omega_-$.
Accordingly, each of the power spectra peak at the same frequency, but have different widths (recall \S\ref{ssec:PowerSpectra_and_CrossSpectralDensity}).

Physically, below the KT point each oscillator has a prominently peaked power spectrum, i.e., for weak coupling, the system behaves as two distinct stochastic oscillators.  
As the coupling increases, the power spectra ``collide'' and become identical at the KT point.
Importantly, the oscillators have identical (mean) frequencies in $Q$-function coordinates because the imaginary parts of the $Q$-function eigenvalues are identical.
In contrast, in their original coordinates, the power spectra of each oscillator have peaks which generically only asymptotically approach each other as coupling strength increases (see Figure \ref{fig:power spectra}, lightly shaded curves).
Above the KT point, the coupled system is \emph{robustly oscillatory}, i.e., there is a unique low-lying eigenmode (the $Q$-function) which dominates system dynamics at intermediate-long times.
Only one of the two power spectra remains prominent, the other flattens and becomes less prominent because the corresponding $Q$-function becomes a high-order mode after the bifurcation.
Recalling the discussion after Theorem \ref{theorem:Q function perturbation}, the low-lying $Q$-function above the bifurcation is analogous to the center-of-mass mode for a harmonic oscillator, and captures synchronous activity.
The complex argument of this $Q$-function is shown in Figure \ref{fig:physics}.
In regions where the magnitude of the $Q$-function is near 0, the color becomes black to indicate that the phase captured by this eigenmode is not robust and could change values very rapidly.
The magnitude of the center-of-mass eigenmode vanishes only at anti-synchronous solutions $y\approx x+\pi$.
In contrast, the non-dominant $Q$-function above the bifurcation is similar to the antiphase mode of a harmonic oscillator, and captures antisynchronous activity.
The complex argument (again scaled according to the magnitude of the $Q$-function) is shown in Figure \ref{fig:physics}, and indicates that this eigenmode captures only antisynchronous behavior.

\begin{figure}[ht]
    \centering
    \includegraphics[scale=.35]{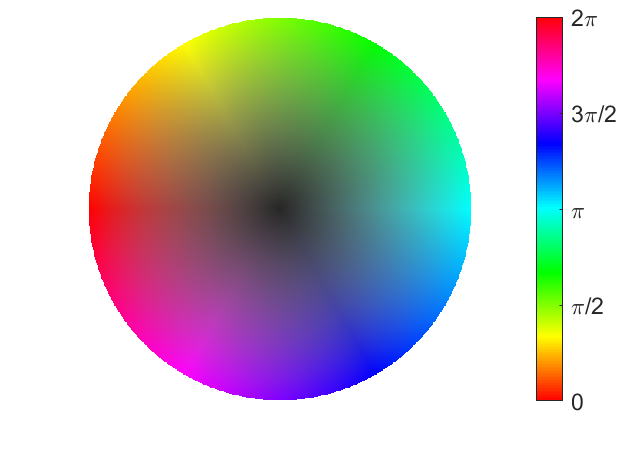}
    
    \includegraphics[scale=.35]{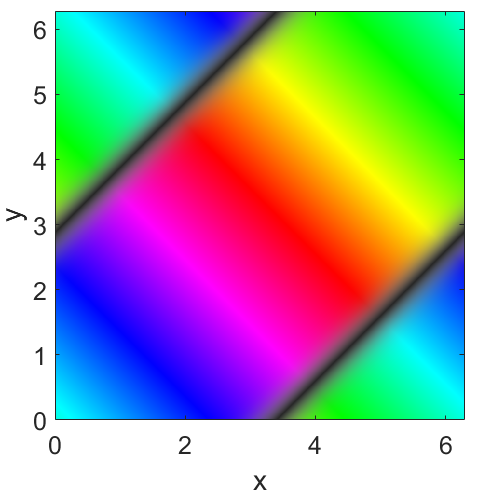}
    
    \includegraphics[scale=.35]{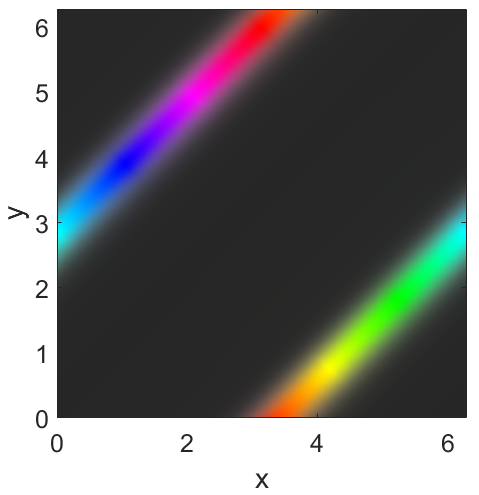}
    \caption{Center-of-mass and antiphase interpretation of the joint $Q$-functions of the 2D ring model \eqref{eq:2D Kuramoto model system}.
    The complex argument (stochastic asymptotic phase) of the center-of-mass $Q$-function (middle) and antiphase $Q$-function (bottom) are displayed according to a two-dimensional colormap (top). 
    The colormap is black when the magnitude of the $Q$-function vanishes, and takes on fully saturated color values when the magnitude  of the $Q$-function is maximal.
    Intermediate hues are linearly interpolated between these extremes as a function of the magnitude.}
    \label{fig:physics}
\end{figure}

In the case of the discrete-state system, the real and imaginary parts of the eigenvalues are distinct, except at the KT point.
While the the mean rotation of each oscillator is non-identical, the isochrons of each oscillator become identical after the KT point; see Appendix \ref{appendix:discrete system bifurcation}.
Figure \ref{fig:power spectra} depicts plots of the power spectra for each oscillator in our example systems, both in $Q$-function coordinates (darkly shaded curves) and in their original coordinates (lightly shaded curves).
Numerical simulations are in excellent agreement with the analytic results below, at, and above the KT point.

We also find that the eigenvalue bifurcation reflects a qualitative change in the cross-spectral density of the coupled systems.
Notably, the cross-spectra is purely real at the bifurcation point.
Results are summarized in Figures \ref{fig:cross spectra ou}, \ref{fig:cross spectra ring}, and \ref{fig:cross spectra discrete}.

\begin{figure*}[ht]
    \centering
    \includegraphics[scale=.35]{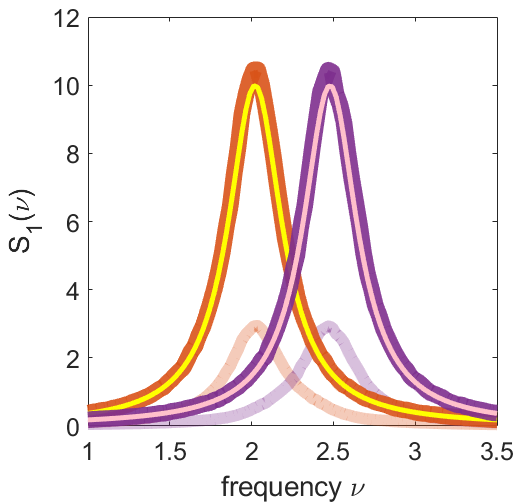}
    \includegraphics[scale=.35]{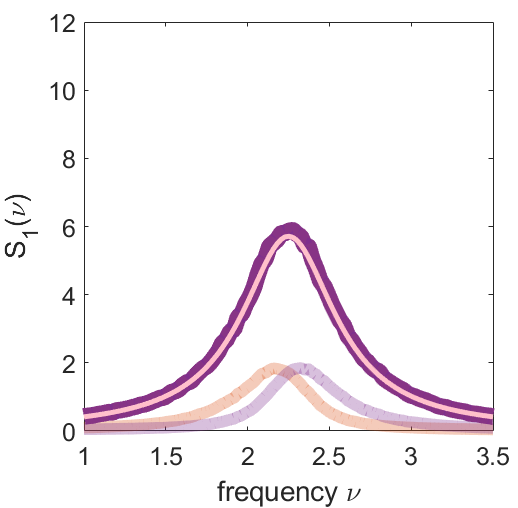}
    \includegraphics[scale=.35]{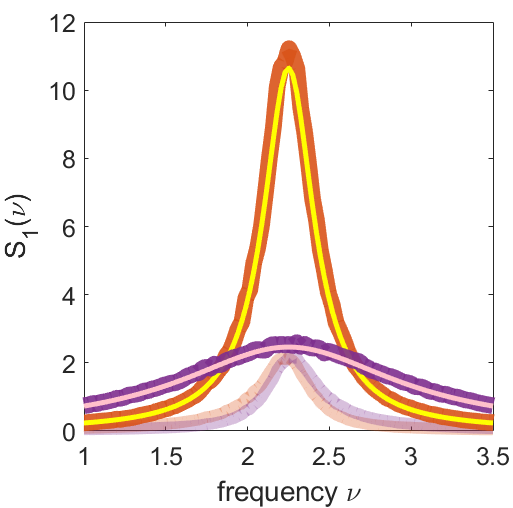}
    
    \includegraphics[scale=.35]{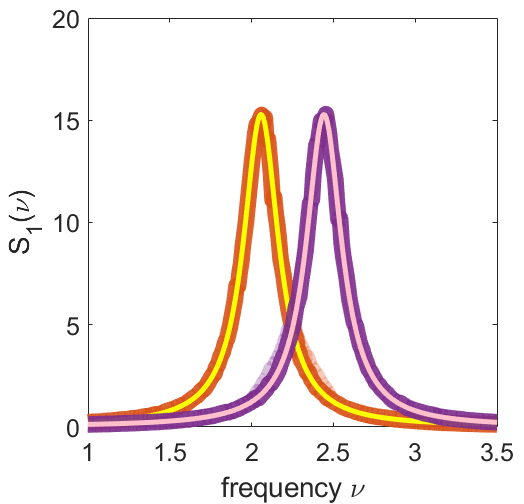}
    \includegraphics[scale=.35]{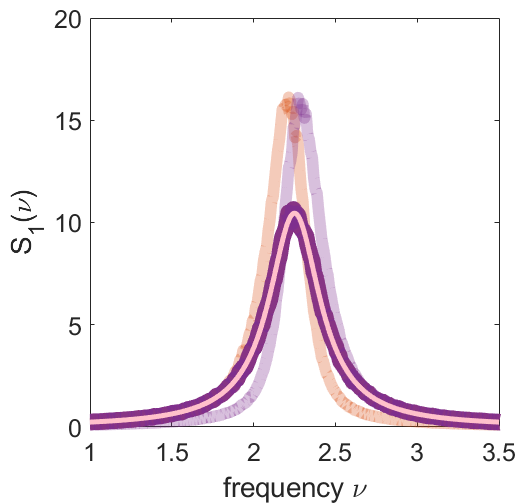}
    \includegraphics[scale=.35]{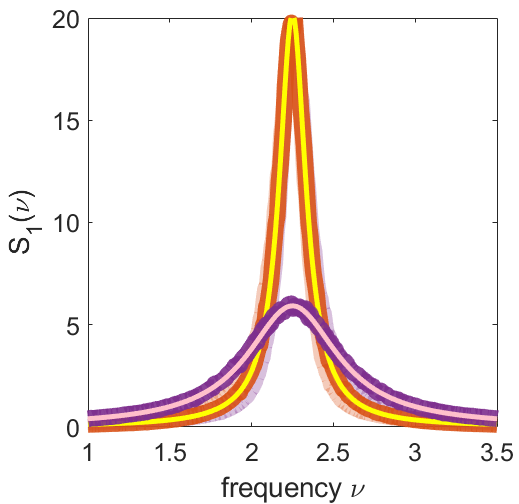}
    
    \includegraphics[scale=.35]{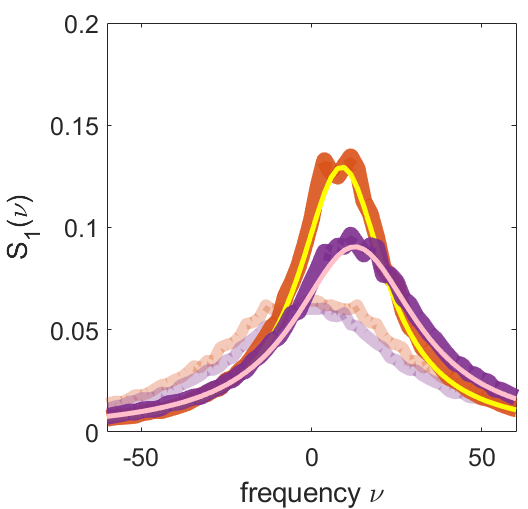}
    \includegraphics[scale=.35]{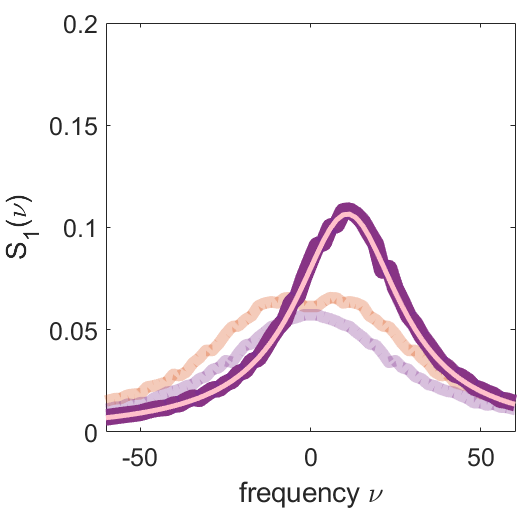}
    \includegraphics[scale=.35]{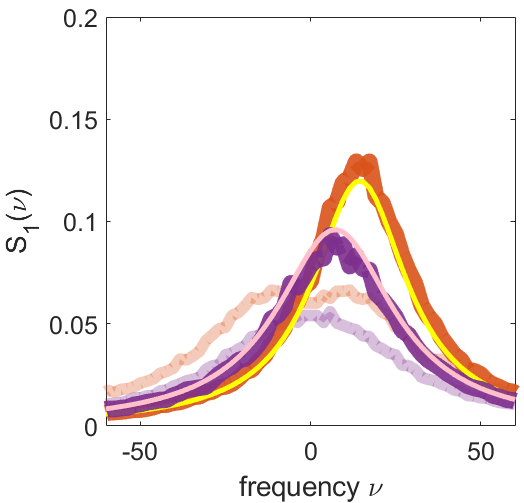}
    \caption{Power spectra below the KT point (left column), at the KT point (middle column), and above the KT point (right column) in $Q$-function coordinates (orange and purple, computations; yellow and pink, theory), and in the original coordinates (light orange and light purple). In $Q$-function coordinates,
    the power spectra of the two oscillators peak at identical frequencies after the bifurcation. \textbf{Top:} 4D linear model with parameters $\eta=0.1$, $D=0.1$, $\omega=2$, $\tau=0.5$, and coupling strengths $\kappa=0.1, \; 0.25, \; 0.4$.
    \textbf{Middle:} 2D ring model with parameters $\omega=2$, $\tau=0.5$, $D=0.1$, and coupling strengths $\kappa=0.2,\; 0.33596, \; 0.38$.
    The power spectra in original coordinates are obscured by those in $Q$-function coordinates in the left and right panels.
    \textbf{Bottom:} 9D discrete-state model with parameters $\omega=10$ and $\tau=5$, and coupling strengths $\kappa=2, \; 5\sqrt{3}/2, \; 6$.}
    \label{fig:power spectra}
\end{figure*}

\begin{figure*}[ht]
    \centering
    \includegraphics[scale=.35]{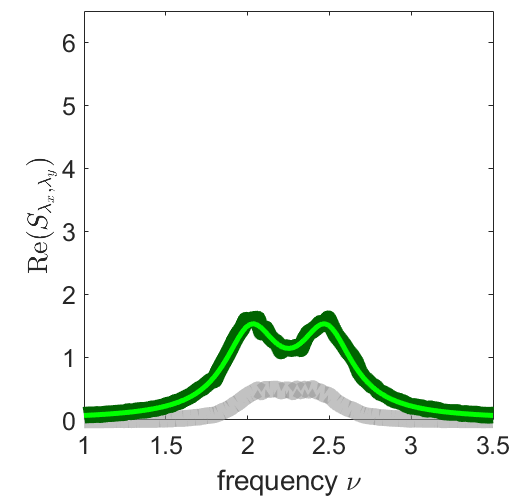}
    \includegraphics[scale=.35]{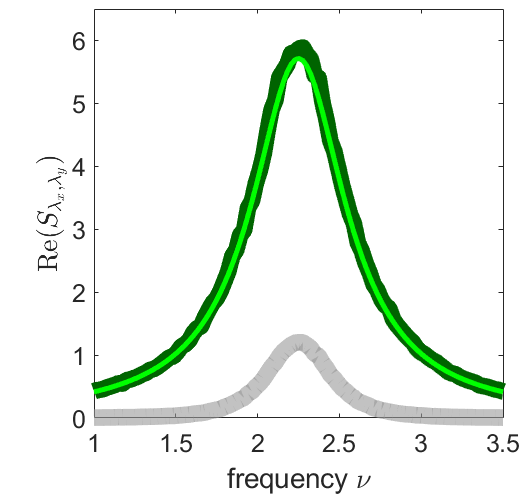}
    \includegraphics[scale=.35]{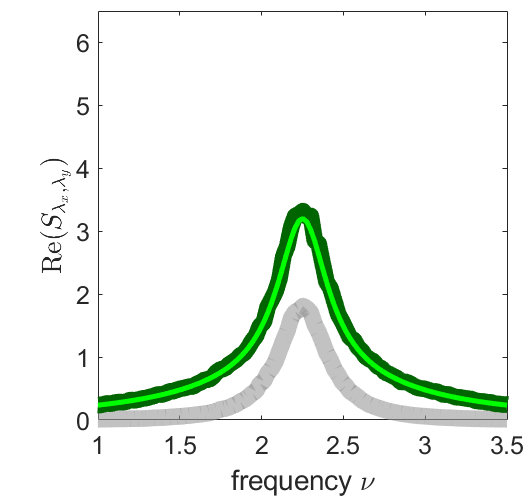}

    \includegraphics[scale=.35]{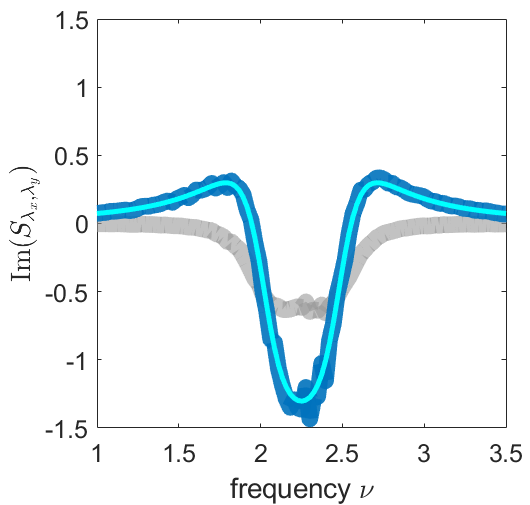}
    \includegraphics[scale=.35]{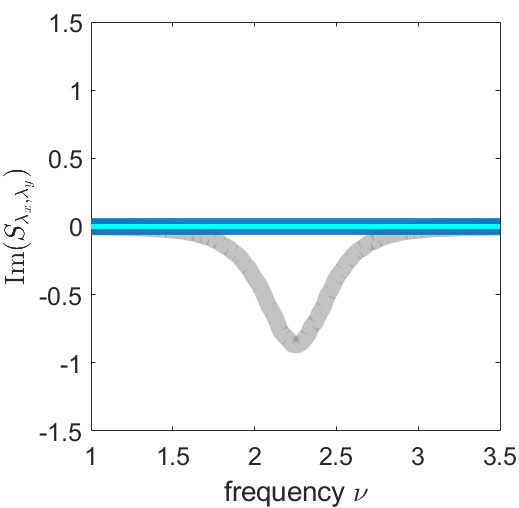}
    \includegraphics[scale=.35]{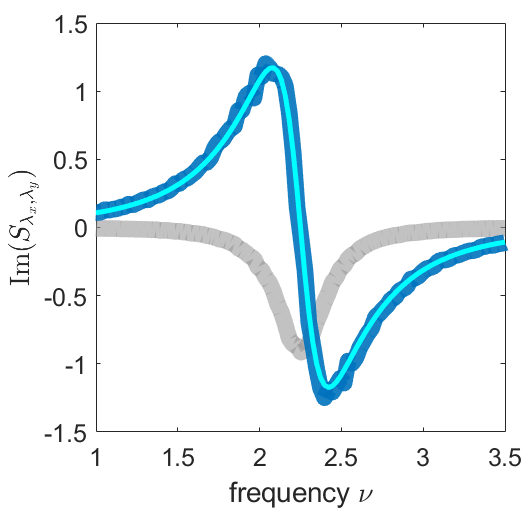}    
    \caption{Real and imaginary parts of the cross spectra below the bifurcation (left column), at the bifurcation (middle column), and above the bifurcation (right column) in $Q$-function coordinates (dark green and blue, computations; light green and blue, theory), and in the original coordinates (light gray). 
    The cross spectra in $Q$-function coordinates are purely real at the bifurcation. 
    Here, we show results for the 4D linear model \eqref{eq:4D linear model} with parameters $\eta=0.1$, $\omega=2$, $\tau=0.5$, $D=0.1$, and coupling strengths $\kappa=0.1, \; 0.25, \; 0.4$.}
    \label{fig:cross spectra ou}
\end{figure*}

\begin{figure*}
    \centering
    \includegraphics[scale=.35]{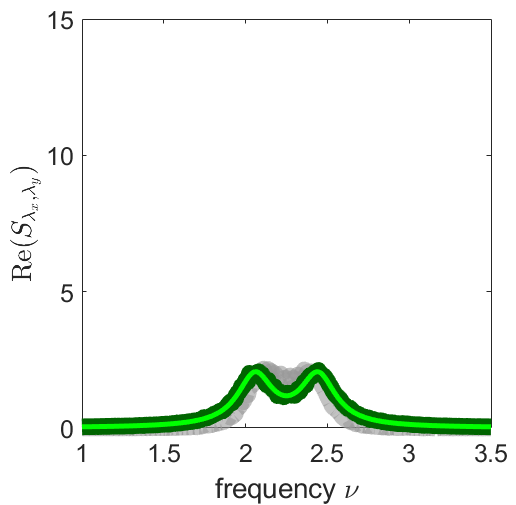}
    \includegraphics[scale=.35]{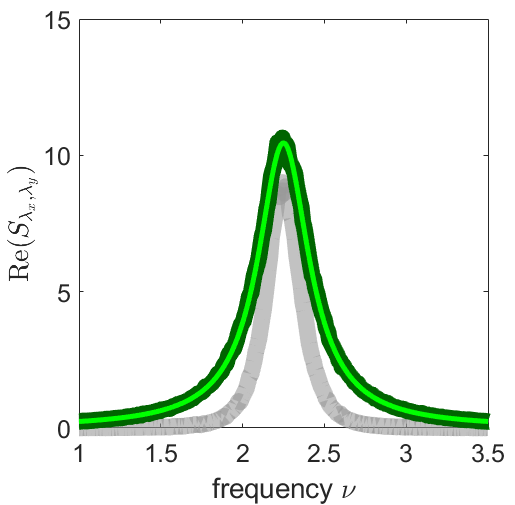}
    \includegraphics[scale=.35]{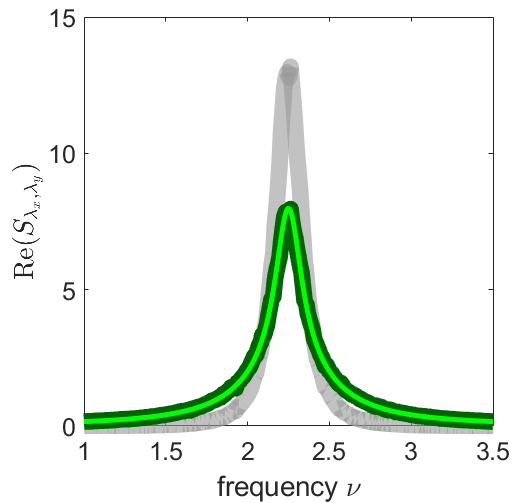}

    \includegraphics[scale=.35]{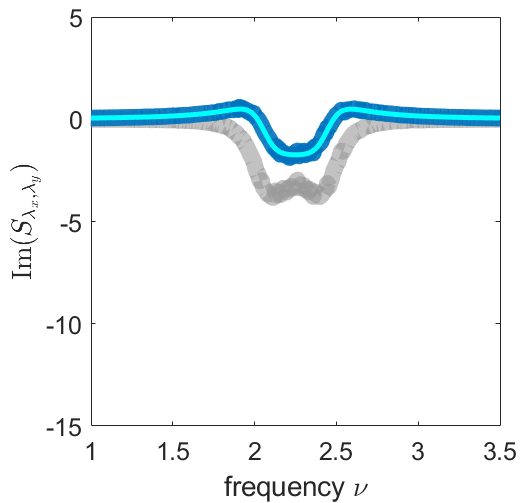}
    \includegraphics[scale=.35]{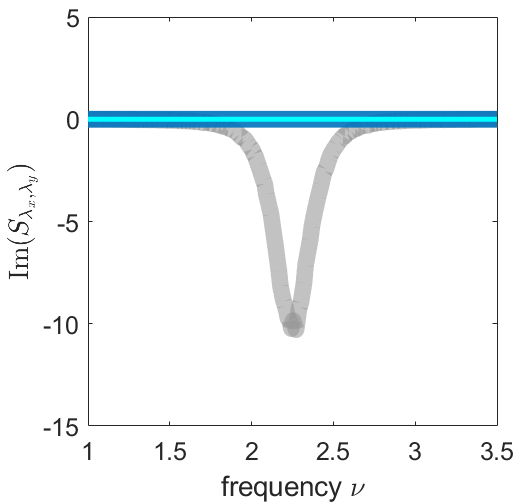}
    \includegraphics[scale=.35]{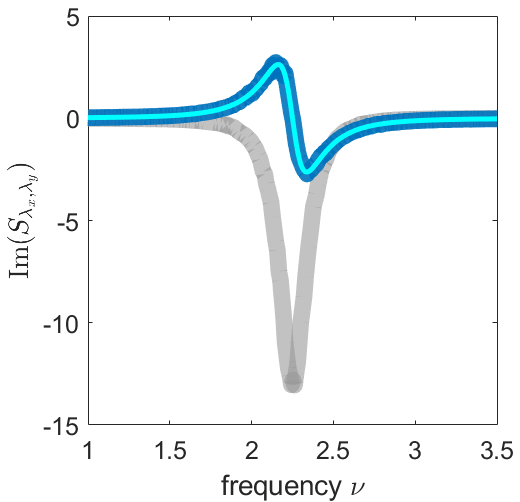}
    \caption{Real and imaginary parts of the cross spectra for the 2D ring model \eqref{eq:2D Kuramoto model system} with parameters $\omega=2$, $\tau=0.5$, $D=0.1$, and coupling strengths $\kappa=0.15,\; 0.33596, \; 0.45$.
    The color scheme is as in Figure \ref{fig:cross spectra ou}.}
    \label{fig:cross spectra ring}
\end{figure*}

\begin{figure*}
    \centering
    \includegraphics[scale=.35]{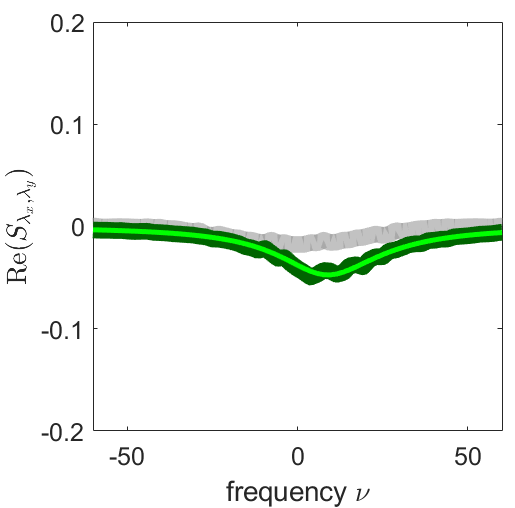}
    \includegraphics[scale=.35]{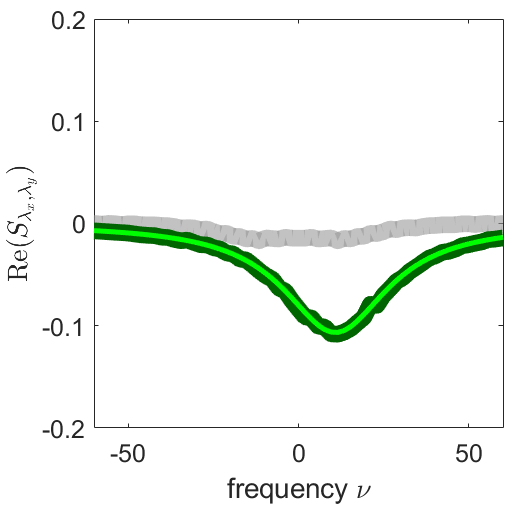}
    \includegraphics[scale=.35]{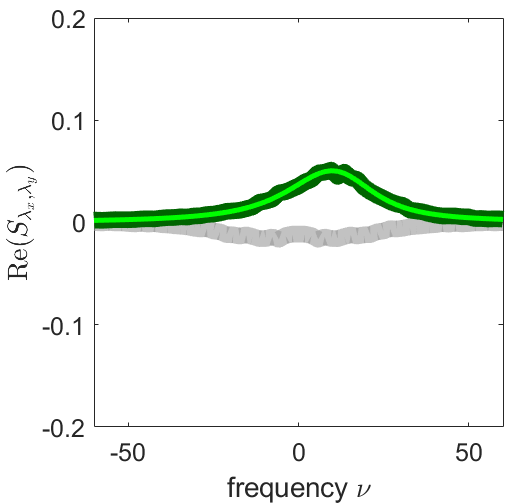}

    \includegraphics[scale=.35]{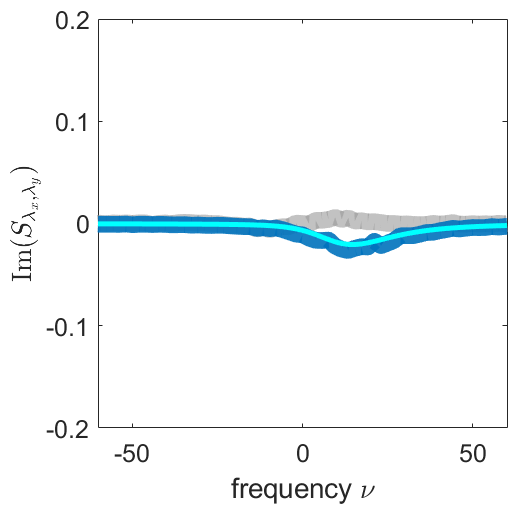}
    \includegraphics[scale=.35]{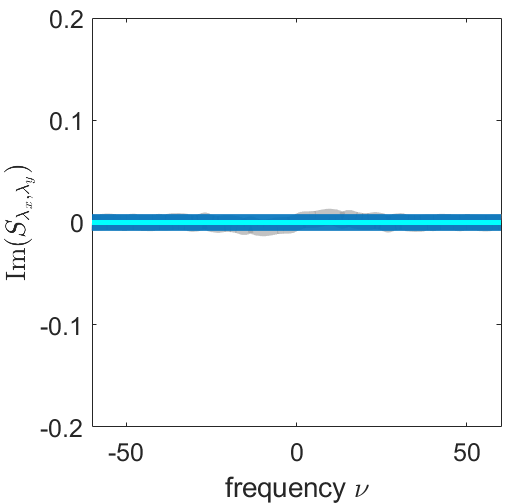}
    \includegraphics[scale=.35]{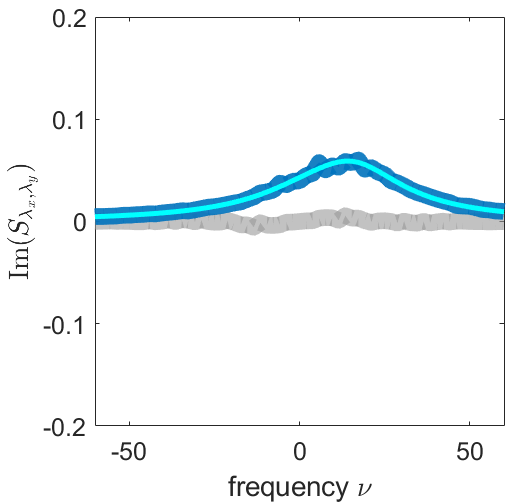}
    \caption{Real and imaginary parts of the cross spectra for the 9D discrete-state system \eqref{eq:discrete_coupled_system} with parameters $\omega=10$ and $\tau=5$, and coupling strengths $\kappa=2, \; 5\sqrt{3}/2, \; 6$.
    The color scheme is as in Figure \ref{fig:cross spectra ou}.}
    \label{fig:cross spectra discrete}
\end{figure*}

\subsection{Arnold tongues}\label{ssec:Arnold-tongues}

Here, we display synchronization regimes for each of our three models.
In \S\ref{subsection: eval bif}, we provided analytic approximations for the synchronization regimes of three distinct systems of coupled stochastic oscillators via Corollary \ref{corollary:splitting discriminant}.
Here, we  supplement theory with numerical simulations outside the small parameter regime.
For the 4D linear system and 9D discrete system, exact expressions are available for the SKO eigenvalues.
For the 2D ring model, we employ the method of continued fractions to accurately approximate the eigenvalues of the SKO (see Appendix \ref{appendix: continued fractions}).
We remark that our synchronization regimes bear a strong resemblance to $1:1$ mode-locking Arnold tongues for deterministic coupled oscillators \cite{arnold2012geometrical, guckenheimer2013nonlinear, izhikevich2007dynamical, schilder2007computing}.

\begin{figure*}[ht]
    \centering
    \includegraphics[scale=.43]{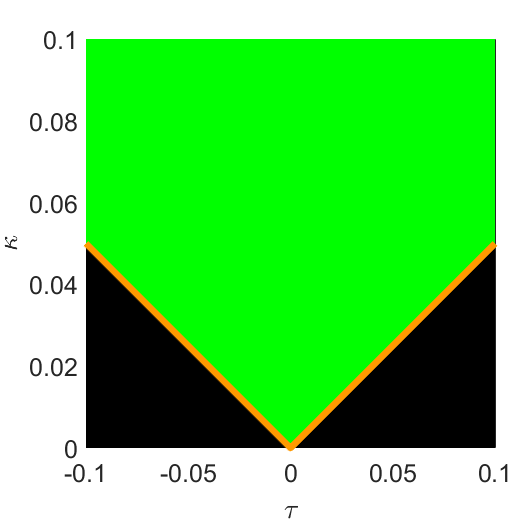}
    \includegraphics[scale=.43]{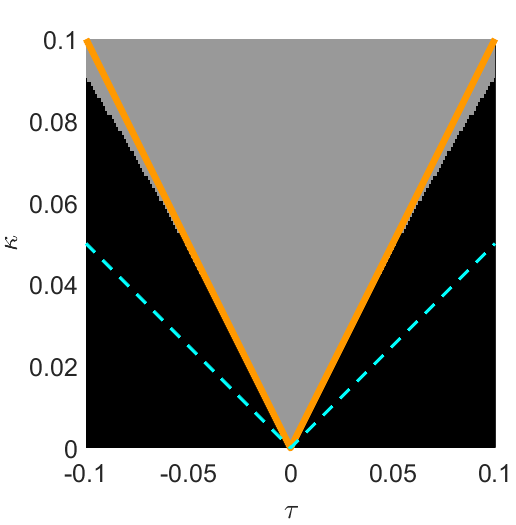}
    \includegraphics[scale=.43]{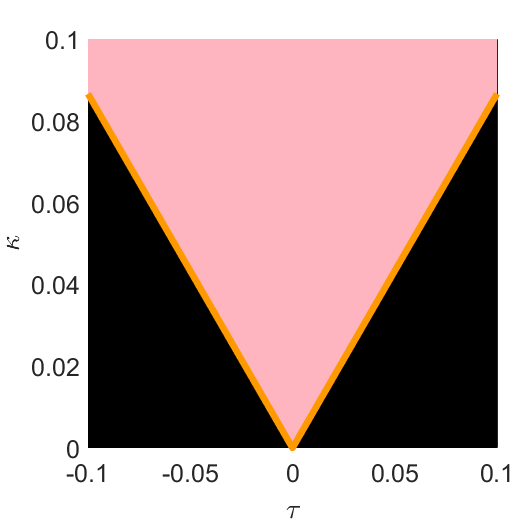}
    \caption{Arnold tongues for stochastic oscillators in the $\tau$,$\kappa$ plane. 
    Colored regions denote parameter pairs $(\tau,\kappa)$ which lie above the KT point, i.e., for which the system is $Q$-synchronized.
    The boundary of each colored region corresponds to $\tau,\kappa$ pairs for which the $Q$-function eigenvalues are repeated.
    The orange line is the analytic approximation of the $Q$-synchronization boundary obtained from Corollary \ref{corollary:splitting discriminant}.
    \textbf{Left:} 4D linear system with $\omega=2$ and $D=0.1$. The boundary is given by $\kappa^*={|\tau^*|}/{2}$ and is independent of noise intensity $D\neq 0$.
    \textbf{Middle:} 2D ring model with $\omega=2$ and $D=0.1$.
    When $D=0$, an exact expression is available for the synchronization boundary: $\kappa^*={|\tau^*|}/{2}$ (blue dashed line).
    When $D\neq 0$, the true boundary is not available analytically, but a leading order approximation is given by $\kappa^* = |\tau^*|$.
    \textbf{Right:} 9D discrete-state system with $\omega=2$. The boundary is given by $\kappa^*=\sqrt{3}|\tau^*|/2$.}
    \label{fig:tongues}
\end{figure*}

The synchronization boundary of the 2D ring model depends on the noise intensity, $D$. 
We find numerically in Figure \ref{fig:tongue_noise} (left panel) that our leading order approximation $\kappa^* = |\tau^*|$ is accurate even for large perturbations $\kappa$,$\tau$ when $D$ is large.
On the other hand, for smaller $D$ values, our leading order approximation is accurate only for small perturbations (Figure \ref{fig:tongue_noise}, middle and left panels).
We remark that the deterministic $(D=0)$ synchronization boundary is given by $\kappa^*={|\tau^*|}/{2}$.
For intermediate values of $\tau$ and $\kappa$, the $Q$-function synchronization region (Figure \ref{fig:tongue_noise}, middle panel, grey region) does not differ significantly from the deterministic Arnold tongue (dashed blue lines) although both boundaries differ from the leading-order analytic expression (orange curve).
For smaller values of $\tau$ and $\kappa$ (Figure \ref{fig:tongue_noise}, right panel) the leading-order approximation captures the $Q$-function synchronization boundary, which differs by a factor of two from the deterministic Arnold tongue boundary on this scale.  

\begin{figure*}[ht]
    \centering
    \includegraphics[scale=.43]{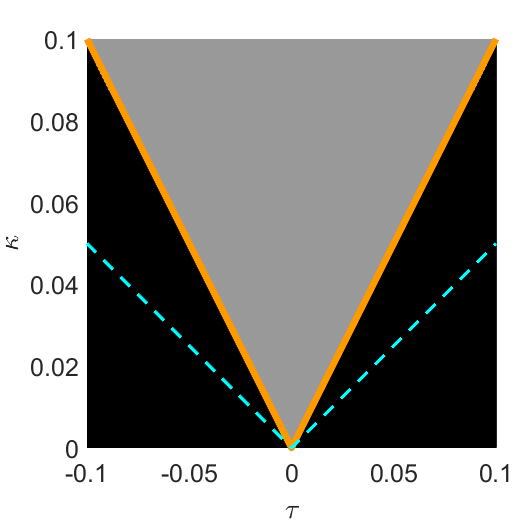}
    \includegraphics[scale=.43]{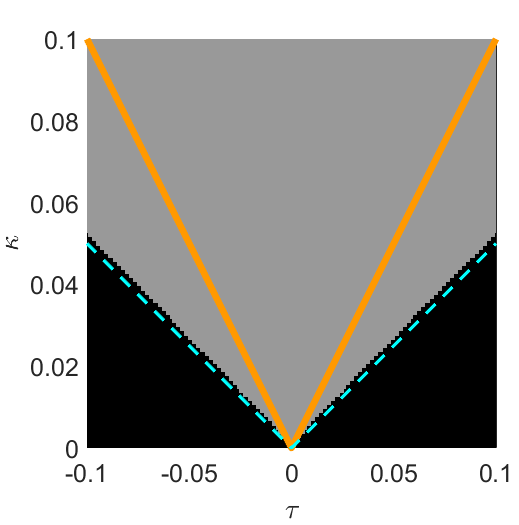}
    \includegraphics[scale=.43]{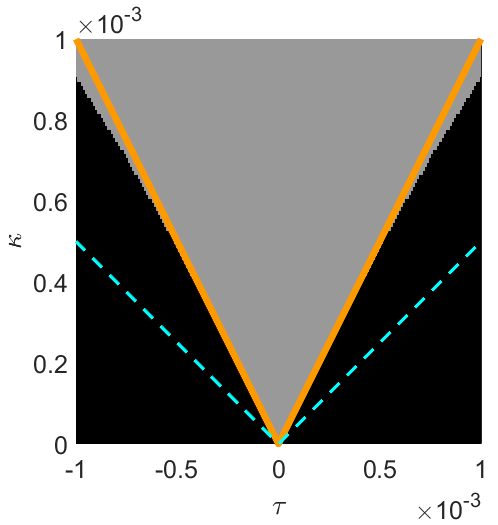}
    \caption{Synchronization boundary for the 2D ring model as a function of $D$ with parameter values and color-scheme as in Figure \ref{fig:tongues}.
    \textbf{Left:} $D=0.3$.  \textbf{Middle:} $D=0.001$. \textbf{Right:} $D=0.001$ (zoomed in, note change in scale).
    The leading order approximation $\kappa^*=|\tau^*|$ (orange) captures the synchronization boundary for small perturbations $\kappa,\tau$, but differs from the known deterministic synchronization boundary $\kappa^*={|\tau^*|}/{2}$ (blue).
    } 
    \label{fig:tongue_noise}
\end{figure*}


\section{Discussion}\label{section:discussion}

Phase reduction approaches have proven  useful  for  studying  \textit{deterministic} oscillators, where the nonlinear change of coordinates from the original variables to (phase, amplitude) coordinates replaces the original nonlinear dynamics with simplified linear-affine dynamics \cite{perez2020global}.
Phase reduction is particularly effective for characterizing the synchronization behavior of  coupled deterministic oscillators \cite{park2021high}. 
Recently, it was shown that the $Q$-function, the slowest decaying mode of the backward operator (Stochastic Koopman operator, SKO), provides a comparable  simplifying nonlinear change of coordinates  for \textit{stochastic} oscillators \cite{perez2023universal}.
While there exist several previously proposed notions of `phase' and `synchronization' for stochastic oscillators, the literature to date is self-contradictory \cite{amro2015phase, callenbach2002oscillatory,deng2016measuring,ermentrout2009noisy,freund2000analytic,han1999interacting,medvedev2010synchronization,nandi2007effective,neiman1994synchronizationlike,zakharova2011analysing, zakharova2013coherence,zaks2003noise,zhang2008interacting}.
To the best of our knowledge, a mathematically founded  notion of phase for stochastic oscillators such as the mean--return-time phase\cite{schwabedal2013phase,cao2020partial}, phase defined from a variational equation\cite{BressloffMacLaurin2018variational,Maclaurin2023phase,AdamsMacLaurin2025isochronal}, or the stochastic asymptotic phase based on the SKO\cite{thomas2014asymptotic,thomas2019phase,perez2022quantitative,perez2023universal} has not previously been applied to study the synchronization  of systems of coupled stochastic oscillators.

In this manuscript, we fill this gap by extending the $Q$-function phase reduction framework to systems of symmetrically coupled stochastic oscillators.
We introduce a novel definition of synchronization for stochastic oscillators in terms of the eigenvalue spectrum of the SKO.
We demonstrate that qualitative changes in the eigenvalue behavior corresponds to qualitative changes in the power spectra and cross-spectral densities of the coupled system in $Q$-function coordinates.
Our definition of synchronization is conceptually consistent with deterministic synchronization based on phase reduction.
Yet it applies whether the stochastic oscillator arises as a deterministic limit cycle system perturbed by noise, or as a noise-induced oscillation (as in our coupled Ornstein-Uhlenbeck process example). 
In contrast, other notions of synchronization are either ad hoc, for example relying on  arbitrary thresholds for the  mean frequency difference, or else rely classical notions of phase which are not strictly well-defined for stochastic systems. 
Moreover, we show that the $Q$-synchronization boundary as a function of coupling strength and frequency difference closely resembles the boundary of $1:1$ Arnold tongues for deterministic oscillators \cite{arnold2012geometrical, guckenheimer2013nonlinear, izhikevich2007dynamical, schilder2007computing}.

Mauroy and Mezić studied the Koopman eigenvalue spectrum of \textit{deterministic} oscillatory systems both in the case of periodic limit cycle systems and in the case of quasiperiodic oscillatory systems \cite{mauroy2012use,mauroy2018global,mezic2020spectrum}.
In the limit cycle case, they showed that the Koopman operator has a unique purely imaginary minimum frequency eigenvalue $\lambda_1= i\omega$ corresponding to the deterministic limit cycle period $T=\frac{2\pi}{\omega}$ that generates a lattice of higher frequency purely imaginary eigenvalues.
In contrast, for a quasiperiodic solution, they showed that the Koopman operator has multiple purely imaginary eigenvalues with incommensurate frequencies, generating multiple incommensurate lattices of imaginary eigenvalues. 
For a system of two coupled deterministic limit cycle oscillators, within a 1:1 mode-locking region, we may view the system as possessing a single limit cycle.
In this case, standard phase reduction $\mbx\to\theta(\mbx)$ guarantees a low-lying Koopman eigenpair, namely the \(Q\)-function \(\exp[i\theta(\mathbf{x})]\) with eigenvalue \(\lambda_1 = i\omega\).  
Outside a 1:1 mode-locked tongue, a system of coupled deterministic oscillators may exhibit a variety of behaviors, including higher-order $p:q$ mode-locking and quasiperiodic behavior.  
For example, Mauroy and Mezić describe quasiperiodic orbits of two coupled van der Pol oscillators \cite{mauroy2012use}.
These observations suggest that a transition from quasiperiodicity to synchronization in the deterministic case, with increasing coupling strength, may be marked by a qualitative change in the Koopman eigenvalue spectrum.
If true, such a result would be analogous to the behavior seen in the stochastic case discussed here.  
More broadly, a Koopman-based framework appears to offer a universal description for oscillatory phenomena, encompassing both deterministic and stochastic oscillators as well as single and coupled systems.  
A formal investigation into the relationship between Koopman eigenvalues with and without noise, particularly in the context of Arnold tongues, remains an open topic for future study.

The present study limits itself to the consideration of synchronization for stochastic oscillators.
However, the framework put forth here is expected to hold for a variety of situations, some of which have already been considered in the literature.
One study analyzed the synchronization of a deterministic Kuramoto network and found that bifurcations in the leading Koopman eigenvalues identified regions of synchrony where the frequency of oscillation reached a consensus value \cite{hu2020koopman}.
Koopman spectral methods have also been implemented to study the entrainment of stochastic oscillators.
In \cite{doi1998spectral}, the entrainment of a stochastic van der Pol oscillator was studied using the spectrum of the Koopman operator.
Bifurcations in the dominant Koopman eigenvalues were observed, and provided a means to identify regions of $p:q$ mode-locking.
In \cite{tateno2000stochastic}, a similar approach was taken to analyze the entrainment of stochastic integrate-and-fire neuron models, where the observed eigenvalue bifurcations were used to construct regimes similar to Arnold tongues for deterministic oscillators.
The present manuscript generalizes these results to the case of coupled stochastic oscillators.

Our $Q$-synchronization  approach  offers an alternative approach complementing stochastic bifurcation theory. 
Classically, a stochastic bifurcation refers to a topological change in the stationary probability distribution of a given system as parameters are varied \cite{horsthemke1984noise}.
We remark that in the examples considered in this manuscript, no qualitative change was observed in the stationary distribution below or above the bifurcation point of the SKO eigenvalues.
Our eigenvalue approach incorporates long term \textit{dynamical} information, as opposed to static information captured by the stationary distribution, i.e., the $Q$-function approach captures qualitative changes in phase, isochrons, mean frequency, and power spectra which are not well captured with just a stationary analysis.

The framework put forth in this manuscript opens up several further avenues for future investigation.
Here, we limited consideration to systems which consist of only two symmetrically coupled oscillators in the $1:1$ mode-locked regime.
Consideration of systems of $n\geq 3$ coupled oscillators, of systems with non-symmetric coupling, or of systems which admit solutions in the $p:q$ mode-locked regime are natural next steps.

The present work develops a theory for the synchronization of stochastic oscillators assuming that the $Q$-functions of the jointly coupled system are known.
However, in practice, obtaining $Q$-functions for systems in $n\geq 3$ dimensions is a challenging numerical task.
Mesh-free techniques such as extended dynamic mode decomposition (EDMD) \cite{brunton2016koopman,brunton2021modern,colbrook2023mpedmd, mezic2022numerical,korda2018convergence,williams2015data}, radial basis function (RBF) approaches \cite{buhmann2000radial,fasshauer2007meshfree,platte2004computing}, or machine-learning methods \cite{li2019data, zhai2022deep, brunton2024promising} may prove useful here.
Synchronization of stochastic oscillators is of interest also in cases for which one lacks an explicit underlying SDE model, or for which one can only observe one of several coordinates  (e.g., voltage  time-series from EEG recordings). 
Numerical methods currently under development enable one to extract the stochastic asymptotic phase and (or) the $Q$-function itself from time series using data-driven Koopman oscillator techniques\cite{MellandCurtu2023attractor,Houzelstein2024generalized,williams2015data}.
The theory developed here offers a test for synchronization: transforming data to empirically-obtained $Q$-function coordinates, one can then compute the power spectra and cross-spectral density and observe whether or not a given system is within a $Q$-synchronization regime.


\begin{acknowledgments}
This work was supported in part by NSF grant DMS-2052109. 
This work was supported in part by the Oberlin College Department of Mathematics.

\end{acknowledgments}

\section*{Author Declarations}
The authors have no conflicts to disclose.


\section*{Data Availability Statement}

The data that support the findings of this study are openly available in the GitHub repository \url{https://github.com/MaxKreider/Arnold_Tongues_for_Stochastic_Oscillators}.


\appendix

\section{Stationary eigenfunction perturbation}\label{appendix:stationary eigenfunction perturbation}

Here, we give an analytic expression for the leading order correction to the stationary distribution for the 2D ring system
\begin{equation}\label{eq:appendix_ring_model}
    \begin{split}
        x' &= \omega + \tau + \kappa \sin(y-x) + \sqrt{2D}\xi_1(t)
        \\
        y' &= \omega + \kappa \sin(x-y) + \sqrt{2D}\xi_2(t)
    \end{split}
\end{equation}
The forward equation for the stationary distribution is given by
\begin{equation}
    [\mathcal{L}_x + \mathcal{L}_y + \tau\mathcal{L}_\tau + \kappa\mathcal{L}_\kappa]\mathcal{P}(x,y) = 0
\end{equation}
where
\begin{equation}
    \begin{split}
        [\mathcal{L}_x + \mathcal{L}_y]\mathcal{P}(x,y) &= \Big(-\omega\partial_x\mathcal{P}(x,y) -\omega\partial_y \mathcal{P}(x,y) 
        \\
        &\qquad+ D\partial_{xx}\mathcal{P}(x,y) + D\partial_{yy}\mathcal{P}(x,y)\Big)
        \\
        \tau\mathcal{L}_\tau \mathcal{P} &= -\tau \partial_x\mathcal{P}(x,y)
        \\
        \kappa\mathcal{L}_\kappa\mathcal{P} &= -\kappa\Big(\partial_x[\sin(y-x)\mathcal{P}(x,y)]
        \\
        &\qquad+\partial_y[\sin(x-y)\mathcal{P}(x,y)]\Big)
    \end{split}
\end{equation}
The forward equation is augmented with $2\pi$-periodic boundary conditions: $\mathcal{P}(0,0) = \mathcal{P}(0,2\pi) = \mathcal{P}(2\pi,0) = \mathcal{P}(2\pi, 2\pi)$.
By Lemma \ref{lemma:stationary distribution perturbation}, the leading order correction for the stationary distribution, $\mathcal{P}_c(x,y)$, satisfies
\begin{equation}\label{eq:appendix_ring_equation}
     (\mathcal{L}_x + \mathcal{L}_y) \mathcal{P}_c(\mbx,\mby) = -(\tau\mathcal{L}_\tau + \kappa\mathcal{L}_\kappa)[P_0(\mbx) P_0(\mby)]
\end{equation}
where the unperturbed stationary distribution is given by $P_0(x)P_0(y)=\frac{1}{4\pi^2}$.
Writing out \eqref{eq:appendix_ring_equation} explicitly gives
\begin{equation}
\begin{split}
    &\left[-\omega\partial_x -\omega\partial_y  + D\partial_{xx} + D\partial_{yy}(x,y)\right]\mathcal{P}_c(x,y) = 
    \\
    &\qquad-\frac{\kappa}{2\pi^2}\cos(y-x)
\end{split}
\end{equation}
One can verify that
\begin{equation}
    \mathcal{P}_c(x,y) = \kappa\frac{\cos(y-x)}{4D\pi^2}
\end{equation}
satisfies \eqref{eq:appendix_ring_equation} with periodic boundary conditions.
Consequently, to leading order in $\kappa$ and $\tau$, the stationary distribution of \eqref{eq:appendix_ring_model} is given by
\begin{equation}
    \mathcal{P}_0(x,y) = \frac{1}{4\pi^2}\left(1 + \kappa\frac{\cos(y-x)}{D}\right) + o(\kappa) + o(\tau)
\end{equation}

\section{Coupled system without a bifurcation}\label{appendix:no bifurcation}

Here, we consider an example for which the there do not exist real-valued perturbations of the form $\kappa^*=\mathfrak{K}(\tau^*)$ which satisfy $\mathcal{D}(\mathfrak{K}(\tau^*),\tau^*) = 0$.
We consider the system
\begin{equation}\label{eq:ring_appendix_model}
    \begin{split}
        x' &= \omega + \tau + \kappa \cos(y-x) + \sqrt{2D}\xi_1(t)
        \\
        y' &= \omega + \kappa \cos(x-y) + \sqrt{2D}\xi_2(t)
    \end{split}
\end{equation}
with periodic boundary conditions on $[0,2\pi)\times[0,2\pi)$.
We solve the related SKO eigenvalue problem with the method of continued fractions (see Appendix \ref{appendix: continued fractions}).

The unperturbed $(\kappa=\tau=0$) eigenvalue-eigenfunction pairs of \eqref{eq:ring_appendix_model} are given by \eqref{eq:ring_unperturbed_efunctions}. 
Directly computing the matrix $\mathcal{M}$ (see equation \eqref{eq:M matrix}) for the repeated eigenvalue $\lambda_1=-D+i\omega$ yields
\begin{equation}
    \mathcal{M} = \begin{bmatrix}
        i\tau & \frac{\kappa}{2}i
        \\
        \frac{\kappa}{2}i & 0
    \end{bmatrix}
\end{equation}
The splitting determinant is see to be
\begin{equation}
    \mathcal{D}(\kappa,\tau) = \sqrt{-\tau^2 - \kappa^2}
\end{equation}
By Corollary \ref{corollary:splitting discriminant}, nonzero perturbations $\kappa^*,\tau^*$ that satisfy $\mathcal{D}(\kappa^*,\tau^*)=0$ satisfy the relation
\begin{equation}
    \kappa^* = \mathfrak{K}(\tau^*) = \pm i |\tau^*| \frac{\sqrt{(i-0)^2}}{i} = \pm i |\tau^*| 
\end{equation}
Consequently, there are no non-trivial real-valued perturbations $\kappa^*,\tau^*$ such that $\mathcal{D}(\kappa^*,\tau^*)=0$.
The system \eqref{eq:ring_appendix_model} does not exhibit $Q$-synchronization for real-valued perturbations.
Figure \ref{fig:appen_eval} shows the SKO eigenvalues of \eqref{eq:ring_appendix_model} when $\tau\neq 0$.

\begin{figure}[ht]
    \centering
    \includegraphics[scale=.45]{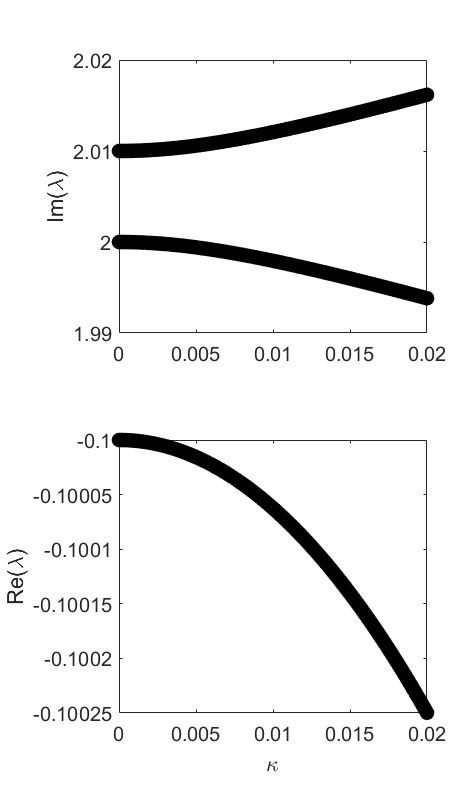}
    \caption{The imaginary parts (top) and real parts (bottom) of the $Q$-function eigenvalues associated with \eqref{eq:ring_appendix_model}. No bifurcation occurs for real-valued perturbations $\kappa$ and $\tau$, which is predicted by our analysis.}
    \label{fig:appen_eval}
\end{figure}

\section{Gauge transformation}\label{appendix:gauge_transformation}

Here, we consider the rotational ambiguity inherent in the $Q$-functions of two coupled stochastic oscillators, and derive a gauge transformation to eliminate this ambiguity.

Consider a function $Q_1^*(\mbx)$ for $\mbx \in \mathbb{R}^n$.
Let $Q_2^*(\mbx)=Q_1^*(\mbx)e^{i\xi}$, for some unknown $\xi\in [0,2\pi)$.
Define an objective function
\begin{equation}
    f(\alpha) = \langle |Q_1^*(\mbx)-Q_2^*(\mbx)e^{i\alpha}|^2\rangle
\end{equation}
where
\begin{equation}
    \langle Q^*(\mbx) \rangle = \int_{\mathbb{R}^n} d\mbx \; Q^*(\mbx)P_0(\mbx)
\end{equation}
and $P_0(\mbx)$ is the corresponding stationary distribution of the system in question.
Our goal is to minimize $f$ with respect to $\alpha$.
First, note that
\begin{equation}
    \begin{split}
        f(\alpha) &= \int_{\mathbb{R}^n} d\mbx \; \Big( Q_1^*(\mbx)Q_1(\mbx) - Q_1^*(\mbx)Q_2(\mbx) e^{-i\alpha} 
        \\
        &\quad- Q_2^*(\mbx)Q_1(\mbx)e^{i\alpha} + Q_2^*(\mbx)Q_2(\mbx)\Big)P_0(\mbx)
        \\
        &= 2 - \int_{\mathbb{R}^n} d\mbx \; Q_1^*(\mbx)Q_2(\mbx)e^{-i\alpha}P_0(\mbx) 
        \\
        &\quad- \int_{\mathbb{R}^n} d\mbx \; Q_2^*(\mbx)Q_1(\mbx)e^{i\alpha}P_0(\mbx)
    \end{split} 
\end{equation}
Simplifying $\frac{d}{d\alpha} f(\alpha) = 0$ yields the expression
\begin{equation}
    \alpha = \frac{1}{2}\log \frac{\int_{\mathbb{R}^n} d\mbx \; Q_1^*(\mbx)Q_2(\mbx) P_0(\mbx)}{\int_{\mathbb{R}^n} d\mbx \; Q_2^*(\mbx)Q_1(\mbx) P_0(\mbx)} + n\pi,\quad n\in \mathbb{Z}
\end{equation}
The integer $n$ arises from consideration of the infinitely-many branches of a complex logarithm.
If $n=0$, then
\begin{equation}
    f(\alpha_{n=0}) = 2 - 2 \left| \int_{\mathbb{R}^n} d\mbx \; Q_1(\mbx)Q_2^*(\mbx)P_0(\mbx)  \right|
\end{equation}
and if $n=\pm 1$, then
\begin{equation}
    f(\alpha_{n=\pm 1}) = 2 + 2 \left| \int_{\mathbb{R}^n} d\mbx \; Q_1(\mbx)Q_2^*(\mbx)P_0(\mbx)  \right|
\end{equation}
Consequently, the principal branch of the logarithm solves the minimization problem.

\section{Computation of the matrix $M$}\label{appendix:M compuation}

Here, we provide details for the computation of the matrix $\mathcal{M}$ for a specific example.
We consider the linear system
\begin{equation}\label{eq:appendix_linear_system}
    \begin{split}
        x_1' &= -\eta x_1 +(\omega + \tau) x_2 + \kappa (y_1-x_1) + \sqrt{2D}\xi_1(t)
        \\
        x_2' &= -(\omega + \tau) x_1 - \eta x_2 + \kappa (y_2-x_2) + \sqrt{2D}\xi_2(t)
        \\
        y_1' &= -\eta y_1 +\omega y_2 + \kappa (x_1-y_1) + \sqrt{2D}\xi_3(t)
        \\
        y_2' &= -\omega y_1 - \eta y_2 + \kappa (x_2-y_2) + \sqrt{2D}\xi_4(t)
    \end{split}
\end{equation}
Let $\mbx=[x_1,x_2]^T$ and let $\mby=[y_1,y_2]^T$.
The unperturbed low-lying eigenvalues and eigenfunctions of the SKO and Fokker-Planck operator associated with \eqref{eq:appendix_linear_system} are given by
\begin{equation}
    \begin{split}
        \lambda_1 &= -\eta + i\omega
        \\
        P_0(\mbx,\mby) &= \frac{1}{4\pi^2}\frac{\eta^2}{D^2}\exp\left(-\frac{\eta}{2D}(x_1^2+x_2^2+x_3^2+x_4^2)\right)
        \\
        Q^*_{1x}(\mbx,\mby) &= (x_2+ix_1)\sqrt{\frac{\eta}{2D}}
        \\
        Q^*_{1y}(\mbx,\mby) &= (y_2+iy_1)\sqrt{\frac{\eta}{2D}}
        \\
        P_{1x}(\mbx,\mby) &= \frac{1}{4\sqrt{2}\pi^2}\left(\frac{\eta}{D}\right)^{5/2}\;\times
        \\
        &\quad(x_2-ix_1)\exp\left(-\frac{\eta}{2D}(x_1^2+x_2^2+x_3^2+x_4^2)\right)
        \\
        P_{1y}(\mbx,\mby) &= \frac{1}{4\sqrt{2}\pi^2}\left(\frac{\eta}{D}\right)^{5/2}\;\times
        \\
        &\quad(y_2-iy_1)\exp\left(-\frac{\eta}{2D}(x_1^2+x_2^2+x_3^2+x_4^2)\right)
    \end{split}
\end{equation}
Now, we compute the matrix $\mathcal{M}$ (recall equation \eqref{eq:M matrix}).
Note that
\begin{equation}
\begin{split}
    \tau \mathcal{L}_\tau^\dagger &= \tau\left[x_2 \partial_{x_1} - x_1 \partial_{x_2} \right]
    \\
    \kappa\mathcal{L}_\kappa^\dagger &= \kappa[(y_1-x_1)\partial_{x_1} + (y_2-x_2)\partial_{x_2} 
    \\
    &\quad+ (x_1-y_1)\partial_{y_1} + (x_2-y_2)\partial_{y_2}]
\end{split}
\end{equation}
Consequently, one has that
\begin{equation}\label{eq:ou_appendix_stuff}
    \begin{split}
         \tau \mathcal{L}_\tau^\dagger [Q_{1x}^*(\mbx,\mby)] &= \tau (-x_1 + ix_2)\sqrt{\frac{\eta}{2D}}
         \\
         \tau \mathcal{L}_\tau^\dagger [Q_{1y}^*(\mbx,\mby)] &= 0
         \\
         \kappa\mathcal{L}_{\kappa}^\dagger [Q_{1x}^*(\mbx,\mby)] &= \kappa[(y_1-x_1)i + (y_2-x_2)]\sqrt{\frac{\eta}{2D}}
         \\
         \kappa\mathcal{L}_{\kappa}^\dagger [Q_{1y}^*(\mbx,\mby)] &= \kappa[(x_1-y_1)i + (x_2-y_2)]\sqrt{\frac{\eta}{2D}}
    \end{split}
\end{equation}
Directly evaluating the integrals in \eqref{eq:M matrix} reveals that
\begin{equation}
    \mathcal{M}(\kappa,\tau) = \begin{bmatrix}
        i\tau - \kappa & \kappa
        \\
        \kappa & -\kappa
    \end{bmatrix}
\end{equation}
as desired.

\section{Continued fractions}\label{appendix: continued fractions}

Here, we sketch the implementation of a continued fraction (CF) approach (following the general procedure described in \cite{risken1996fokker}) to solve the 2D ring model, which we recall is given by
\begin{equation}\label{eq: appendix 2D Kuramoto model system}
    \begin{split}
        x' &= \omega + \tau + \kappa \sin(y-x) + \sqrt{2D}\xi_1(t)
        \\
        y' &= \omega + \kappa \sin(x-y) + \sqrt{2D}\xi_2(t)
    \end{split}
\end{equation}
The backward equation, expressed in the form $\mathcal{L}^\dagger[Q^*]-\lambda Q^* = 0$, corresponding to \eqref{eq: appendix 2D Kuramoto model system} is 
\begin{equation}\label{eq: appendix 2D ring backwards operator}
\begin{split}
    &[\omega+\tau + \kappa\sin(y-x)]\partial_x Q^* + [\omega + \kappa\sin(x-y)] \partial_y Q^* 
    \\
    &\quad+ D \partial_{xx} Q^* + D \partial_{yy} Q^* - \lambda Q^* = 0
\end{split}
\end{equation}
along with $2\pi$-periodic boundary conditions
\begin{equation}\label{eq; appendix 2D ring boundary conditions}
    Q^*(0,0) = Q^*(0,2\pi)=Q^*(2\pi,0) = Q^*(2\pi,2\pi)
\end{equation}

To begin, we adopt a double Fourier ansatz
\begin{equation}\label{eq:ansatz}
    Q^* =\sum_{k=-\infty}^\infty\sum_{j=-\infty}^\infty c_{j,k}e^{i(jx+ky)}
\end{equation}
Insertion of \eqref{eq:ansatz} into \eqref{eq: appendix 2D ring backwards operator} and rearranging gives
\begin{equation}
\begin{split}
    \sum_{k=-\infty}^\infty&\sum_{j=-\infty}^\infty e^{i(jx+ky)} \Bigg(  c_{j-1,k+1}\Bigg[\frac{\kappa}{2}(k-j+2)\Bigg] 
    \\
    &+ c_{j,k}\Bigg[-D(j^2+k^2)+i((\omega+\tau)j+\omega k)-\lambda \Bigg] 
    \\
    &+ c_{j+1,k-1}\Bigg[\frac{\kappa}{2}(j-k+2)\Bigg] \Bigg)  = 0
\end{split}
\end{equation}
By orthogonality, one finds that
\begin{equation}
\begin{split}
     &c_{j-1,k+1}\Bigg[\frac{\kappa}{2}(k-j+2)\Bigg] 
     \\
     &\quad+ c_{j,k}\Bigg[-D(j^2+k^2)+i((\omega+\tau)j+\omega k)-\lambda \Bigg] 
     \\
    &\quad+ c_{j+1,k-1}\Bigg[\frac{\kappa}{2}(j-k+2)\Bigg] = 0
\end{split}
\end{equation}
for each $j,k$.
Note that only those coefficients $c_{j,k}$ with $j+k=\text{const} \equiv N$ are coupled.
Therefore, we define
\begin{equation}
    z^N_j = c_{j,N-j}
\end{equation}
and in this notation the recurrence relation becomes
\begin{equation}\label{eq:tridiagonal}
\begin{split}
    &z^N_{j-1}\Bigg[\frac{\kappa}{2}(N-2(j-1)) \Bigg] 
    \\
    &\quad+ z^N_j\Bigg[-D(j^2+(N-j)^2)+i\left((\omega+\tau)j+\omega(N-j)\right) - \lambda \Bigg] 
    \\
    &\quad+ z^N_{j+1}\Bigg[\frac{\kappa}{2}(2(j+1)-N) \Bigg] = 0
\end{split}
\end{equation}
where $N\in (-\infty, \infty)$ is an integer.
For fixed $N$, \eqref{eq:tridiagonal} is a tridiagonal recurrence relation  corresponding to an infinite dimensional linear system.
Let $J\in \mathbb{Z}^+$ be a cutoff index.
Demanding that $j\in [-J,J]$ results in a truncated, finite-dimensional linear system.
Diagonalizing the truncated system gives an approximation of the eigenvalues of \eqref{eq: appendix 2D ring backwards operator}.
Note that the periodic boundary conditions are automatically satisfied by the choice of basis functions \eqref{eq:ansatz}.
The eigenvectors of the truncated linear system give the coefficients $z_j^N$, which may be translated to coefficients $c_{j,k}$.
Substitution of the computed $c_{j,k}$ into \eqref{eq:ansatz} yields an approximation of the desired eigenfunction.

Alternatively, one may take advantage of the special tridiagonal structure of the truncated linear system to efficiently compute the eigenfunctions using continued fractions.
To compute the eigenfunction corresponding to an (approximated) eigenvalue $\lambda$, rewrite \eqref{eq:tridiagonal}, for fixed $N$, in the general form
    \begin{equation}\label{eq:tri_gen}
    Q_n^- c_{n-1} + Q_n c_n + Q_n^+c_{n+1} = 0
\end{equation}
where the scalar coefficients $Q_n,$ $Q_n^+$, and $Q_n^1$ may be read off directly from the recurrence relation.\footnote{Here we follow the notation of \cite{risken1996fokker}; these coefficients should not be confused with the $Q$-functions themselves.}
Define the ratio $S_n = c_{n+1}/c_n$.
Then, \eqref{eq:tri_gen} may be expressed in terms of $S_n$
\begin{equation}
    S_n = -\frac{Q_{n+1}^-}{Q_{n+1} + Q_{n+1}^+ S_{n+1}}
\end{equation}
which, in turn, may be expressed as an infinite continued fraction
\begin{equation}\label{eq:cf}
    S_n = 
    -\cfrac{Q_{n+1}^-}{Q_{n+1} -  \cfrac{Q_{n+1}^+Q_{n+2}^-}{Q_{n+2} - \frac{Q_{n+2}^+ Q^-_{n+1}}{Q_{n+3}-\dots} }}
\end{equation}
One has $c_n = S_{n-1}S_{n-2}\dots S_0c_0 $; computation of the ratios $S_j$ provides the means to compute the coefficients $c_n$.
Note that one may take $c_0$ as any non-zero constant (a constant multiple of an eigenfunction is still an eigenfunction).
Note that \eqref{eq:cf} is an infinite continued fraction.
For practical applications, one seeks the $M$th approximant of such expressions (an approximation upon truncating the infinite fraction after $M$ terms). 
To solve for the $M$th approximant of a continued fraction
\begin{equation}
    K_M\approx b_0+\cfrac{a_1}{b_1+\cfrac{a_2}{b_2+\dots}}
\end{equation}
one solves
\begin{equation}
    \begin{split}
   A_n&=b_nA_{n-1}+a_nA_{n-2}
   \\
   B_n&=b_nB_{n-1}+a_nB_{n-2}
\end{split}
\end{equation}
with initial conditions $A_{-1}=1, \; A_0=b_0,  \;B_{-1}=0, \; B_0=1$.
The $M$th approximant is given by
\begin{equation}
    K_M = \frac{A_M}{B_M}
\end{equation}

In this manuscript, we compute the eigenfunctions of the 2D ring model \eqref{eq: appendix 2D Kuramoto model system} using continued fractions.
To solve for the $Q$ functions, we set $N=1$ (because we are interested in the leading eigenpair), $J=250$ (which sets the size of the linear system), and $M=50$ (which sets the truncation point for the continued fraction approximation).

We remark that the CF approach closely connects with Koopman theory \cite{brunton2016koopman,brunton2021modern,colbrook2023mpedmd, mezic2022numerical, williams2015data}. 
The double Fourier ansatz \eqref{eq:ansatz} can be viewed as a `dictionary' of functions, under which the finite dimensional nonlinear system \eqref{eq: appendix 2D Kuramoto model system} is represented as an infinite dimensional linear system \eqref{eq:tridiagonal}.
Truncation of the infinite dimensional linear system gives a finite dimensional linear approximation of the original nonlinear dynamics.

\section{Discrete system bifurcation}\label{appendix:discrete system bifurcation}

Here, we analyze the synchronization behavior of the 9D discrete system.
Above the KT point, the imaginary parts of the two dominant eigenvalues are distinct.
Nevertheless, we find that the phases of the coupled oscillators are locked only above the KT point, and thus the bifurcation corresponds intuitively to the notion of synchronization.
To see this, we define the projection matrices
\begin{equation}\label{eq:projection_matrices}
    \begin{split}
        P_A &= \begin{bmatrix}
            I_{3\times 3} & I_{3\times 3} & I_{3\times 3}
        \end{bmatrix}, \quad P_B = I_{3\times 3} \otimes \begin{bmatrix}
            1 & 1 & 1
        \end{bmatrix}
    \end{split}
\end{equation}
where $\otimes$ denotes the Kronecker product.
The two 9D eigenvectors of the jointly coupled system \eqref{eq:discrete_coupled_system} corresponding to the two dominant backward eigenmodes are projected via their respective projection matrix \eqref{eq:projection_matrices} to two 3D eigenvectors corresponding to each individual oscillator.
We plot the complex argument of these vectors as a function of coupling strength in Figure \ref{fig:phase_dif}, and find that the stochastic asymptotic phase of each oscillator becomes locked after the bifurcation, i.e., while the two oscillators do not have identical frequencies in $Q$-function coordinates, the structure of the isochrons of each oscillator are the same above the bifurcation.

\begin{figure}[ht]
    \centering
    \includegraphics[scale=.4]{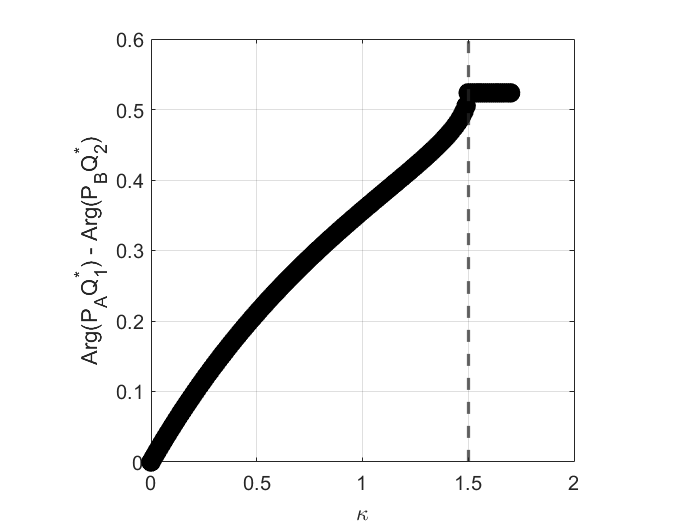}
    \caption{Difference of the complex argument of the projected $Q$-functions of the 9D discrete-state system.
    The difference becomes constant after the KT point (dotted line).}
    \label{fig:phase_dif}
\end{figure}

\section{Linear system eigenfunctions}\label{appendix: linear system eigenfunctions}

Here, we prove a lemma giving the form of the low-lying backward eigenfunctions of a general linear system of the form
\begin{equation}\label{eq: appendix linear system}
    \mathbf{X}' = \mathcal{A}\mathbf{X} + \mathcal{D}\vec{\xi}(t)
\end{equation}
where $\mbX \in \mathbb{R^N}$, $\mathcal{A}$ is an $N\times N$ drift matrix, $\mathcal{D}$ is a $N \times M$ noise matrix, and $\vec{\xi}$ is an $M \times 1$ white noise vector.
The backward operator for \eqref{eq: appendix linear system} is given by 
\begin{equation}\label{eq: appendix linear backwards operator}
    \mathscr{L}^\dagger [Q^*] = \sum_{i} (\mathcal{A}\mbx)_i [\partial_{x_i} Q^*] + \frac{1}{2}\sum_{ij} \mathcal{B}_{ij} [\partial_{x_i x_j} Q^*]
\end{equation}
where $\mathcal{B} = \mathcal{D}\mathcal{D}^T$.
The spectral properties of the backward operator \eqref{eq: appendix linear backwards operator} have been thoroughly studied (see for instance \cite{leen2016eigenfunctions, metafune2002spectrum, zhang2021computing}).
Here, we derive the form of the first non-trivial eigenfunctions of \eqref{eq: appendix linear backwards operator}.
In the case when \eqref{eq: appendix linear system} corresponds to a system of robustly oscillatory oscillators, our computation provides a simple expression for the corresponding $Q$ functions.

\begin{lemma}
    Assume that $\mathcal{A}$ is diagonalizable with a complete set of left and right eigenvectors.
Then, the first non-trivial eigenfunctions of \eqref{eq: appendix linear backwards operator} are given by
\begin{equation}\label{eq: appendix Q function ansatz}
    Q_i^*(\mbx) = \mathbf{w}_i^*\mbx
\end{equation}
with eigenvalue $\lambda_i$, where $\mathbf{w}_i^* \mathcal{A} = \lambda_i \mathbf{w}_i^*$, i.e., $\mathbf{w}_i^*$ is a left eigenvector of $\mathcal{A}$ with eigenvalue $\lambda_i$.
\end{lemma}

\begin{proof}
    Let $\mathbf{w}_i^* = [w^{(1)}, \; w^{(2)}, \; \dots \; w^{(N)}]$ denote a left eigenvector of $\mathcal{A}$ with eigenvalue $\lambda_i$.
    Substituting the ansatz \eqref{eq: appendix Q function ansatz} into \eqref{eq: appendix linear backwards operator} gives
    \begin{equation}
        \begin{split}
            \mathscr{L}^\dagger[Q^*] &= 
            \sum_{i} (\mathcal{A}\mbx)_i [\partial_{x_i} Q^*] + \frac{1}{2}\sum_{ij} \mathcal{B}_{ij} [\partial_{x_i x_j} Q^*]
            \\
            &= 
            (\mathcal{A}\mbX)_1 [\partial_{x_1} \mathbf{w}_i^*\mathbf{x}] + (\mathcal{A}\mbX)_2 [\partial_{x_2} \mathbf{w}_i^*\mathbf{x}] + 
            \\
            &\quad\dots + (\mathcal{A}\mbX)_{N} [\partial_{x_{N}} \mathbf{w}_i^*\mathbf{x}]
            \\
            &= \begin{bmatrix}
                w^{(1)} & w^{(2)} & \dots & w^{(N)}
            \end{bmatrix}\begin{bmatrix}
                (\mathcal{A}\mbx)_1
                \\
                (\mathcal{A}\mbx)_2
                \\
                \vdots
                \\
                (\mathcal{A}\mbx)_{N}
            \end{bmatrix}
            \\
            &= \mathbf{w}_i^*\mathcal{A}\mbx
            \\
            &= \lambda_i \mathbf{w}_i^* \mbx
            \\
            &= \lambda_i Q^*
        \end{split}
    \end{equation}
\end{proof}

\bibliographystyle{plain}
%

\end{document}